\documentclass[11pt]{article}

\usepackage{amssymb}
\usepackage{amsmath}
\usepackage{amsfonts}
\usepackage{amsthm}
\usepackage{stmaryrd}
\usepackage[all]{xy}
\usepackage{mathrsfs}
\usepackage{hyperref}
\usepackage[pdftex]{color}

\numberwithin{equation}{subsection}
\newtheorem{theorem}[equation]{Theorem}

\newtheorem{corollary}[equation]{Corollary}
\newtheorem{lemma}[equation]{Lemma}
\newtheorem{proposition}[equation]{Proposition}

\theoremstyle{definition}

\newtheorem{construction}[equation]{Construction}
\newtheorem{convention}[equation]{Convention}

\newtheorem{definition}[equation]{Definition}

\newtheorem{example}[equation]{Example}

\newtheorem{fake-assumption}[equation]{Fake-assumption}
\newtheorem{fake-definition}[equation]{Fake-definition}
\newtheorem{fake-proposition}[equation]{Fake-proposition}
\newtheorem{fake-theorem}[equation]{Fake-equation}
\newtheorem{notation}[equation]{Notation}

\newtheorem{remark}[equation]{Remark}
\newtheorem{situation}[equation]{Situation}
\newtheorem{hypothesis}[equation]{Hypothesis}

\def\ee{\mathrm{e}}

\def\bbb{\mathbf{b}}

\def\bbe{\mathbf{e}}
\def\bbf{\mathbf{f}}

\def\bbi{\mathbf{i}}
\def\bbj{\mathbf{j}}

\def\bbv{\mathbf{v}}
\def\bbw{\mathbf{w}}

\def\bbz{\mathbf{z}}

\def\bbF{\mathbf{F}}

\def\bbP{\mathbf{P}}
\def\bbQ{\mathbf{Q}}
\def\bbR{\mathbf{R}}

\def\bbX{\mathbf{X}}
\def\bbY{\mathbf{Y}}

\def\AAA{\mathbb{A}}

\def\FF{\mathbb{F}}

\def\NN{\mathbb{N}}

\def\QQ{\mathbb{Q}}
\def\RR{\mathbb{R}}

\def\TT{\mathbb{T}}

\def\ZZ{\mathbb{Z}}

\def\calD{\mathcal{D}}
\def\calE{\mathcal{E}}
\def\calF{\mathcal{F}}

\def\calI{\mathcal{I}}

\def\calL{\mathcal{L}}

\def\calO{\mathcal{O}}

\def\calR{\mathcal{R}}

\def\gothm{\mathfrak{m}}

\def\goths{\mathfrak{s}}

\def\gothE{\mathfrak{E}}
\def\gothF{\mathfrak{F}}

\def\gothI{\mathfrak{I}}

\def\gothR{\mathfrak{R}}

\newcommand{\bs}{\backslash}

\newcommand{\exact}[5]{#1\rightarrow#2\rightarrow#3\rightarrow#4\rightarrow#5}

\newcommand{\serie}[2]{{#11},\dots,{#1{#2}}}
\newcommand{\seriezero}[2]{{#10}, \dots, {#1{#2}}}

\newcommand{\rar}{\rightarrow}

\newcommand{\lrar}{\longrightarrow}

\newcommand{\Rar}{\Rightarrow}

\newcommand{\LRar}{\Leftrightarrow}
\newcommand{\inj}{\hookrightarrow}
\newcommand{\surj}{\twoheadrightarrow}
\newcommand{\isom}{\stackrel \sim \rar}

\newcommand{\GL}{\mathrm{GL}}

\DeclareMathOperator{\Sp}{Sp}

\newcommand{\Diag}{\mathrm{Diag}}
\newcommand{\dual}{\vee}
\newcommand{\End}{\mathrm{End}}
\newcommand{\Ext}{\mathrm{Ext}}

\newcommand{\Fil}{\mathrm{Fil}}

\newcommand{\Hom}{\mathrm{Hom}}

\newcommand{\Ker}{\mathrm{Ker}\,}
\newcommand{\id}{\mathrm{id}}
\newcommand{\Mat}{\mathrm{Mat}}
\renewcommand{\mod}{\mathrm{\;mod\;}}

\DeclareMathOperator{\rank}{rank}

\renewcommand{\log}{\mathrm{log}}

\newcommand{\alg}{\mathrm{alg}}
\newcommand{\Art}{\mathrm{Art}}

\newcommand{\Char}{\mathrm{char}\,}

\DeclareMathOperator{\Frac}{Frac}

\newcommand{\Gal}{\mathrm{Gal}}

\newcommand{\Ind}{\mathrm{Ind}}

\newcommand{\sep}{\mathrm{sep}}
\newcommand{\Swan}{\mathrm{Swan}}

\newcommand{\unr}{\mathrm{unr}}

\DeclareMathOperator{\Spec}{Spec}

\newcommand{\bd}{\mathrm{bd}}

\renewcommand{\sp}{\mathrm{sp}}

\DeclareMathOperator{\Spf}{Spf}

\newcommand{\Fp}{\FF_p}
\newcommand{\Fq}{{\FF_q}}

\newcommand{\OK}{{\calO_K}}
\newcommand{\Qp}{\QQ_p}

\newcommand{\Zp}{\ZZ_p}

\newcommand{\bbpi}{\boldsymbol{\pi}}

\newcommand{\FE}{\mathrm{FE}}

\newcommand{\inte}{\mathrm{int}}

\newcommand{\multi}{\mathrm{multi}}
\newcommand{\rsw}{\mathrm{rsw}}

\newcommand{\nlog}{\mathrm{nlog}}

\begin{document}

\title{On the refined ramification filtrations in the equal characteristic case}

\author{Liang Xiao \\ Department of Mathematics, University of Chicago \\ 5734 S. University Ave, \\ Chicago, IL 60637 \\
\texttt{lxiao@math.uchicago.edu}
}

\date{December 19, 2011}

\maketitle

\begin{abstract}
Let $k$ be a complete discrete valuation field of equal characteristic $p>0$.  Using the tools of $p$-adic differential modules, we define
refined Artin and Swan conductors for a representation of the
absolute Galois group $G_k$ with finite local monodromy; this leads to a description of the subquotients of the ramification filtration on $G_k$.  We prove that our definition of the refined Swan conductors coincide with that is given by Saito, which uses \'etale cohomology.  We also study its relation with the toroidal variation of the Swan conductors.
\end{abstract}

\tableofcontents

\section*{Introduction}

The ramification theory for a complete discrete valuation field $k$ with possibly imperfect residue field $\kappa_k$ was first studied
by K. Kato \cite{kato}; he used \'etale cohomology and Milnor $K$-theory to give a detailed description of the ramification of a character of the absolute Galois group $G_k$, or rather its maximal abelian quotient $G_k^\mathrm{ab}$.  A. Abbes and T. Saito \cite{abbes-saito1, abbes-saito2} extended Kato's work by providing $G_k$ with the \emph{ramification filtrations} $\Fil^aG_k$ and the
\emph{log ramification filtrations} $\Fil_\log^aG_k$ satisfying certain properties.  Saito \cite{saito-wild-ram} later defined a natural injective homomorphism
\[
\rsw: \Hom(\Fil^a_\log G_k / \Fil_\log^{a+} G_k, \Fp) \rar \Omega_{\calO_k}^1 (\log) \otimes_{\calO_k} \pi_k^{-a} \kappa_k
\]
for each $a \in \QQ_{>0}$, where $\calO_k$ is the ring of integers of $k$, $\pi_k$ is a uniformizer, $\kappa_k$ is the residue field, and $\Omega^1_k(\log)$ is the module of logarithmic differentials; he called it the \emph{refined Swan conductor homomorphism}.  This provides some further information about the subquotients for the log ramification filtration on $G_k$.

Along a different path, 
G. Christol, B. Dwork, S. Matsuda, Z. Mebhkout, and their collaborators used $p$-adic differential modules to give an interpretation of the Swan conductors of representations of $G_k$ when the residue field $\kappa_k$ is perfect.  They associated a $p$-adic differential module over an annulus to any continuous representation of $G_k$, and proved that the Swan conductor of the representation is related to the radii of convergence of the local solutions for the differential module.
K. Kedlaya \cite{kedlaya-swan1} generalized this approach to include the case
in which the residue field is imperfect, by giving the definitions of Artin conductors and Swan conductors for a representation of $G_k$.  The author \cite{xiao1} verified that this pair of definitions coincide with those naturally associated to the ramification filtrations and log ramification filtrations of Abbes and Saito \cite{abbes-saito1, abbes-saito2}.  An important consequence of this comparison result is the \emph{Hasse-Arf theorem} for the ramification filtrations and the log ones \cite[Theorem~4.4.1]{xiao1}, which states that the Artin conductors and Swan conductors are all integers.

In this paper, we give an alternative definition of the refined Swan conductors as well as their nonlog counterparts, using $p$-adic differential modules, and we will compare our definition withe that of Saito.
Let us describe the basic idea of the definition.  In this introduction, we assume for simplicity that $\kappa_k$ has a finite $p$-basis $\{\bar b_1, \dots, \bar b_m\}$.
Let $K$ be the fraction field of the Cohen ring of $\kappa_k$ with respect to $\bar b_1, \dots, \bar b_m$. Let $B_1, \dots, B_m$ denote the canonical lifts of $\bar b_1,
\dots, \bar b_m$ to $K$, respectively.  Let $A^1_K(\eta_0, 1)$ be the annulus over $K$ with coordinate $T$ and  with radii in $(\eta_0, 1)$ for some $\eta_0 \in (0, 1)$.  By aforementioned series of work, one can associate to an irreducible $p$-adic representation $\rho$ of $G_k$ with finite image a differential module $\calE$ over $A^1_K(\eta_0, 1)$ for the differential operators $\partial_0 = \partial / \partial T$ and $\partial_1 = \partial/\partial B_1, \dots, \partial/\partial B_m$.  Let $\bbpi = {-p}^{1/(p-1)}$ denote a \emph{Dwork pi} and put $K' = K(\bbpi)$. When $\rho$ is of pure ramification break $b$, i.e., 
 when $\rho(\Fil^{b+}G_k)$ is trivial but $\rho(\Fil^bG_k)$ is not, 
the following na\"ive picture is helpful as a guide to intuition: there exists a basis of $\calE \otimes_K K'$, with respect to
 which, $\partial_0, \partial_1, \dots, \partial_m$ act as per the prescription:
\begin{equation}
\label{E:differential operators}
\partial_0 = \bbpi T^{-b-1}N_0, \quad \partial_1 = \bbpi T^{-b}N_1, \quad \dots\quad, \partial_m = \bbpi T^{-b} N_m,
\end{equation}
where $N_0, \dots, N_m$ are matrices in $\calO_{K'}\llbracket T\rrbracket$. For each $j \in \{0, \dots, m\}$, we use $\overline N_j$ to denote reduction of
$ N_j$ modulo the ideal $(\bbpi, T)$; these matrices commute and have coefficients in $\kappa_k$.  Take a common (generalized) eigen-basis $e_1, \dots, e_d$ for all $\overline N_j$; set $\theta_{i, j}$ to be
the (generalized) eigenvalue of $\overline N_j$ associated to $e_i$,
viewed as an element in $\kappa_k^\alg$. One might na\"ively tend to define the multiset of \emph{refined Swan conductors} of $\rho$ to be $\{\pi_k^{-b}\big(\theta_{i, 0} \frac {d\pi_k}{\pi_k} + \theta_{i, 1}d\bar b_1 + \cdots + \theta_{i, m}d\bar b_m\big)\,|\,i = 1, \dots, d\} \subset \Omega_{\calO_k}^1(\log) \otimes_{\calO_k} \pi_k^{-b}\kappa_k^\alg$.
Of course, such a basis of $\calE \otimes_K K'$ over the annulus $A^1_{K'}[\eta_0, 1)$ with the described properties might not exist.  In practice, we need the following two technical arguments to read off the multiset of refined Swan conductors.
\begin{itemize}
\item [(a)] The above picture can be better described over a field.  Namely, we have the description of the actions of $\partial_0, \dots, \partial_m$ as in \eqref{E:differential operators} over the completion of $K(T)$ with respect to the $\eta$-Gauss norm for any $\eta\in [\eta,1)$. By taking common eigenvalues as explained above, we can define a version of refined Swan conductors, called the \emph{refined radii}, of the differential module at each radius $\eta$.  We then show that the refined radii, as we vary the radius of the Gauss norm, also vary in a nice way when $\eta$ is sufficiently close to $1$: they form a unique multiset consisting of elements of $\Omega_{\calO_k}^1(\log) \otimes_{\calO_k} \pi_k^{-b}\kappa_k^\alg$, independent of the choice of $\eta$.  We then just simply define this multiset to be the multiset of \emph{refined Swan conductors} of the representation $\rho$; this does not require any good matrices representing the actions of $\partial_j$ over the entire annulus.
\item [(b)] When the spectral norms of the differential operators are smaller than their operator norms over the base field, the description \eqref{E:differential operators} requires some modification.  Over the completion of $K(T)$ for the $\eta$-Gauss norm, we may find a basis with respect to which $\partial_j^{p^r}$ for some appropriate $r\in \NN$ acts by some nice matrix as in \eqref{E:differential operators}.  We then take the common eigenvalues of those matrices and define the refined radii to be the $p^r$-th roots of these eigenvalues.  When trying to prove results in this case, we use a technique called \emph{Frobenius antecedents} developed in \cite{kedlaya-xiao}, which reduces the question at hand to the case when the spectral norms are bigger than the operator norms.
\end{itemize}

We can also define the notion of refined Artin conductors using a variant of the definition of the refined Swan conductors, in which the effect of log structure is removed, which amounts to replacing the factor $T^{-b-1}$ by $T^{-b}$ in \eqref{E:differential operators}.

Part of the content in this paper on refined Swan conductors has been already included in the author's thesis~\cite{thesis}.  However, we feel the present paper provides a better context
for our development of refined Swan conductors.  We also fill in some gaps in \cite{thesis}. 

 To compare our definition of refined Swan conductors with Saito's, we proceed as in \cite{xiao1} by introducing the thickening spaces
which tie the $p$-adic differential equations together with the rigid analytic spaces considered by Abbes and Saito. More precisely, we may first realize a finite Galois extension $l$ of $k$ as the corresponding extension of the function fields of a finite \'etale extension of smooth affine varieties $Y \to X$.  We may further assume that both $X$ and $Y$ lift to smooth formal schemes $\bbX$ and $\bbY$.  The differential module associated to a $p$-adic representation of $\Gal(l/k)$ lives over the a subspace of the tube of $X$ embbeded diagonally in $\bbX \times \bbX$, which is a rigid analytic subspace of the generic fibre of $\bbX\times \bbX$ and is called the \emph{thickening space}.  We carefully study the construction of the differential module and compare that with Saito's description of the special fibre of the formal scheme $\bbY$.  
The
 core of the comparison result is to identify the data  defining an Artin-Schreier cover of $\AAA^m_{\kappa_k}$ with the data coming from the associated Dwork isocrystals as a differential module.

We also remark that when $k$ is an $n$-dimensional higher local field of characteristic $p>0$, the refined conductors induce a ramification filtration on $G_k$ indexed by $\QQ^n$ with lexicographic order.  This is expected to be compatible with certain filtration on the Milnor $K$-groups via Kato's class field theory.

Finally,
we study the relation of the refined Swan conductors with the variation of intrinsic radii (certain form of Swan conductors) over a polyannulus.  We prove that the valuations of the refined Swan conductors at a vertex of the polygon associated to the polyannulus encode some information about the slopes of the log-affine functions of the intrinsic radii at that vertex.  For the precise statement, one may
consult Proposition~\ref{P:polyannuli-Swan-conductor}. 

\subsection*{Plan of the paper}

Section 1 is devoted to developing the theory of refined radii, the analog of refined conductors over a complete nonarchimedean field.  In the first two subsections, we set up notation and recall some basic results on differential modules from \cite{kedlaya-xiao}.  We define the refined radii in Subsection~\ref{S:refined-radii} and prove a decomposition Theorem~\ref{T:refined-decomposition}, which separates pieces with different refined radii in a differential module. In Subsection~\ref{S:multi-derivations} we consider the case where we allow multiple derivations to interact.  In Subsection~\ref{S:one-dim} we study how the refined radii vary on an annulus or a disc, when the radii are log-affine functions.  We then define
the refined conductors for solvable differential modules over  an annulus in Subsection~\ref{S:differential-conductor}.

In Section 2 we apply the theory of refined conductors for solvable differential modules to define refined conductors for Galois representations.
In the first two subsections we recall the construction of differential modules following \cite{kedlaya-swan1}, and deduce some basic properties.  In Subsection~\ref{S:refined} we define the homomorphism of refined conductors.  Subsection~\ref{S:higher-local-fields} briefly discusses an application	to the higher local fields.

In Section 3 we compare our definition with that of Saito, which is reviewed in Subsection~\ref{S:saito-refined}. In Subsection~\ref{S:lifting-spaces} we realize the extension of fields as a finite \'etale cover of varieties and lift them to rigid analytic spaces over $K$.
In Subsection~\ref{S:Dwork-isoc} we do a crucial calculation on the differential module structure of 
Dwork isocrystals to determine their refined radii; this calculation forms the heart of our proof of
the comparison theorem.  We wrap up Section 3 with 
a proof of the comparison  Theorem~\ref{T:comparison} in Subsection~\ref{S:comparison}.

In Section 4 we focus on the interplay of refined Swan conductors with the toroidal variation of Swan conductors.  A few technical lemmas are discussed in Subsection~\ref{S:boundary}, and the main theorems are proved in Subsection~\ref{S:polyannuli-variation}.

\subsection*{Acknowledgments}

I would like to thank
my advisor, Kiran Kedlaya, for generating some of
crucial ideas, for helpful discussions, and for spending hours reviewing early drafts.
I thank
Ivan Fesenko for suggesting the application to higher local fields.
I have benefited from discussions with Takeshi Saito.
I thank Ahmed Abbes for inviting me to the conference \emph{Journ\'ees de G\'eom\'etrie Arithm\'etique de Rennes}, which has benefited me a lot.
I am also grateful for Matthew Morrow and Sandeep Varma for their help on the presentation of this paper.

\section{Theory of differential modules}
\label{Section 1}
Our systematic study of differential modules proceeds in two stages: first over a complete nonarchimedean field, and then over an annulus over a complete nonarchimedean field.  In the former case, the spectral norm, or equivalently \emph{the radius of convergence}, of the differential operator is a very important invariant;  when the differential module has pure radii, we will focus on a secondary information of the differential module, called the \emph{refined radii}.
In the latter case, Kedlaya and the author \cite{kedlaya-xiao} proved that  the radii of convergence of a differential module over an annulus give rise to piecewise log-affine functions as one vary the radii on the annulus; we will again focus on the secondary data: the refined radii. In the case when the aforementioned piecewise log-affine functions are in fact log-affine, we prove that the multisets of refined radii of the differential module at all radii are the \emph{same}, if we naturally identify the spaces where these refined radii live in.

\subsection{Setup}
\label{S:setup}
This subsection is mainly to explain our convention on notations; however, the commutative algebra Lemma~\ref{L:proj-intersect} will become a very useful tool later as explained in Remark~\ref{R:proj-intersect}.

\begin{notation}
By a \emph{multiset} $S$, we mean a set where we allow elements to have multiplicity.  For $s \in S$, the \emph{multiplicity} of $s$ in $S$ is denoted by $\multi_s(S)$.  When $S$ consists of a single element (with multiplicity), we call it \emph{pure}.
\end{notation}

\begin{notation}

For any field $K$ that will be considered in this paper,
$K^{\alg}$ will denote a fixed algebraic closure. We let $K^\sep$ denote the separable closure of $K$ inside $K^\alg$.  Set $G_K = \Gal(K^\sep/K)$.
For a finite Galois extension $L/K$ (inside $K^\sep$), we denote its Galois group
 by $G_{L/K} = \Gal(L/K)$.
 
 For $e \in \NN$, we use $\mu_e$ to denote the set of $e$th roots of unity in $K^\alg$.
\end{notation}

\begin{notation}
\label{N:fields}
By a \emph{nonarchimedean} field, we mean a field $K$ equipped with a nonarchimedean norm $|\cdot| = |\cdot|_K: K^\times \rar \RR^\times_+$.  A subring of $K$ (with the induced norm and topology) is called a \emph{nonarchimedean ring}.

For a nonarchimedean field $K$, denote the 
ring of integers of $K$ by $\calO_K = \{x \in K| |x| \leq 1 \}$ and the  maximal
ideal of $\calO_K$ 
 by $\gothm_K = \{ x\in K | |x|<1\}$; denote the residue field of $K$ by $\kappa_K = \calO_K / \gothm_K$. 
We reserve the letter $p$ for the characteristic of $\kappa_K$.  If $\Char \kappa_K = p>0$ and $\Char K = 0$, we normalize the norm on $K$ so that $|p| = 1/p$.  For an element $a \in \calO_K$, we denote its image in $\kappa_K$ under the reduction map by $\bar a$.  In case $K$ is discretely valued, let $\pi_K$ denote a uniformizer of $\calO_K$ and let $v_K(\cdot)$ be the corresponding valuation on $K$, normalized so that $v_K(\pi_K) = 1$.

For a nonarchimedean field $K$ and $s \in \RR$, we set
\[
\gothm^{(s)}_K = \{x \in K\ \big|\ |x| \leq \ee^{-s}\}, \quad
\gothm^{(s)+}_K = \{x \in K\ \big|\ |x| < \ee^{- s}\}, \quad \textrm{and } \kappa_K^{(s)} = \gothm^{(s)}_K / \gothm^{(s)+}_K.
\]
If $s \in -\log\,|K^\times|$, there exists a non-canonical isomorphism $\kappa_K \simeq \kappa_K^{(s)}$.
For $a \in K$ with $|a| \leq \ee^{- s}$, we sometimes denote its image in $\kappa^{(s)}_K$ by $\bar a^{(s)}$.  In particular, $\kappa_K^{(0)} = \kappa_K$ and $\bar a^{(0)} = \bar a$ if $v(a) \geq 0$.
\end{notation}

\begin{notation}
Let $J$ be an index set.  We  use $e_J$ to denote a tuple $(e_j)_{j \in J}$.  For another tuple $u_J$, set $u_J^{e_J} = \prod_{j \in J} u_j^{e_j}$, if all but finitely many of the $e_j$'s are equal to $0$. We also use $\sum_{e_J = 0}^{n}$ to denote the sum over $e_j \in \{0, 1, \dots, n\}$ for each $j \in J$ provided 
$e_j \neq 0$ for only finitely many $j$; for notational simplicity, we may suppress the range of the summation when it is clear.  If $J$ is finite, put $|e_J| = \sum_{j \in J}|e_j|$ and $(e_J)!=\prod_{j \in J} (e_j!)$.
\end{notation}

\begin{convention}
Throughout this paper, all derivations on topological modules will
be assumed to be continuous; in particular, $\Omega^1_{R/S}$ will denote the
module of continuous differentials on the (topological) ring $R$ relative to the (topological) base ring $S$; we may suppress $S$ from the notation when $S = \Fp$, 
$\ZZ$ or $\Zp$.
Moreover, all derivations  on nonarchimedean rings will be assumed to be bounded (i.e.,
to have bounded operator norms). All connections considered will be
assumed to be integrable. 
\end{convention}

\begin{notation}
For a matrix $A = (A_{ij})$ with coefficients in a nonarchimedean ring, we use $|A|$ to denote the supremum among the norms of the entries $A_{ij}$ of $A$.
\end{notation}

\begin{hypothesis}
For the rest of this subsection, we assume that $K$ is a complete nonarchimedean field.
\end{hypothesis}

\begin{notation}\label{N:affinoids}
Let $I \subset [0, +\infty)$ be an interval and let $n \in \NN$. Let 
\[
A_K^n(I) = \big\{(x_1, \dots, x_n) \in K^\alg\ \big|\ |x_i| \in I \textrm{ for } i = 1, \dots, n\big\}
\]
denote
the polyannulus of dimension $n$ with radii in $I$.
(We do not impose any rationality condition on the endpoints of $I$, so this
space should be viewed as an analytic space in the sense of 
Berkovich \cite{berkovich}.)
If $I$ is written explicitly in terms of its
endpoints (e.g., $[\alpha, \beta]$),
we suppress the parentheses around $I$ (e.g., 
$A_K^n[\alpha, \beta]$).
\end{notation}

\begin{notation}
 Let $0 \leq \alpha \leq \beta < +\infty$.  We put 
\begin{align*}
K\langle \alpha/t, t / \beta \rangle &= \Big\{ \sum_{n \in \ZZ} a_nt^n \ \Big|\  |a_n| \eta^n \rightarrow 0 \textrm{ as } n \rightarrow \pm \infty, \textrm{ for any } \eta \in [\alpha, \beta] \Big\}, \\
K\langle \alpha/t, t / \beta \}\} &= \Big\{ \sum_{n \in \ZZ} a_nt^n\ \Big|\  |a_n| \eta^n \rightarrow 0 \textrm{ as } n \rightarrow \pm \infty, \textrm{ for any } \eta \in [\alpha, \beta) \Big\},\\
K\{\{ \alpha/t, t / \beta \rrbracket_0 &= \Big\{ \sum_{n \in \ZZ} a_nt^n\ \Big|\  |a_n| \eta^n \rightarrow 0 \textrm{ and } |a_n|\beta^n \textrm{ is bounded, as } n \rightarrow \pm \infty,\textrm{ for any }\eta \in (\alpha, \beta) \Big\}.
\end{align*}

When $\alpha = 0$, we simply use $K\langle t / \beta\rangle$ and $K\{\{ t/ \beta\}\}$ to denote the first and second rings above, respectively.  We also put 
\[
K \llbracket t /\beta\rrbracket _0 = \big\{ \sum_{n =0}^\infty a_n t^n \big| |a_n|\beta^n \textrm{ is bounded as }n \rar \infty\big\}.
\]

For $I = \{1, \dots, n\}$ and a nonarchimedean ring $R$, we use $R \langle u_I \rangle$ to denote the Tate algebra, consisting of formal power series $\sum_{e_I \in \ZZ_{\geq 0}} a_{e_I}u_I^{e_I}$ with $a_{e_I} \in R$ and $|a_{e_I}| \rightarrow 0$ as $|e_I| \rightarrow +\infty$.  For $(\eta_i)_{i \in I} \in (0, +\infty)^n$, the \emph{$\eta_I$-Gauss norm} on the polynomial ring $R[t_I]$
is the norm $|\cdot|_{\eta_I}$ given by $$
\left| \sum_{e_I} a_{e_I} t_I^{e_I}\right|_{\eta_I} = \max_{e_I} \left\{ |a_{e_I}| \cdot \eta_I^{e_I}\right\};
$$
this norm extends uniquely to multiplicative norms on $\Frac \big(R[t_I]\big)$, and on $R\langle t_I\rangle$ in case $|\eta_i| \leq 1$ for any $i \in I$.

For $\eta \in [\alpha, \beta]$ with $\eta \neq 0$, the $\eta$-Gauss norm on $K[t]$ extends to multiplicative norms 
on $K \langle \alpha / t, t / \beta \rangle$ and $K \llbracket t / \beta \rrbracket_0$, on
$K \langle \alpha / t, t / \beta \}\}$ in case $\eta \neq\beta$, and on $K \{\{ \alpha / t, t/ \beta \rrbracket_0$ in case $\eta \neq \alpha$.
\end{notation}

We record here a lemma in
commutative algebra which will be frequently used (implicitly) when gluing decompositions.

\begin{lemma}
\label{L:proj-intersect}
Let 
\[
\xymatrix{
R \ar[r] \ar[d] &  S \ar[d] \\
T \ar[r] & U
}
\]
be a commuting diagram of inclusions of integral domains,
such that the intersection
$S \cap T$ within $U$ is equal to $R$. Let $M$ be a finite
locally free $R$-module. Then the intersection of $M \otimes_R S$
and $M \otimes_R T$ within $M \otimes_R U$ is equal to $M$.
\end{lemma}
\begin{proof}
See \cite[Lemma~2.3.1]{kedlaya-xiao}.
\end{proof}

\begin{remark}
\label{R:proj-intersect}
Let us remark on how this lemma
is used in this paper.  We often apply this lemma to the $R$-module $\mathrm{End}(M)$ over $R$ for a differential module $M$.  More precisely, we often encounter the situation when we can write both $M \otimes_R S$ and $M \otimes_R T$ as direct sums of two submodules such that both direct sum decompositions, when tensored with $U$,
give the same direct sum decomposition of
$M \otimes_R U$.  We view the projections constituting the direct sum decompositions
as elements in $\mathrm{End}(M) \otimes_R S$, $\mathrm{End}(M) \otimes_R T$, and $\mathrm{End}(M) \otimes_R U$, respectively.  By Lemma~\ref{L:proj-intersect}, we see that the projections above are actually the images of one element  of $\End(M)$ under the natural maps; this element defines a direct sum decomposition of $M$ which when tensored with $S$ (respectively, $T$) yields the given
direct sum decomposition of $M \otimes_R S$ (respectively, $M \otimes_R T$).  In other words, we can ``glue" the direct sum decompositions of $M \otimes_R S$ and of $M \otimes_R T$ along $M\otimes_R U$ to get a direct sum decomposition of $M$ (over $R$).
\end{remark}

\subsection{Differential modules and radii of convergence}

The starting point of the theory of nonarchimedean differential modules is to understand differential modules over a nonarchimedean field.  One of the important tools is the \emph{Newton polygon} associated to a \emph{cyclic vector}, which gives many numerical information if the spectral norm of the differential operator is strictly bigger than the operator norm on the base field.  To extend interesting results across the threshold imposed by the operator norm mentioned above, we restrict ourselves to the case when the differential operator is of \emph{rational type}, i.e. its metric properties resembles $d/dX$ acting on the completion of $\Qp(X)$ with respect to the $1$-Gauss norm; in this case, we may entirely remove the restriction on spectral norms by considering the \emph{Frobenius antecedents} of the differential modules.

\begin{definition}
\label{D:differential-module}
Let $K$ be a differential ring, i.e.,
a ring equipped with a derivation $\partial$.  Let $K\{T\}$ denote the (noncommutative) ring of twisted polynomials over $K$ \cite{ore}; its elements are finite formal sums $\sum_{i \geq 0} a_iT^i$ with $a_i \in K$, multiplied according to the rule $Ta = aT + \partial(a)$ for $a \in K$.

A \emph{$\partial$-differential module} over $K$ is a finite projective $K$-module $V$ equipped with an action of $\partial$ (subject to the Leibniz rule); any $\partial$-differential module over $K$ inherits a left action of $K\{T\}$ where $T$ acts via $\partial$.  The \emph{rank} of $V$ is the rank of $V$ as a $K$-module.  The module dual $V^\dual  = \Hom_K(V,K)$ of $V$
may be viewed as a $\partial$-differential module
by setting $(\partial f)(\bbv) = \partial(f(\bbv)) - f(\partial(\bbv))$.
We say $V$ is \emph{free} if $V$ is free as a module over $K$. We say $V$ is \emph{trivial} if it is isomorphic to $K^{\oplus d}$ for some $d\in \NN$ as a $\partial$-differential module.

For a $\partial$-differential module $V$  free of rank $d$ over $K$, an element $\bbv \in V$ is called a \emph{cyclic vector} if $\bbv, \partial \bbv, \dots, \partial^{d - 1}\bbv$ form a basis of $V$ as a $K$-module.  A cyclic vector defines
an isomorphism $V \simeq K\{T\} / K \{T\}P$ of $\partial$-differential modules, where $P\in K\{T\}$ is some monic twisted polynomial of degree $d$, and the $\partial$-action on $K\{T\}/K\{T\}P$ is the left multiplication by $T$.  If $K$ is a differential field  of characteristic 0, $V$ always has a cyclic vector (see  \cite[Theorem~III.4.2]{dgs} or \cite[Theorem~5.4.2]{kedlaya-course}). 

For a $\partial$-differential module $V$, we put $H_\partial^0(V) = \Ker \partial$.
\end{definition}

\begin{hypothesis} \label{H:K-diff-field}
For the rest of this subsection,
we assume that $K$ is a complete nonarchimedean field of characteristic zero, equipped with a derivation $\partial$ with
operator norm $|\partial|_K < \infty$,
and that $V$ is a nonzero $\partial$-differential module over $K$.
\end{hypothesis}

\begin{definition}\label{D:partial-radii}
Let $p$ denote the residual characteristic of $K$; we conventionally set
\[
\omega = \begin{cases} 1 & p = 0 \\ p^{-1/(p-1)} & p > 0 \end{cases}.
\]
The \textit{spectral norm of $\partial$ on $V$} is defined to be
$
|\partial|_{\sp, V} = \lim_{n \rar \infty} |\partial^n|_V^{1/n}
$
for any fixed $K$-compatible norm $|\cdot|_V$
on $V$.
Define the \emph{generic $\partial$-radius} of $V$ to be
$
R_\partial(V) = \omega|\partial|_{\sp, V}^{-1};
$
note that $R_\partial(V) > 0$.
Let $\serie{V_}d$ be the Jordan-H\"older constituents of $V$ as a $K\{T\}$-module.
We define the multiset $\gothR_\partial(V)$ of \emph{(extrinsic) subsidiary $\partial$-radii} of $V$  to be the collection of $R_\partial(V_i)$ with multiplicity $\dim V_i$ for $i = \serie{}d$.  Let $R_\partial(V;1) \leq \cdots \leq R_\partial(V;\dim V)$ denote the elements of $\gothR_\partial(V)$ in nondecreasing order.
We say that $V$ has \emph{pure $\partial$-radii} if $\gothR_\partial(V)$ consists of $\dim V$ copies of $R_\partial(V)$.
\end{definition}

\begin{definition}\label{D:Taylor-series}
Let $R$ be a complete $K$-algebra.  For $\bbv \in V$ and $x \in R$, we define the \emph{$\partial$-Taylor series} of $\bbv$ with respect to $x$ to be
\begin{equation}\label{E:Taylor}
\TT (\bbv; \partial; x) = \sum_{n = 0}^\infty \frac {\partial^n (\bbv)}{n!}x^n \in V \otimes_K R,
\end{equation}
in case this series converges. When $V = K$, the $\partial$-Taylor series \eqref{E:Taylor} with respect to a fixed $x \in R$ gives 
a homomorphism $K \rar R$ of rings, if it converges for all $\bbv \in V=K$.
  For general $V$, the $\partial$-Taylor series \eqref{E:Taylor} with respect to the same fixed $x \in R$ gives a homomorphism of $K$-modules $V \rar V \otimes_K R$ respecting the aforementioned ring homomorphism, if both homomorphisms converge.
\end{definition}

\begin{lemma}\label{L:basic-IR-proposition}
Let $V$, $V_1$, and $V_2$ be nonzero $\partial$-differential modules over $K$.

\emph{(a)} If $\exact 0 {V_1} V {V_2} 0$ is exact, then we have
$
\gothR_\partial(V) = \gothR_\partial(V_1) \cup \gothR_\partial(V_2).
$

\emph{(b)} We have
$
\gothR_\partial(V^\dual) = \gothR_\partial(V ).
$

\emph{(c)} We have $R_\partial(V_1 \otimes V_2) \geq \min \left\{ R_\partial(V_1),R_\partial(V_2) \right\}$.  If $V_1$ is irreducible and $R_\partial(V_1) < R_\partial(V_2)$, then $V_1 \otimes V_2$ has pure $\partial$-radius $R_\partial(V_1)$.

\emph{(d)} Let $f: K \rar K\llbracket T/u\rrbracket_0$ be the homomorphism given by $f(x) = \TT(x; \partial; T)$.  Then $f^*V = V \otimes_{K, f}K \llbracket T/u\rrbracket_0$ is a $\partial_T = \partial / \partial T$-differential module over $K\llbracket T/u \rrbracket_0$. For $r \in (0, R_\partial(K))$, $R_\partial(V) \geq  r$ if and only if $f^*V$ restricts to a trivial $\partial_T$-differential module over $A^1_K[0, r)$. 
\end{lemma}
\begin{proof}
The statements (a)--(c) are \cite[Lemma~1.2.9]{kedlaya-xiao} and the statement (d) is \cite[Proposition~1.2.14]{kedlaya-xiao}.
\end{proof}

\begin{definition}\label{D:newton-polygon}
For $P(T) = \sum_i a_iT^i \in K[T]$ or $K\{T\}$ a nonzero (possibly twisted) polynomial, define the \emph{Newton polygon} of $P$ as the lower convex hull of the set $\{(-i, -\log|a_i|)\} \subset \RR^2$.
\end{definition}

\begin{proposition}[Christol-Dwork]
\label{P:spec-norm-from-NP}
Suppose that $V \simeq K\{T\} / K\{T\}P$,
and let $s$ be the lesser of $-\log |\partial|_K$ and the
least slope of the Newton polygon of $P$. Then we have
$
\max\{|\partial|_K, |\partial|_{\sp, V} \} = \ee^{-s}.$
More generally, the 
multiplicity of any $s' < -\log |\partial|_K$ as a slope of the Newton polygon 
of $P$ coincides with the multiplicity of $\omega \ee^{s'}$ in $\gothR_\partial(V)$. 
\end{proposition}
\begin{proof}
This is \cite[Theorem~6.5.3]{kedlaya-course}.
\end{proof}

\begin{definition}\label{D:admissible-operator}
We say a derivation $\partial$ on $K$ is of \emph{rational type} if there exists
$u \in K$ such that the following conditions hold (in this case,
we call $u$ a \emph{rational parameter} for $\partial$):
\begin{enumerate}
\item[(i)]
we have $\partial(u) = 1$ and $|\partial|_K = |u|^{-1}$, and
\item[(ii)]
for each positive integer $n$, $|\partial^n/n!|_K \leq |\partial|_K^n$.
\end{enumerate}
If $\partial$ is of rational type, the inequalities in (ii) are in fact equalities, which yields that $|\partial|_{\sp, K} = \omega |\partial|_K$; see \cite[Definition~1.4.1]{kedlaya-xiao}.
\end{definition}

\begin{lemma}\label{L:tame-extension}
Let $\partial$ be a derivation on $K$ of rational type with $u$ as a rational parameter and let $L/K$ be a finite tamely ramified extension.  Then the unique extension of $\partial$ to $L$ is of rational type with $u$ again as a rational parameter.
\end{lemma}
\begin{proof}
This is \cite[Lemma~1.4.5]{kedlaya-xiao}.  
\end{proof}

\begin{remark}\label{R:enlarge-K}
We sometimes need to replace $K$ by the completion of $K(x)$ with respect to
the $\eta$-Gauss norm for some $\eta \in \RR_{>0}$, where $x$ is transcendental over $K$ and we set $\partial x=0$.  The derivation $\partial$ is again of rational type when acting on the new field.
\end{remark}

\begin{definition}
When $\partial$ is  of rational type, it is more convenient to consider $\partial$-radii with a differential normalization, as follows.
For $V$ a $\partial$-differential module, we define the \emph{intrinsic $\partial$-radius} of $V$ to be $IR_\partial(V) = |\partial|_{\sp, K} / |\partial|_{\sp, V} = |\partial|_K \cdot R_\partial(V)$. We define the multiset of \emph{intrinsic subsidiary $\partial$-radii} to be $\gothI\gothR_\partial(V) = |\partial|_K \cdot \gothR_\partial(V)$.  We put $IR_\partial(V;i)  = |\partial|_K \cdot R_\partial(V;i)$ for $i = 1, \dots, \dim V$.  We say that $V$ has \emph{pure intrinsic $\partial$-radii} if $\gothI\gothR_\partial(V)$ consists of $\dim V$ copies of one single number.
\end{definition}

\begin{hypothesis}\label{H:partial}
From now on, we assume that $K$ is a complete nonarchimedean field 
of characteristic zero and residual characteristic $p$, 
equipped with a derivation 
$\partial$ of rational type.  We \emph{fix} $u \in K$ a rational parameter of $\partial$.
We also assume $p>0$ unless otherwise specified.
\end{hypothesis}

\begin{construction}\label{Cstr:j-frobenius}
We construct the $\partial$-Frobenius as follows.
If $K$ contains a primitive $p$-th root of unity $\zeta_p$, we may define an isometric action of the group $\ZZ/p\ZZ$ on $K$ using $\partial$-Taylor series:
$$
x^{(i)} = \TT(x; \partial; (\zeta_p^i - 1) u), \qquad (i \in \ZZ / p\ZZ, x \in K);
$$
in particular, $u^{(i)} = \zeta_p^i u$.  Let $K^{(\partial)}$ be the fixed subfield of $K$ under this action; in particular, $u^p \in K^{(\partial)}$.  Hence, we have a Galois extension $K/K^{(\partial)}$ generated by $u$ with Galois group $\ZZ / p\ZZ$. 
If $K$ does not contain all $p$-th roots of unity, we may still define $K^{(\partial)}$ by first constructing $(K(\mu_p))^{(\partial)}$ and then applying the Galois descent; in this case, the extension  $K/K^{(\partial)}$ may not be Galois.

We call the inclusion $\varphi^{(\partial)*}: K^{(\partial)} \inj K$ 
the \emph{$\partial$-Frobenius morphism}.  We view $K^{(\partial)}$ as being equipped with the derivation
$\partial' = \partial / (pu^{p-1})$; it is a derivation on $K^{(\partial)}$ because a simple calculation shows that $(\partial(x))^{(i)} = \zeta_p^i \partial(x^{(i)})$ for any $x \in K$, yielding that $\partial'(x)$ is invariant under the $\ZZ / p\ZZ$-action if $x \in K^{(\partial)}$. By \cite[Lemma~1.4.9]{kedlaya-xiao}, $\partial'$ is of rational type on $K^{(\partial)}$.

We sometimes use $\varphi^{(\partial, n)}: K^{(\partial, n)} \inj K$ to denote the $p^n$-th $\partial$-Frobenius obtained by applying the above construction $n$ times; if $K$ contains a primitive
$p^n$-th root of unity $\zeta_{p^n}$, this is the same as the fixed field for the natural
action of $\ZZ/p^n \ZZ$ on $K$ given by $x^{(i)} = \TT(x; \partial; (\zeta_{p^n}^i - 1)u)$ for $i \in \ZZ/ p^n\ZZ$.
\end{construction}

\begin{remark}\label{R:Frobenius-depends-on-u}
We point out that the definitions of $\partial$-Frobenius and $K^{(\partial)}$ depend on the choice of the rational parameter $u$.
\end{remark}

\begin{lemma}\label{L:Frob-does-to-kappa}
The residue field $\kappa_{K^{(\partial)}}$ contains $\kappa_K^p$.
\end{lemma}
\begin{proof}
We know that $K$ is generated by $u$ over $K^{(\partial)}$.  If $|u| \notin |K^{(\partial)\times}|$, $K^{(\partial)}$ will have the same residue field as $K$ does.  If $|u| \in |K^{(\partial)\times}|$, let $x \in K^{(\partial)}$ be an element such that  $|x| = |u|$.  Then $\kappa_K$ is generated over $\kappa_{K^{(\partial)}}$ by $\overline{u/x}$, whose $p$-th power lies in $\kappa_{K^{(\partial)}}$.  The statement follows.
\end{proof}

\begin{definition}
Given a $\partial'$-differential module $V'$ over $K^{(\partial)}$, 
its \emph{$\partial$-Frobenius pullback} is the $\partial$-differential module $\varphi^{(\partial)*} V' = V' \otimes_{K^{(\partial)}} K$ over $K$, where
$$
\partial (\bbv' \otimes x) = pu^{p-1} \partial'(\bbv') \otimes x + \bbv' \otimes \partial(x) \qquad (\bbv' \in V', x \in K).
$$

For a $\partial$-differential module $V$ over $K$, we define the \emph{$\partial$-Frobenius descendant} of $V$ to be the $K^{(\partial)}$-module
$\varphi^{(\partial)}_* V$ obtained from $V$ by restriction along $\varphi^{(\partial)*}: K^{(\partial)} \rar K$ and viewed as a $\partial'$-differential module over $K^{(\partial)}$ with the action given by $\partial'(v) = \partial(v)/ pu^{p-1}$ for any $v \in V$.

Let $V$ be a $\partial$-differential module over $K$ such that $IR_\partial(V) > p^{-1/(p-1)}$. A \emph{$\partial$-Frobenius antecedent} of $V$ (which always exists as is shown in Lemma~\ref{L:frob-properties}(c) below)  is a $\partial'$-differential module $V'$ over $K^{(\partial)}$ such that $V \cong \varphi^{(\partial)*}V'$ and $IR_{\partial'}(V') > p^{-p/(p-1)}$.
\end{definition}

\begin{lemma} \label{L:frob-properties}
The $\partial$-Frobenius pullbacks and descendants have the following properties.
\begin{enumerate}
\item[\emph{(a)}]
For $V'$ a $\partial'$-differential module over $K^{(\partial)}$, we have
$IR_\partial(\varphi^{(\partial)*} V') \geq \min \{IR_{\partial'}(V')^{1/p}, p\, IR_{\partial'}(V')\}$.  Moreover, if $IR_{\partial'}(V') \neq p^{-p/(p-1)}$, the above inequality is in fact an equality.
\item[\emph{(b)}]
For $V$ a $\partial$-differential module over $K$, there is a canonical isomorphism
$
\varphi^{(\partial)*}\varphi^{(\partial)}_* V \cong V^{\oplus p}.
$
\item[\emph{(c)}]  For $i = 0, \dots, p-1$, let $W_i^{(\partial)}$ be the $\partial'$-differential module over $K^{(\partial)}$ with one generator $\bbv$ (which is a proxy of $u^i$), such that $\partial'(\bbv) = \frac ip u^{-p}\bbv$.  Then we have $IR_{\partial'}(W_i^{(\partial)}) = p^{-p/(p-1)}$ for $i = 1, \dots, p-1$.  For any $\partial$-differential module $V$ over $K$, we have canonical isomorphisms
$
\iota_i : (\varphi^{(\partial)}_* V) \otimes W_i^{(\partial)} \cong \varphi^{(\partial)}_* V$ for $i = \seriezero{}p-1$.
Moreover, 
 a submodule $U$ of $\varphi^{(\partial)}_* V$ is itself the $\partial$-Frobenius descendant of a submodule of $V$ if and only if $\iota_i(U \otimes W_i^{(\partial)}) = U$ for $i = \seriezero{}p-1$.
 
For $V_1$ and $V_2$ $\partial$-differential modules over $K$, we have
\[
 \varphi^{(\partial)}_*V_1 \otimes \varphi^{(\partial)}_*V_2 = \big(\varphi^{(\partial)}_*(V_1 \otimes V_2) \big)^{\oplus p}.
 \]

For $V'$ a $\partial'$-differential module over $K^{(\partial)}$, we have $\varphi^{(\partial)}_*\varphi^{(\partial)*}V'
\cong V' \oplus \bigoplus_{i=1}^{p-1} V'\otimes W^{(\partial)}_i$.

\item[\emph{(d)}]\emph{(Christol-Dwork)} Let $V$ be a $\partial$-differential module over $K$ such that $IR_\partial(V) > p^{-1/(p-1)}$. Then there exists a unique $\partial$-Frobenius antecedent $V'$ of $V$.  Moreover, we have $IR_{\partial'}(V') = IR_\partial(V)^p$.
\item[\emph{(e)}] Let $V$ be a $\partial$-differential module over $K$.  Then we have
$$
\gothI\gothR_{\partial'}(\varphi^{(\partial)}_*V) = 
\bigcup_{r \in \gothI\gothR_\partial(V)} \left\{
\begin{array}{ll}
\big\{r^p, \underbrace{p^{-p/(p-1)}, \dots, p^{-p/(p-1)}}_{p-1 \textrm{ times}}\big\} & r > p^{-1/(p-1)}\\
\big\{\underbrace{p^{-1} r, \dots, p^{-1} r}_{p \textrm{ times}} \big\} & r \leq p^{-1/(p-1)}.
\end{array}
\right.
$$
In particular, we have $IR_{\partial'}(\varphi^{(\partial)}_* V) = \min \{p^{-1} IR_\partial(V),\ p^{-p/(p-1)}\}$.
\end{enumerate}
\end{lemma}
\begin{proof}
For (a), see \cite[Lemma~1.4.11]{kedlaya-xiao} and \cite[Corollary~1.4.20]{kedlaya-xiao}.  (b) and (c) are straightforward.  For (d), see \cite[Theorem~10.4.2]{kedlaya-course}.  For (e), see \cite[Theorem~1.4.19]{kedlaya-xiao}.
\end{proof}

\begin{remark}\label{R:frob-over-annulus}
As in \cite[Theorem~10.4.4]{kedlaya-course}, one can form a version of 
Lemma~\ref{L:frob-properties}(d)
for differential modules over discs or annuli.
\end{remark}

For the following theorem, we do not assume $p>0$.
\begin{theorem}\label{T:decomp-over-field-complete}
Let $V$ be a $\partial$-differential module over $K$.  Then there exists a unique decomposition of $\partial$-differential modules:
$$
V = \bigoplus_{r \in (0, 1]}V_r,
$$
where every subquotient of $V_r$ has pure intrinsic $\partial$-radii $r$.  Moreover, $V_r =0$ if $r \notin |K^\times|^{\QQ}$.
\end{theorem}
\begin{proof}
For the decomposition, see \cite[Theorem~1.4.21]{kedlaya-xiao}.  The rationality of those $r$ such that $V_r\neq0$ follows from Proposition~\ref{P:spec-norm-from-NP} when $r < \omega$ and from taking $\partial$-Frobenius antecedents in the general case.
\end{proof}

\begin{definition}
We call $\oplus_{r \in (0, \omega)} V_r$ the \emph{visible part} of $V$ and $\oplus_{r \in [\omega, 1]} V_r$ the \emph{non-visible part} of $V$.  If $V$ consists of only its visible part, we say $V$ has \emph{visible (intrinsic) $\partial$-radii}; similarly, if $V$ consists of only its non-visible part, we say $V$ has \emph{non-visible (intrinsic) $\partial$-radii}.
\end{definition}

\begin{remark}
\label{R:antecedent-vs-pushforward}
Let $V$ be a $\partial$-differential module over $K$ with pure intrinsic $\partial$-radii $IR_\partial(V) > p^{-1/(p-1)}$.  By Lemma~\ref{L:frob-properties}(d), $V$ has a $\partial$-Frobenius antecedent $V'$.
By Lemma~\ref{L:frob-properties}(c),
\[
\varphi^{(\partial)}_*V = \varphi^{(\partial)}_*\varphi^{(\partial)*}V'
\cong V' \oplus \Big( \bigoplus_{i=1}^{p-1} V'\otimes W^{(\partial)}_i \Big).
\]
This decomposition coincides with the decomposition obtained by applying Theorem~\ref{T:decomp-over-field-complete} to $\varphi^{(\partial)}_*V$.
\end{remark}

\subsection{Refined radii}
\label{S:refined-radii}

When a $\partial$-differential module $V$ has \emph{pure} $\partial$-radii, we will define the multiset of \emph{refined $\partial$-radii}, certain secondary information for the differential module.  Similar to the case of $\partial$-radii, we may canonically write $V$ as a direct sum of $\partial$-differential submodules such that the multiset of refined $\partial$-radii for each direct summand consists of elements pairwise-conjugate under the action of $\Gal(K^\alg/K)$.

\begin{hypothesis}\label{H:refined}
In this subsection, let $K$ be a complete nonarchimedean field of characteristic zero and residual characteristic $p$ (possibly $p=0$), equipped with a derivation $\partial$ of rational type.  We fix $u \in K$ a rational parameter for $\partial$.  Unless otherwise specified, we assume that $V$ is a $\partial$-differential module of rank $d$ over $K$ with \emph{pure} intrinsic $\partial$-radii $IR_\partial(V)$.   Denote $s = -\log(\omega R_\partial(V)^{-1}) = -\log |\partial|_{\sp, V}$.
\end{hypothesis}

\begin{notation}
For $P(T) = T^d + a_1 T^{d-1}+ \cdots + a_d \in K[T]$ a polynomial whose Newton polygon has pure slope $s$, the multiset of the \emph{reduced roots} of $P$ consists of the reductions of the roots of $P$ in $\kappa_{K^\alg}^{(s)}$, counted with multiplicity.  If $P$ is the characteristic polynomial of a matrix $A \in \Mat(\gothm_K^{(s)})$, we call the reduced roots of $P$ \emph{the reduced eigenvalues} of $A$.
\end{notation}

\begin{notation}\label{N:dichotomy}
For $\bbb \in (0,1]$ (a proxy of $IR_\partial(V)$), we define $\lambda = \lambda(\bbb)$ and $r = r(\bbb)$ as follows.
\begin{itemize}
\item[(i)] When $\bbb < \omega$ (which happens if $V$ has pure visible intrinsic $\partial$-radii), we let $\lambda(\bbb) = 0$ and $r(\bbb) = 1$.
\item[(ii)] When $\bbb \in [\omega,1)$ and hence $p>0$ (which happens if $V$ has pure non-visible  $\partial$-radii),  let $\lambda(\bbb)$ denote the unique positive integer such that  $ \bbb \in [p^{-1/p^{\lambda(\bbb)-1}(p-1)}, p^{-1/p^{\lambda(\bbb)}(p-1)})$, and put $r(\bbb) = p^{\lambda(\bbb)}$.

\item[(iii)] When $\bbb = 1$, we let $\lambda(\bbb)= r(\bbb) = \infty$.
\end{itemize}
\end{notation}

\begin{definition}\label{D:good-norm}
Let $\bbb \in (0,1]$.
A $K$-norm $|\cdot|_V$ on $V$ is called \emph{$\bbb$-good} (or simply \emph{good} if $\bbb = IR_\partial(V)$), if it admits an orthogonal (not necessarily orthonormal) basis, and
\begin{itemize}
\item[(i)] when $\bbb < \omega$ (which happens when $\bbb = IR_\partial(V)$ for $V$ visible), we have $|\partial|_V \leq \omega (\bbb|u|)^{-1}$;

\item[(ii)] when $\bbb \in [\omega,1)$ and hence $p>0$ (which happens when $\bbb = IR_\partial(V)$ for $V$ non-visible), we have
\begin{equation}\label{E:good-norm}
\Big|\frac{\partial^i}{i!} \Big|_V \leq |\partial|^i_K, \textrm{ for }i = 1, \dots, r-1, \quad
\Big|\frac{\partial^r} {r!} \Big|_V \leq p^{-1/(p-1)}(\bbb|u|)^{-r}; \textrm{ and}
\end{equation}

\item[(iii)] when $\bbb = 1$, we have $|\partial^i/i!|_V \leq |\partial|_K^i$ for all $i \geq 0$.
\end{itemize}

One may summarize the conditions (i)--(iii) by writing 
\[
\textrm{(iv)}\quad \big|\partial^i / i! \big|_V \leq \max\big\{ |\partial|^i_K, (\omega \bbb^{-1}|u|^{-1})^i / |i!|\big\}  \textrm{ for }i = 1, \dots, r. 
\]
Indeed, the equivalence of (1) or (iii) with (iv) is straightforward and the equivalence of (ii) and (iv) (when necessarily $p>0$) follows from the observation that the maximum above is equal to  $|\partial|^i_K$ if $i<r$ and to $p^{-1/(p-1)}(\bbb |u|)^{-r}$ if $i=r$.  From condition (iv), it is obvious that a $\bbb$-good norm is also $\bbb'$-good for any $\bbb' \leq \bbb$.

For the rest of this definition, we assume that $\bbb = IR_\partial(V)<1$.
By Lemma~\ref{L:good-norm-exists} below there exists a good norm for $V$.

Using this good norm, we define the multiset of \emph{refined $\partial$-radii} of $V$, denoted by $\Theta_\partial(V)$, as follows.  Enlarge the value group of $K$ in the sense of Remark~\ref{R:enlarge-K} so that $V$ admits an orthonormal basis.  Let $N_r$ be the matrix of $\partial^r$ with respect to the chosen basis.  If $\alpha_1, \dots, \alpha_d$ are the reduced eigenvalues of $N_r$, viewed as elements in $\kappa_{K^\alg}^{(rs)}$, we put $\Theta_\partial(V, |\cdot|) =\{ \alpha_1^{1/r}, \dots, \alpha_d^{1/r} \}$ as the multiset consisting of elements in $\kappa_{K^\alg}^{(s)}$ (note that there is no ambiguity of taking $r$-th roots for elements in $\kappa_{K^\alg}^{(rs)}$ when $p>0$).  
We will see in Lemmas~\ref{L:refined-well-defined1} and \ref{L:refined-well-defined2} that the multiset of refined $\partial$-radii is independent of the choices of the good norm and the orthonormal basis of $V$.  After these lemmas, we will abbreviate $\Theta_\partial(V, |\cdot|)$ to $\Theta_\partial(V)$.  When $\Theta_\partial(V)$ consists of $\dim V$ copies of a single element $\theta$, we say that $V$ has \emph{pure refined $\partial$-radii}.

We remark that $\Theta_\partial(V)$ does not depend on the choice of the rational parameter $u$.  But it is sometimes convenient to use the multiset of \emph{intrinsic refined $\partial$-radii} $\calI\Theta_\partial(V) = u \Theta_\partial(V)$ for a fixed rational parameter $u \in K$.

Finally, in the case when $IR_\partial(V) =1$, we conventionally define $\Theta_\partial(V)$ and $\calI\Theta_\partial(V)$ to be the multisets consisting of $0$ with multiplicity $\dim V$.
\end{definition}

\begin{remark}
\label{R:refined-in-K}
In the definition of refined $\partial$-radii, we first enlarged $K$ to $K'$, the completion of $K(x_1, \dots, x_n)$ for some $(\eta_1, \dots, \eta_n)$-Gauss norm.  However, the multiset of refined $\partial$-radii $\Theta_\partial(V, |\cdot|)$ is still composed of elements in $\kappa_{K^\alg}^{(s)}$.  Indeed,  since the construction is canonical, for any $\theta \in \Theta_\partial(V, |\cdot|)$, we have $g \theta \in \Theta_\partial(V, |\cdot|)$ for any automorphism $g$ of $K'$ fixing $K$.  But $\Theta_\partial(V, |\cdot|)$ is a finite multiset.  So it can consist only of elements in $\kappa_{K^\alg}^{(s)}$.  Alternatively, we can carefully keep track of the  new  variables we introduced in the computation of reduced eigenvalues; from this, we can also see that the multiset of refined $\partial$-radii is composed of elements in $\kappa_{K^\alg}^{(s)}$.
\end{remark}

\begin{remark}\label{R:refined-for-critical}
We also remark that when $p>0$ and $\bbb = \omega^{1/p^\lambda}$, the condition \eqref{E:good-norm} for $i = 1, \dots, p^{\lambda-1}$ is equivalent to \eqref{E:good-norm} for $i =1, \dots, p^\lambda$. But we need the matrix $N_{p^\lambda}$ to define refined $\partial$-radii.
For example, when $\bbb = IR_\partial(V) =\omega$,  we will see in Lemma~\ref{L:good-norm-exists} below that the twisted polynomial from Proposition~\ref{P:spec-norm-from-NP} gives us a good norm on $V$.  However, one \emph{cannot} compute the refined $\partial$-radii by taking the reduced roots of this twisted polynomial.  Instead, one has to find the matrix for $\partial^p$.
\end{remark}

\begin{remark}
For a good norm, one can show that the inequalities in \eqref{E:good-norm} are in fact equalities, but we will not use this fact later (see \cite[Lemma~6.2.4]{kedlaya-course} for a proof of similar flavor).
\end{remark}

\begin{lemma} \label{L:good-norm-exists}
Let $V$ be as in Hypothesis~\ref{H:refined}.  Assume that $\bbb \leq IR_\partial(V)$, and that $\bbb<1$ if $p>0$.  Then $V$ admits a $\bbb$-good norm.  In particular, any $V$ with pure intrinsic radius $IR_\partial(V)<1$ admits a good norm.
\end{lemma}
\begin{proof}
We first assume that $\bbb \leq \omega$.  We take a cyclic vector $\bbv \in V$ with twisted polynomial $P$.  By Proposition~\ref{P:spec-norm-from-NP}, the lesser of $-\log |\partial|_K$ and the least slope of the Newton polygon of $P$ equals $\min\{s, -\log |\partial|_K\} \geq  -\log(\omega \bbb^{-1}|u|^{-1})$.  Then we can define a $\bbb$-good norm on $V$ by taking the orthogonal basis to be $\bbv, \partial \bbv,  \dots,\partial^{d-1} \bbv$ with $|\partial^i \bbv| = \omega^i(\bbb|u|)^{-i}$ for $i = 0, \dots, d-1$.  When $\bbb = \omega$, as pointed out in Remark~\ref{R:refined-for-critical}, our bound $|\partial|_V \leq |u|^{-1}$ alone implies condition \eqref{E:good-norm} for $r=1, \dots, p$ when $p>0$, and condition (iii) in Definition~\ref{D:good-norm} when $p=0$.

The remainder case is when $p>0$ and $\bbb \in (p^{-1/(p-1)}, 1)$.  We let $n = \lambda-1$ if $\bbb = p^{-1/p^{\lambda-1}(p-1)}$ and $n=\lambda$ otherwise.  In other words, $n$ is the unique nonpositive integer such that $\bbb \in (p^{-1/p^{n-1}(p-1)}, p^{-1/p^{n}(p-1)}]$.  Let $\varphi^{(\partial, n)}: K^{(\partial, n)} \rar K$ be the $p^{n}$-th $\partial$-Frobenius and let $\tilde \partial = \partial / (p^{n}u^{p^{n}-1})$ be the corresponding derivation on $K^{(\partial, n)}$.  Since $IR_\partial(V) \geq \bbb > p^{-1/p^{n-1}(p-1)}$, by repeatedly applying Lemma~\ref{L:frob-properties}(d), we obtain an $n$-fold $\partial$-Frobenius antecedent $W$ over $K^{(\partial, n)}$; it has intrinsic $\tilde \partial$-radii $IR_{\tilde \partial}(W) = IR_\partial(V)^{p^{n}} \geq \bbb^{p^n} \in (p^{-p/(p-1)}, p^{-1/(p-1)}]$.  In particular, $W$ has a $\bbb^{p^n}$-good norm by the argument in previous paragraph.  We have 
\begin{align}
\nonumber
\big|u^{p^{n}}\tilde \partial \big|_W & \leq p^{-1/(p-1)} \bbb^{-p^n} \in [1,p) \\
\label{E:good-norm-exists}
\Rar \big| u\partial \big|_W & = p^{-n} \big|u^{p^{n}}\tilde \partial \big|_W
\left\{
\begin{array}{ll}
< p^{-n} \cdot p  = p^{\lambda-1} & \textrm{when } n = \lambda, \\
\leq p^{-n} \cdot 1 = p^{\lambda - 1} & \textrm{when }n = \lambda-1.
\end{array} \right.
\end{align}
This norm on $W$ gives rise to a $K$-norm $|\cdot|_V$ on $V$, which we will show is $\bbb$-good.  By \eqref{E:good-norm-exists}, we have $|u\partial- i|_V = |u \partial-i|_W \leq |i|$ for $i = 1, \dots, p^\lambda -1$.  Hence we have, for $i = 1, \dots, p^\lambda$,
\begin{align*}
&\Big|\frac{u^i\partial^i}{i!} \Big|_V = 
\Big|\frac{u^i\partial^i}{i!} \Big|_W =
\Big| \frac{u\partial(u\partial - 1) \cdots (u\partial -(i-1))}{i!}\Big|_W \leq \Big|\frac{u\partial} i \Big|_W \\
=\ &  \Big|\frac{p^n}{i} u^{p^{n}}\tilde \partial \Big|_W 
\left\{\begin{array}{ll} 
\leq 1 & \textrm{if }i =1, \dots, p^\lambda - 1, \\
\leq p^{-1/(p-1)} \bbb^{-p^n} = p^{-1/(p-1)} \bbb^{-p^\lambda} & \textrm{if }i = p^\lambda \textrm{ and }n =\lambda,
\\
\leq p^{-p/(p-1)} \bbb^{-p^n} = p^{-1} & \textrm{if }i = p^\lambda \textrm{ and }n =\lambda-1.
\end{array}\right.
\end{align*}
This verifies \eqref{E:good-norm}.
\end{proof}

\begin{lemma}\label{L:refined-well-defined1}
Assume that $IR_\partial(V)<1$.
Let $|\cdot|$ be a good norm on $V$.  Then the multiset of refined $\partial$-radii $\Theta_\partial(V, |\cdot|)$ is well-defined.
\end{lemma}
\begin{proof}
By possibly enlarging $K$ in the sense of Remark~\ref{R:enlarge-K}, we have two orthonormal bases $\underline \bbe$ and $\underline \bbe'$ for $|\cdot|_V$ such that $\underline \bbe' = \underline \bbe A$ for a transition matrix $A \in \GL_d(\calO_K)$.  For $i= 1, \dots, r$, let $N_i$ denote the matrix of $\partial^i$ with respect to $\underline \bbe$; by \eqref{E:good-norm}, we have $|N_i / i!|\leq |\partial|^i_K$ for $i=1, \dots, r-1$.
Then we have
\[
\frac{ \partial^r(\underline \bbe')}{r!} = 
\frac{ \partial^r(\underline \bbe A)}{r!} = 
\sum_{i=0}^r \frac{\partial^i(\underline \bbe)}{i!} \frac {\partial^{r-i}(A)}{(r-i)!} =
\underline \bbe' A^{-1} \Big( \sum_{i=0}^{r}\frac{N_i}{i!}\frac{\partial^{r-i}(A)}{(r-i)!} \Big)
\]
If $A^{-1}MA$ denote the matrix of $\partial^r / r!$ with respect to $\underline \bbe'$, we have
\[
M = \frac{N_r}{r!} + \sum_{i=0}^{r-1} \frac{N_i}{i!} \frac{\partial^{r-i}(A)A^{-1}}{(r-i)!}.
\]
Note that $|N_i/ i!| \leq |\partial|_K^i$ and $| \partial^{r-i}(A) A^{-1}/ (r-i)!| \leq |\partial|^{r-i}_K|A||A^{-1}| \leq |\partial|^{r-i}_K$ imply that $|M - N_r/r!| \leq |\partial|^r_K < \omega R_\partial(V)^{-r}$, which is smaller than any singular value of $N_r/r!$.  By \cite[Theorem~4.2.2]{kedlaya-course}, the reduced eigenvalues of $N_r/r!$ coincide with those of $A^{-1}  MA$.  Therefore, $\Theta_\partial(V)$ does not depend on the choice of good norms $|\cdot | $ on $V$.
\end{proof}

\begin{lemma} \label{L:refined-well-defined2}
Assume that $IR_\partial(V)<1$. Let $|\cdot|_1$ and $|\cdot|_2$ be two good norms on $V$.  Then we have $\Theta_\partial(V, |\cdot|_1) = \Theta_\partial(V, |\cdot|_2)$.
\end{lemma}
\begin{proof}
By possibly enlarging $K$ as in Remark~\ref{R:enlarge-K}, we may choose orthonormal bases $\underline \bbe$ and $\underline \bbf$ of $|\cdot|_1$ and $|\cdot|_2$, respectively, so that $\underline \bbe A= \underline \bbf $ with $A = \Diag\{a_{11}, \dots, a_{dd}\}$.

Let $N_i$ denote the matrix of $\partial^i$ with respect to $\underline \bbe$; by \eqref{E:good-norm}, we have $|N_i/i!|\leq 1$ for $i=1, \dots, r-1$.  Then we have
\[
\frac{\partial^r (\underline \bbf)}{r!} = 
\frac{\partial^r(\underline \bbe A)}{r!} = 
\sum_{i=0}^r \frac{\partial^i(\underline \bbe)}{i!} \frac {\partial^{r-i}(A)}{(r-i)!} =
\underline \bbf A^{-1} \Big( \sum_{i=0}^{r}\frac{N_i}{i!}\frac{\partial^{r-i}(A)}{(r-i)!}A^{-1} \Big) A.
\]
It suffices to show that $N_r/r!$ has the same reduced eigenvalues as $\sum_{i=0}^r\frac{N_i}{i!}\frac{\partial^{r-i}(A)}{(r-i)!}A^{-1}$.  This is true by \cite[Theorem~4.4.2]{kedlaya-course} since 
\[
\Big|\frac{N_i}{i!}\frac{\partial^{r-i}(A)}{(r-i)!}A^{-1}\Big|
 = \Big|\frac{N_i}{i!}\Big| \cdot \Big|\Diag\Big(
\frac{\partial^{r-i}(a_{11})}{(r-i)!}a_{11}^{-1},
 \dots, \frac{\partial^{r-i}(a_{dd})}{(r-i)!}a_{dd}^{-1}
 \Big)\Big| \leq |\partial|_K^i \cdot |\partial_K|^{r-i} < \omega R_\partial(V)^{-1},
\]
for $i = 0, \dots, r-1$.
\end{proof}

\begin{corollary}\label{C:refined-radii=reduced-roots}
Assume that $V$ has pure visible $\partial$-radii.  For \emph{any} cyclic vector $\bbv \in V$, the multiset of the reduced roots of the twisted polynomial associated to $\bbv$ is exactly the multiset of refined $\partial$-radii of $V$.  In particular, this multiset is composed of nonzero elements of $\kappa_{K^\alg}^{(s)}$.

More generally, we may drop the hypothesis that $V$ has pure $\partial$-radii, and only assume that $V$ has visible $\partial$-radii $R_\partial(V) = \omega \ee^s$.  Let $h$ denote the multiplicity of $R_\partial(V)$ in the multiset $\calR_\partial(V)$.  In this case, for any cyclic vector $\bbv \in V$, if we write the associated monic twisted polynomial as $X^d + a_1X^{d-1} + \cdots + a_d$, then $|a_i|\leq \ee^{-is}$ for $i \leq h$ and $|a_h| = \ee^{-ih}$.  Moreover, if $V_{\omega \ee^s}$ is the direct summand of $V$ with pure $\partial$-radii $\omega \ee^s$ as given by Theorem~\ref{T:decomp-over-field-complete}, then $\Theta_\partial(V_{\omega \ee^s})$ consists of the reduced roots of the polynomial $X^h + a_1X^{h-1} + \cdots + a_h = 0$.
\end{corollary}
\begin{proof}
As already pointed out in Remark~\ref{R:refined-for-critical}, we emphasize again that the case $IR_\partial(V) = \omega$ is not included in the statement.
We first treat the case when $V$ has pure visible $\partial$-radii.
We can construct the good norm using the twisted polynomial as in Lemma~\ref{L:good-norm-exists}.  This twisted polynomial is then exactly the characteristic polynomial of the matrix of $\partial$ with respect to this basis.  The claim follows.

For $V$ not necessarily pure of $\partial$-radii, the bound for norms on $a_i$ follows from Proposition~\ref{P:spec-norm-from-NP}.  For the statement about refined $\partial$-radii, we need to dig into the proof of Theorem~\ref{T:decomp-over-field-complete} a bit more.  By \cite[Corollary~3.2.4]{kedlaya-semistable3}, we can write $P = QR$ where $Q$ and $R$ are monic twisted polynomials such that the Newton polygon of $Q = X^h + a'_1X^{h-1} + \cdots + a'_h$ has pure slopes $s$ and the Newton polygon of $R$ has slope strictly bigger than $s$.  Moreover, we have $V_{\omega \ee^s} = K\{T\} / QK\{T\}$.  The upshot is that the formal multiplication satisfies $|a_i- a'_i| < \ee^{is}$ for any $i = 1, \dots, h$. Therefore, the reduced roots of $X^h + a_1X^{h-1} + \cdots + a_h = 0$ are the same as the reduced roots of $X^h + a'_1X^{h-1} + \cdots + a'_h = 0$, which are the same as the elements of $\Theta_\partial(V)$ by the discussion in the previous paragraph.
\end{proof}

\begin{lemma}
\label{L:converse-good-norm}
Fix $\bbb \in (0,1)$ and set $r = r(\bbb), \lambda = \lambda(\bbb)$, and $s = -\log(\omega (\bbb|u|)^{-1})$.
Let $V'$ be a $\partial$-differential module of rank $d$, equipped with a basis $\underline \bbe$, with respect to which the action of $\partial$ satisfies the conditions in Definition~\ref{D:good-norm} with the chosen $\bbb$.  Assume that the reduced eigenvalues of the matrix $N_r \in \Mat(\gothm_K^{(rs)})$ of $\partial^r$ on $V'$ are all nonzero in $\kappa_{K^\alg}^{(rs)}$. Then $V'$ has pure intrinsic $\partial$-radii $\bbb$.  As a consequence,  $\Theta_\partial(V')$ is exactly the multiset of the reduced eigenvalues of $N$.
\end{lemma}
\begin{proof}
Since $N_r \in \Mat(\gothm_K^{(rs)})$, we have $IR_\partial(V') \geq \bbb$.  Suppose that $V'$ does not have pure intrinsic  $\partial$-radii $\bbb$. Enlarging $K$ as in Remark~\ref{R:enlarge-K} if needed, we may apply Theorem~\ref{T:decomp-over-field-complete} and Lemma~\ref{L:good-norm-exists} to $V'$ and its Jordan-H\"older constituents to find a basis $\underline \bbf$ for which the conditions in Definition~\ref{D:good-norm} hold and the matrix  $\widetilde N_{r} \in \Mat(\gothm_K^{(rs)})$ of $\partial^r$ is degenerate modulo $\gothm_{K^\alg}^{(rs)+}$ (when identifying $\kappa_K^{(rs)}$ with $\kappa_K$).
Now, the same argument in  the proof of Lemma~\ref{L:refined-well-defined2} implies that $N_r$ and $\widetilde N_r$ must have the same multiset of reduced eigenvalues.  But zero is a reduced eigenvalue of $\widetilde N_r$ but not of $N_r$, which is a contradiction. Hence $V'$ has pure intrinsic $\partial$-radii $\bbb$.  The last statement follows from Definition~\ref{D:good-norm}.
\end{proof}

\begin{lemma}\label{L:refined-dual}
We have $\Theta_\partial(V^\dual)  = -\Theta_\partial(V) = \{-\theta\; | \; \theta \in \Theta_\partial(V)\}$.
\end{lemma}
\begin{proof}
This is straightforward.
\end{proof}

We will prove in Theorem~\ref{T:refined-decomposition} a direct sum decomposition of $V$ parametrized by the multiset of refined $\partial$-radii.  For this, we
start with some basic properties of refined $\partial$-radii when $V$ has visible $\partial$-radii.

\begin{lemma}\label{L:interpretation-refined}
Let $V$ and $W$ be two $\partial$-differential modules over $K$ with pure and \emph{visible} $\partial$-radii $R_\partial(V) = R_\partial(W)$.  Then the following two statements are equivalent.
\begin{itemize}
\item[\emph{(1)}] The refined $\partial$-radii of $V$ and $W$ are distinct, i.e., $\Theta_\partial(V) \cap \Theta_\partial(W) = \emptyset$.
\item[\emph{(2)}] The tensor product $V \otimes W^\dual$ has pure $\partial$-radii $R_\partial(V)$.
\end{itemize}
Moreover, if either statement holds, we have an equality of multisets: $\Theta_\partial(V \otimes W^\dual) = \{\theta_1 - \theta_2 | \theta_1 \in \Theta_\partial(V), \theta_2 \in \Theta_\partial(W)\}$.
As a corollary, we have the following
\begin{itemize}
\item[\emph{(a)}] If $\Theta_\partial(V) \cap \Theta_\partial(W) = \emptyset$, then any homomorphism $f: W\rar V$ of $\partial$-differential modules is zero.
\item[\emph{(b)}] If $\Theta_\partial(W)$ has pure refined $\partial$-radii $\theta \in \kappa_{K^\alg}^{(s)}$, then $\theta \in \Theta_\partial(V)$ if and only if $V \otimes W^\dual$ does \emph{not} have pure $\partial$-radii $R_\partial(V)$.
\item[\emph{(c)}] If $\Theta_\partial(V)$ and $\Theta_\partial(W)$ both have the same pure $\partial$-radii $\theta \in \kappa_{K^\alg}^{(s)}$, then we have $R_\partial(V \otimes W^\dual) > R_\partial(V)$.
\end{itemize}
\end{lemma}
\begin{proof}
By Lemma~\ref{L:refined-dual}, $\Theta_\partial(W^\dual) = -\Theta_\partial(W)$.
We may enlarge $K$ as in Remark~\ref{R:enlarge-K} so that we have good norms on both $V$ and $W^\dual$ given by orthonormal bases.  Equip $V \otimes W^\dual$ with the tensor product norm.  Let $N_0, N_1 \in \Mat(\gothm_{K^\alg}^{(s)})$ be the matrices of $\partial$ acting on $V$ and $W^\dual$ with respect to the given bases, respectively.  By Definition~\ref{D:good-norm}, $\Theta_\partial(V)$ and $-\Theta_\partial(W)$ are the multisets of reduced eigenvalues of $N_0$ and $ N_1$, respectively.  Then the multiset of  reduced eigenvalues of the matrix $N = N_0 \otimes 1 + 1 \otimes N_1$ is exactly $\{\theta_1 - \theta_2 | \theta_1 \in \Theta_\partial(V), \theta_2 \in \Theta_\partial(W)\}$.  

If (1) holds, then all reduced eigenvalues of $N$ are nonzero and hence $|N^n| = \ee^{-ns}$ for all $n \in \NN$.  Moreover, the reduction of $N^n$ in $M_d(\kappa_{K^\alg}^{(ns)})$ has full rank if we identify $\kappa_{K^\alg}^{(ns)}$ with $\kappa_{K^\alg}$.  Therefore, $V \otimes W^\dual$ has pure $\partial$-radii $R_\partial(V)$ by Lemma~\ref{L:converse-good-norm}, proving (2).

Conversely, if (2) holds, then the tensor product norm is a good norm on $V \otimes W^\dual$ already and the multiset of reduced eigenvalues of $N$ is the multiset of refined $\partial$-radii of $V \otimes W^\dual$.  By Corollary~\ref{C:refined-radii=reduced-roots}, $0 \notin \Theta_\partial(V \otimes W^\dual)$.  This implies (1).

We now prove (a).  Since $V \otimes W^\dual$ has pure $\partial$-radii $R_\partial(V)< \omega$, we have $H_\partial^0(V \otimes W^\dual) = 0$ and hence there is no nonzero homomorphism of $\partial$-differential modules from $W$ to $V$.

The statement (b) is just (a special case of) the inverse statement of $(1)\LRar(2)$.

For (c), we know that $N_0$ and $N_1$ have pure reduced eigenvalues $\theta$ and $-\theta$, respectively.  Hence $N = N_0 \otimes 1+ 1 \otimes N_1$ reduces to a matrix in $\kappa_{K^\alg}^{(s)}$ with zero eigenvalues (if we identify $\kappa_{K^\alg}^{(s)}$ with $\kappa_{K^\alg}$).  It is then nilpotent, i.e.,  $N^{n} \in \Mat \big(\gothm_{K^\alg}^{(ns)+} \big)$ for $n \geq \dim V \cdot \dim W$.  This implies that $R_\partial(V \otimes W^\dual) > R_\partial(V)$.
\end{proof}

\begin{lemma}\label{L:V-otimes-W-not-pure}
Let $V$ and $W$ be two $\partial$-differential modules over $K$.  Assume that $V$ has pure and visible $\partial$-radii and $R_\partial(V) < R_\partial(W)$.  Then $V \otimes W^\dual$ has pure $\partial$-radii $R_\partial(V)$ and multiset of refined $\partial$-radii is composed of $\dim W$ copies of  $\Theta_\partial(V \otimes W^\dual)$.
\end{lemma}
\begin{proof}
By Theorem~\ref{T:decomp-over-field-complete}, we may assume that $W$ has pure $\partial$-radii.  By Lemma~\ref{L:good-norm-exists} we may find a $\bbb$-good norm on $W$ with $\bbb =\min\{IR_\partial(W), \omega\}> IR_\partial(V)$.

We proceed as in Lemma~\ref{L:interpretation-refined}.  If $N_0$ and $ N_1$ are the matrices of $\partial$ with respect to some orthonormal bases of $V$ and $W^\dual$, respectively, then we have $N_1 \in \Mat(\gothm_K^{(s)+})$ and that $N_0$ has pure reduced eigenvalue $\Theta_\partial(V)$.  Hence the multiset of reduced eigenvalues of $N_0 \otimes 1 + 1 \otimes N_1$ is simply composed of $\dim W$ copies of the set of reduced eigenvalues of $N_0$. The lemma follows.
\end{proof}

The refined $\partial$-radii of a non-visible $\partial$-differential module is closely related to the $\partial'$-radii of its Frobenius antecedent; this fact would allow us to reduce many computation to the visible case.  To establish this relation explicitly, it is more convenient to work with the refined \emph{intrinsic} $\partial$-radii.

\begin{proposition}\label{P:refined-frob}
Assume $p>0$.
Let $\varphi^{(\partial)}: K^{(\partial)} \rar K$ be the $\partial$-Frobenius with respect to the parameter $u$.
\begin{itemize}
\item[\emph{(a)}] Assume that $IR_\partial(V) \in(p^{-1/(p-1)}, 1)$, and then Lemma~\ref{L:frob-properties}(d) implies that $V = \varphi^{(\partial)*}W$ for some $\partial'$-differential module $W$ on $K^{(\partial)}$ such that $IR_{\partial'}(W) = IR_\partial(V)^p$.  We have
\[
\Theta_\partial(V) = \big\{ -(p\theta')^{1/p}\; \big|\; \theta' \in\Theta_{\partial'}(W) \big\}.
\]
\item[\emph{(b)}] Assume that $IR_\partial(V) = p^{-1/(p-1)}$, and then $\varphi^{(\partial)}_*(V)$ has pure intrinsic $\partial'$-radii $p^{-p/(p-1)}$. The elements in $\calI\Theta_{\partial'}(\varphi^{(\partial)}_*(V))$ can be grouped into $p$-tuples $\big(\frac\theta p, \frac{\theta+1}p, \dots, \frac{\theta + p-1}p\big)$ with $\theta \in \kappa_{K^\alg}$, and $\calI \Theta_\partial(V)$ is the multiset composed of
$(\theta^p - \theta)^{1/p} \in \kappa_{K^\alg}$
for each $p$-tuple above.
\item[\emph{(c)}] Assume $IR_{\partial}(V)< p^{-1/(p-1)}$. Then we have $\calI\Theta_{\partial'}(\varphi^{(\partial)}_*V) = \big\{\underbrace{p^{-1}\theta,\dots, p^{-1}\theta}_{p\textrm{ times}} \big| \theta \in \calI\Theta_\partial(V) \big\}$.
\end{itemize}
\end{proposition}
\begin{proof}
(a) By Lemma~\ref{L:good-norm-exists} and by possibly enlarging $K$ in the sense of Remark~\ref{R:enlarge-K}, we can take an orthonormal basis $\underline \bbe$ on $W$ which defines a good norm.  The norm induces a good norm on $V$ by the explicit construction in Lemma~\ref{L:good-norm-exists}.  Let $\lambda$ and $r$ be as in Notation~\ref{N:dichotomy}.  We have
\begin{align*}
u^{p^\lambda}\partial^{p^\lambda} &= u\partial(u\partial - 1) \cdots (u\partial - p^\lambda+1) = pu^p\partial' (pu^p\partial' -1) \cdots (pu^p \partial' - p^\lambda+1) \\
&= p^{p^{\lambda-1}} u^{p^\lambda}\partial'^{p^{\lambda-1}} \prod_{i = 1, p\nmid i}^{p^\lambda-1} (pu^p\partial' -i);
\end{align*}
this operator also acts on $W$.
Since $|u^p\partial'|_W \leq \max\{1, p^{-1/(p-1)} IR_{\partial'}(W)\} < p$, we have
\[
\big| u^{p^\lambda}\partial^{p^n} - p^{p^{\lambda-1}} \big((-1) \cdots (-p+1)\big)^{p^{\lambda-1}} u^{p^\lambda}\partial'^{p^\lambda}\big|_W < \big| u^{p^n\lambda}\partial^{p^\lambda} \big|_W.
\]
Therefore, the matrix of $\partial^{p^\lambda}$ with respect to $\underline \bbe$ is congruent to the matrix of $(-1)^{p^{\lambda-1}(p-1)}(p!)^{p^{\lambda-1}}\partial'^{p^{\lambda-1}}$ modulo $\gothm_K^{(p^\lambda s)+}$.  We then must have 
\[
\Theta_\partial(V, |\cdot|) = \big\{ \big((-1)^{(p-1)}(p!)\theta \big)^{1/p} \big| \theta \in\Theta_{\partial'}(W) \big\} = \big\{ -\big(p \theta \big)^{1/p}\big| \theta \in\Theta_{\partial'}(W) \Big\}
.\]

(b)
When $IR_\partial(V) = p^{-1/(p-1)}$,  Lemma~\ref{L:frob-properties}(e) implies that $\varphi^{(\partial)}_* V$ has pure intrinsic $\partial'$-radii $p^{-p/(p-1)}$.  By Lemma~\ref{L:frob-properties}(e) and Lemma~\ref{L:interpretation-refined}, the elements in $\calI\Theta_{\partial'}(\varphi^{(\partial)}_*(V))$ can be grouped into $p$-tuples $\big(\frac\theta p, \frac{\theta+1}p, \dots, \frac{\theta + p-1}p\big)$ with $\theta \in \kappa_{K^\alg}$.  By possibly enlarging $K$ in the sense of Remark~\ref{R:enlarge-K}, we may assume that $\varphi^{(\partial)}_*V$ admits a good norm defined by an orthonormal basis $\underline \bbe$.  Let $N$ be the matrix of $pu^p\partial'$ with respect to $\underline \bbe$.  Then $u^p\partial^p$ acts on $\varphi^{(\partial)*}\varphi^{(\partial)}_*(V) = V^{\oplus p}$ as per description
\[
u\partial(u\partial-1) \cdots (u\partial-p+1) = pu^p\partial'(pu^p\partial' -1) \cdots (pu^p\partial' - p+1).
\]
Hence the matrix for this action with respect to $\underline \bbe$ is congruence to $N(N-1) \cdots (N-p+1)$ modulo $p \calO_{K^{(\partial)}}$ since $|pu^p\partial'|_{K^{(\partial)}} = p^{-1}$; then the multiset of its reduced eigenvalues is composed of $\theta^p - \theta$ with multiplicity $p$ for each tuple $(\frac \theta p, \frac {\theta+1}p, \dots, \frac{\theta+p-1}p)$ in the multiset of reduced eigenvalues of $N$.  The statement follows.

(c) By Lemma~\ref{L:frob-properties}(e), $\varphi^{(\partial)}_*V$ has pure intrinsic $\partial'$-radii $p^{-1}IR_\partial(V) \leq p^{-p/(p-1)}$.  Since $u^p \partial' = u\partial/p$, we can take a good norm of $\varphi^{(\partial)}_*V$ and deduce that $\calI\Theta_{\partial'}(\varphi^{(\partial)}_*V)
= \frac 1p\calI\Theta_\partial \big(\varphi^{(\partial)*} \varphi^{(\partial)}_*V\big)$, which in turn equals $\frac 1p\calI\Theta_\partial(V^{\oplus p})$ by Lemma~\ref{L:frob-properties}(b).  The statement follows.
\end{proof}

\begin{proposition}\label{P:refined-properties}
Lemma~\ref{L:interpretation-refined} holds with only assuming $IR_\partial(V) =IR_\partial(W)<1$ instead of the visible hypothesis. Similarly, Lemma~\ref{L:V-otimes-W-not-pure} holds with only assuming $IR_\partial(V) <1$ instead of the visible hypothesis on $V$.
\end{proposition}
\begin{proof}
It suffices to check the remaining cases: $p>0$ and $IR_\partial(V) \geq p^{-1/(p-1)}$.  
If $IR_\partial(V) > p^{-1/(p-1)}$, the statements for $V$ and $W$ follow from the statements on their $\partial$-Frobenius antecedents by Proposition~\ref{P:refined-frob}(a).  If $IR_\partial(V) = p^{-1/(p-1)}$, the statements for $V$ and $W$ follow from the statements on their $\partial$-Frobenius descendants by Proposition~\ref{P:refined-frob}(b) and Lemma~\ref{L:frob-properties}(c).
\end{proof}

The following is an example of $\partial$-differential modules with pure refined $\partial$-radii.  It will serve as a testing object later.  

Our convention is to use Gothic letter $\goths$ instead of $s$ when discussing intrinsic radii of convergence; we will never use both $s$ and $\goths$ together.

\begin{example}\label{Ex:pure-refined}
Fix $\goths \in -\log |K^\times|^\QQ$ such that $\goths<0$ if $p=0$, and $\goths<\frac 1p \log p$ if $p>0$.  Let $\theta \in \kappa_{K^\alg}^{(\goths)}$ be a nonzero element.
\begin{itemize}
\item[(1)]  If $p=0$, then we have $\goths \in -\log |(K')^\times|$ and $\theta \in \kappa_{K'}^{(\goths)}$ for some finite \emph{tamely ramified} extension $K'$ of $K$.  Let $x$ be a lift of $\theta$ to $ \gothm_{K'}^{(\goths)} $. Put $d=1$ and $n=0$.
\item[(2)] If $p>0$, there exists $n\in\NN$ such that $\theta^{p^n} \in (\kappa_{K'}^{(p^{n-1}\goths)})^p$ with $p^{n-1}\goths \in -\log|(K')^\times|$ for some finite \emph{tamely ramified} extension $K'$ of $K$.  By Lemma~\ref{L:Frob-does-to-kappa}, we may find a lift
$x \in u^{-p^n}\gothm_{K'^{(\partial)}}^{(p^n\goths)}$ of $u^{-p^n}\theta^{p^n}$, where the extra $u^{-p^n}$ reflects the different normalizations of refined intrinsic $\partial$-radii and refined $\partial$-radii. Put $d=p^n$.
\end{itemize}

Let $\calL_{x, (n)}$ denote  the  $\partial$-differential module over $K'$ of rank $d$ with basis $\{\bbe_1, \dots, \bbe_d\}$, where the $\partial$-action is given by $\partial \bbe_i = \bbe_{i+1}$ for $i = 1, \dots, d-1$ and $\partial \bbe_d = x\bbe_1$.
\end{example}

\begin{remark}\label{R:pure-refined-local}
When $p>0$, we point out that $\mathfrak s < \frac 1p \log p$ also includes some part of the non-visible range.
The restriction $\goths <\frac 1p \log p$ in Example~\ref{Ex:pure-refined} is linked with the choice $x \in u^{-p^n}\gothm_{K'^{(\partial)}}^{(p^n\goths)}$.  In general, we may extend the range of $\goths$ to be $(-\infty, (\frac 1{p-1} - \frac 1{p^c(p-1)})\log p)$ for some $c \in \NN$, but the price we pay is to take $x \in u^{-p^n}\gothm_{K'^{(\partial, c)}}^{(p^n\goths)}$ lifting $u^{-p^n}\theta^{p^n}$ for some $n \in\NN$ and some finite tamely ramified extension $K'$ of $K$.  However, as $c$ gets larger, we need to enlarge $n$ to guarantee the existence of such a lift $x$.  This is why we may not assume that $\goths<\frac 1{p-1}\log p$.
\end{remark}

\begin{remark}
In the non-visible case, one can construct a $\partial$-differential module with pure refined  $\partial$-radii by simply pulling back a $\partial'$-differential module over $K^{(\partial)}$ with appropriate refined  $\partial'$-radii.  However, the action of $\partial$ is not in a form that works nicely when we study the one-dimensional variation of refined  $\partial$-radii later.  We will construct Example~\ref{Ex:pure-IR-1-dim}, a family version of Example~\ref{Ex:pure-refined}, which looks similar in both visible and \emph{non-visible} ranges.
\end{remark}

\begin{lemma}\label{L:pure-refined}
Keep the notation as in Example~\ref{Ex:pure-refined}.  Then $\calL_{x, (n)}$ has pure intrinsic $\partial$-radii $IR_\partial(\calL_{x, (n)}) = \omega \ee^\goths$ and pure refined intrinsic $\partial$-radii $\theta$.
\end{lemma}
\begin{proof}
We may replace $K$ by the completion of $K(z)$ with respect to the $|u|^{-1}\ee^{-\goths}$-Gauss norm (and set $\partial z = 0$).

We first assume that either we have $p=0$ or we have $p>0$ and $\goths <0$, i.e. we consider the visible $\partial$-radii case.  We note that $\bbe_1, z^{-1}\bbe_2, \dots,$ and $ z^{-(d-1)}\bbe_d$ together define a good norm on $\calL_{x, (n)}$; it is a straightforward computation to check that the statement in this case.

We now tackle the case when $p>0$ and $\goths \in [0, \frac 1p \log p)$.  For $i = 1, \dots, p$, we have
\[
\partial^i \bbe_l = \bbe_{i+l},  \textrm{ when }i+l \leq p^n,\ \textrm{ and }\
\partial^i \bbe_{p^n-l} = \partial^{i-l}(x\bbe_1), \textrm{ when } i\geq l.
\]
We will show that $\{\bbe_1, z^{-1}\bbe_2, \dots, z^{-(p^n-1)}\bbe_{p^n}\}$ defines a good norm on $\calL_{x, (n)}$.  Indeed, for $i =1, \dots, p$, the matrix of $\partial^i$ with respect to this basis is
\begin{equation}\label{E:N_l}
N_i = 
\begin{pmatrix}
0 & 0 & \cdots & z^i & 0 & \cdots & 0\\
0 & 0& \cdots & 0 & z^i & \cdots & 0 \\
\vdots &\vdots & \ddots & \vdots & \vdots & \ddots & \vdots \\
0 & 0& \cdots & 0 &0 & \cdots & z^i\\
z^{-p^n+i}x & 0 & \cdots &0 & 0 &\cdots &0\\
z^{-p^n+i}\partial x & z^{-p^n+i} x & \cdots &0 & 0 &\cdots &0\\
\vdots &\vdots & \ddots & \vdots & \vdots & \ddots & \vdots \\
z^{-p^n+i}\partial^{i-1} x & (i-1)z^{-p^n+i} \partial^{i-2} x & \cdots &0 & 0 &\cdots &0
\end{pmatrix}
\end{equation}
Note that \[
|\partial|_{K^{(\partial)}} = p^{-1}|u|^{p-1} |\partial'|_{K^{(\partial)}} = p^{-1}|u|^{-1} \leq \omega |z| < |z|.
\]
Hence, modulo $\gothm_K^{(-\log|z|)+}$, the nonzero terms of $N_i$ are the $z^i$'s and $z^{-p^n+i}x$'s in \eqref{E:N_l}; they form a 2-by-2 block matrix
\[
\overline{N_i}^{(-\log|z|)} =
\begin{pmatrix}
0 & z^i \cdot I_{(p^n-i) \times (p^n-i)} \\
z^{-p^n+i}x \cdot I_{i \times i} &0
\end{pmatrix} \in \Mat_{p^n \times p^n}(\kappa_K^{(-\log|z|)}).
\]
Note that $|z^{-p^n+i}x| = |z|^i$.  By Lemma~\ref{L:converse-good-norm}, we have $IR_\partial(\calL_{x, (n)}) = \omega \ee^\goths$ and that this basis defines a good norm on $V$. Moreover, the multiset of reduced eigenvalues of $N_p$ is composed of the element $x^{1/p^{n-1}}$ with multiplicity $p^n$-.  This implies that $\calI\Theta_\partial(V) = \{\theta$ $(p^n$ times$)\}$ by the choice of $x$ in Example~\ref{Ex:pure-refined}.
\end{proof}

\begin{lemma}\label{L:refined-subquotient}
Let $V$ be a $\partial$-differential module over $K$ with pure \emph{visible} $\partial$-radii $R_\partial(V) = \omega \ee^s$.  Then we have the following.

(a) For any subquotient $V_0$ of $V$, the elements in $\Theta_\partial(V_0)$ already appears in $\Theta_\partial(V)$.  

(b) For any $\theta \in \kappa_{K^\alg}^{(s)}$, there is a \emph{unique} maximal $\partial$-differential submodule of $V$ which has pure refined $\partial$-radii $\theta$.
\end{lemma}
\begin{proof}
For $\theta \in \kappa_{K^\alg}^{(s)}$ such that $\theta \notin \Theta_\partial(V)$, let $\calL_{x, (n)}$ be the $\partial$-differential module constructed in Example~\ref{Ex:pure-refined}.  By Lemmas~\ref{L:pure-refined} and \ref{L:interpretation-refined}, $V \otimes \calL_{x, (n)}^\dual$ has pure $\partial$-radii $R_\partial(V)$, and so does $V_0 \otimes \calL_{x, (n)}^\dual$. By the same lemmas again, we have $\theta \notin \Theta_\partial(V_0)$. This proves (a).  We point out that this, however, does not prove the inclusion $\Theta_\partial(V_0) \subseteq \Theta_\partial(V)$ as a multiset, which will be a corollary of Theorem~\ref{T:refined-decomposition} below.

The second statement follows from the observation that if two submodules $V_1$ and $V_2$ of $V$ both have pure refined $\partial$-radii $\theta$, so does their sum $V_1 + V_2$ because it is a quotient of $V_1 \oplus V_2$.
\end{proof}

Similarly to the direct sum decomposition by intrinsic $\partial$-radii, we have a direct sum decomposition by refined intrinsic $\partial$-radii.  The latter is in fact deduced from the former by twisting $\partial$-differential modules of the form $\calL_{x, (n)}$.

\begin{theorem}\label{T:refined-decomposition}
Let $K$ and $V$ be as in Hypothesis~\ref{H:refined}.  Then $V$ admits a unique direct sum  decomposition
\begin{equation}\label{E:refined-decomposition}
V = \bigoplus_{\{\theta\} \subset \kappa_{K^\alg}^{(s)}} V_{\{\theta\}},
\end{equation}
where the direct sum runs through all $\Gal(K^\alg/K)$-orbits $\{\theta\}$ in $\kappa_{K^\alg}^{(s)}$,
such that the refined $\partial$-radii of $V_{\{\theta\}}$ is a multiset consisting of the $\Gal(K^\alg/K)$-orbit $\{\theta\}$  with appropriate multiplicities.

Moreover, if $K'$ is a finite tamely ramified extension of $K$ such that all the $\theta$'s in the above decomposition belong to $\cup_n \big(\kappa_{K'}^{(p^ns)} \big)^{1/p^n}$,  then we have a unique direct sum decomposition
\[
V \otimes_K K'= \bigoplus_{\theta \in \kappa_{K'^\alg}^{(s)}} V_{\theta},
\]
of $\partial$-differential modules over $K'$ such that each $V_\theta$ has pure refined $\partial$-radii $\theta$. 
\end{theorem}
\begin{proof}
The statement is void if $IR_\partial(V) = 1$.  We assume $IR_\partial(V)<1$ from now on.
We first replace $K$ by the $K'$ in the theorem; using the uniqueness of such direct sum decomposition and Galois descent, we may recover the statement over $K$. Note that Lemma~\ref{L:tame-extension} implies that $\partial$ is still a derivation of rational type.

We first assume that either $p=0$, or $p>0$ and $IR_\partial(V) < p^{-1/(p-1)}$.
For each $\theta \in \Theta_\partial(V)$, we construct $\calL_{x, (n)}$ as in Example~\ref{Ex:pure-refined}, which is a rank $d$ $\partial$-differential module with pure $\partial$-radii $R_\partial(V)$ and pure refined radii $\theta$.  By Lemma~\ref{L:interpretation-refined}(b), $V \otimes \calL_{x, (n)}^\dual$ does not have pure radii $R_\partial(V)$.  Theorem~\ref{T:decomp-over-field-complete} then gives rise to a decomposition $V \otimes \calL_{x, (n)}^\dual = W_0 \oplus W_1$, where $R_\partial(W_0) > R_\partial(V)$ and $W_1$ has pure $\partial$-radii $R_\partial(V)$.

Put $\widetilde W_0 = W_0 \otimes \calL_{x, (n)}$ and $\widetilde W_1 = W_1 \otimes \calL_{x, (n)}$.
Consider the following homomorphisms of $\partial$-differential modules:
\[
\xymatrix{
V \ar@/^1pc/[r]^-{i} & V \otimes \calL_{x, (n)}^\dual \otimes \calL_{x, (n)} \ar[r]^-{\sim} \ar@/^1pc/[l]^-{j} & \widetilde W_0 \oplus \widetilde W_1,
}
\]
where $i$ is induced by the diagonal embedding $K \inj \calL_{x, (n)}^\dual \otimes \calL_{x, (n)}$ and $j$ is induced by the trace map $\calL_{x, (n)}^\dual \otimes \calL_{x, (n)} \surj K$ normalized so that $ji=\id$.  Let $p_0$ and $p_1$ be the projections from $V \otimes \calL_{x, (n)}^\dual \otimes \calL_{x, (n)} $ to the factors $\widetilde W_0$ and $\widetilde W_1$, respectively, viewed as submodules of the source.  We then have $p_0^2 = p_0$, $p_1^2 = p_1$, and $p_0 + p_1 =1$.

We claim that $jp_0i$ and $jp_1i$ are projectors on $V$.
Indeed, Lemma~\ref{L:interpretation-refined}(c) implies that $R_\partial(\calL_{x, (n)}^\dual \otimes \calL_{x, (n)} ) > R_\partial(V)$.  By Lemma~\ref{L:V-otimes-W-not-pure}, $V \otimes \calL_{x, (n)}^\dual \otimes \calL_{x, (n)} $ and hence $\widetilde W_0$ and $\widetilde W_1$ have pure $\partial$-radii $R_\partial(V)$.
Lemma~\ref{L:V-otimes-W-not-pure} also implies that $\Theta_\partial(\widetilde W_0)$ consists of solely $\theta$, and by the ``moreover" part of Lemma~\ref{L:interpretation-refined},  we have
\[
\Theta_\partial(\widetilde W_1) = \big\{ \theta_1 + \theta \textrm{ (with multiplicity } d)\;|\; \theta_1 \in \Theta_\partial(W_1) \big\}.
\]
In particular, we have $\theta \notin \Theta_\partial(\widetilde W_1)$.  Hence any homomorphism of $\partial$-differential modules between $\widetilde W_0$ and $\widetilde W_1$ has to be zero by Lemma~\ref{L:interpretation-refined}(a).  In particular, $p_1ijp_0 = p_0ijp_1 = 0$.  Thus, we have
\begin{align*}
(jp_0i)(jp_0i) &= jp_0ij(1-p_1)i = jp_0i(ji) - j(p_0ijp_1)i = jp_0i\\
(jp_1i)(jp_1i) &= jp_1ij(1-p_0)i = jp_1i(ji) - j(p_1ijp_0)i = jp_1i\\
jp_0i + jp_1i &= j(p_0 + p_1) i = ji = 1.
\end{align*}
This proves that $V = jp_0i(V) \oplus jp_1i(V)$.  Moreover, by Lemma~\ref{L:refined-subquotient}(i), $\Theta_\partial(jp_0i(V))$ consists of only $\theta$ since it is a quotient of $\widetilde W_0$, and $\Theta_\partial(jp_1i(V))$ does not contain $\theta$ since it is a quotient of $\widetilde W_1$.  Applying this process to each of $\theta \in \Theta_\partial(V)$ gives the desired decomposition~\eqref{E:refined-decomposition}.

The uniqueness of the direct sum decomposition follows from Lemma~\ref{L:refined-subquotient}(b).

Now if $p>0$ and $IR_\partial(V) = p^{-1/(p-1)}$, the decomposition~\eqref{E:refined-decomposition} comes from the decomposition of its $\partial$-Frobenius descendent, via the relation described in Proposition~\ref{P:refined-frob}(2).  If $p>0$ and $IR_\partial(V) > p^{-1/(p-1)}$, the decomposition~\eqref{E:refined-decomposition} comes from the decomposition of its $\partial$-Frobenius antecedent, via the relation described in Proposition~\ref{P:refined-frob}(b).
\end{proof}

We now can prove some fundamental properties for tensor products of $\partial$-differential modules with pure $\partial$-radii and pure refined $\partial$-radii.  One can combine this with Theorems~\ref{T:decomp-over-field-complete} and \ref{T:refined-decomposition} to obtain corresponding results for general $\partial$-differential modules.

\begin{proposition}\label{P:refined-properties-general}
Let $V$ and $W$ be two $\partial$-differential modules over $K$ with pure $\partial$-radii $R_\partial(V) = R_\partial(W)< |u|^{-1}$ and pure refined $\partial$-radii $\theta_V$ and $\theta_W$, respectively. 
\begin{itemize}
\item[\emph{(a)}] Then $W^\dual$ has pure refine $\partial$-radii $-\theta_W$.
\item[\emph{(b)}] If $\theta_V = \theta_W$, then we have $R_\partial(V\otimes W^\dual) > R_\partial(V)$.
\item[\emph{(c)}] If $\theta_V \neq \theta_W$, then $V \otimes W^\dual$ has pure $\partial$-radii $R_\partial(V)$ and pure refined $\partial$-radii $\theta_V - \theta_W$.
\item[\emph{(d)}]
Moreover, if we do not assume that $V$ and $W$ has pure refined $\partial$-radii and let $U$ denote the maximal submodule of $V \otimes W^\dual$ that has $\partial$-radii strictly larger than $R_\partial(V)$, then we have
\[
\dim U = \sum_{\theta \in \kappa_{K^\alg}^{(s)}} \multi_\theta(\Theta_\partial(V)) \cdot \multi_\theta(\Theta_\partial(W)).
\]
\end{itemize}
\end{proposition}
\begin{proof}
(a) is straightforward, and (d) follows from (b) and (c) by the decomposition~\eqref{E:refined-decomposition}.

When $IR_\partial(V) = IR_\partial(W) < \omega$, (b) follows from Lemma~\ref{L:interpretation-refined}(c) and (c) follows from the ``moreover part" of  the same lemma.  When $p>0$ and $IR_\partial(V) = IR_\partial(W)>p^{-1/(p-1)}$, (b) and (c) for $V$ and $W$  follow from the same statement for the $\partial$-Frobenius antecedents of $V$ and $W$, by the relation described in Proposition~\ref{P:refined-frob}(a).  
It then suffices to prove (b) and (c) in the case when $p>0$ and $IR_\partial(V) = IR_\partial(W)=p^{-1/(p-1)}$.  

In this case, Lemma~\ref{L:frob-properties}(3) implies that $\varphi^{(\partial)}_*V \otimes (\varphi^{(\partial)}_*W)^\dual = \big(\varphi^{(\partial)}_*(V \otimes W^\dual) \big)^{\oplus p}$.  Note that Proposition~\ref{P:refined-frob}(2) implies that the multiset of refined intrinsic $\partial$-radii of $V$ (resp. $W$) is composed of all the solutions to $(\frac xp)^p - \frac xp = u\theta_V$ (resp. $(\frac xp)^p - \frac xp = u\theta_W$), each with multiplicity $\dim V$ (resp. $\dim W$).  If $\theta_V \neq \theta_W$, by (c) in the visible case together with Theorem~\ref{T:refined-decomposition}, the multiset of refined intrinsic $\partial'$-radii of $\varphi^{(\partial)}_*V \otimes (\varphi^{(\partial)}_*W)^\vee$ consists of roots of $(\frac xp)^p - \frac xp = u(\theta_V - \theta_W)$, each with multiplicity $p \dim V \dim W$.  The statement (c) then follows from Proposition~\ref{P:refined-frob}(b).  If $\theta_V = \theta_W$, by (b) in the visible case together with Theorem~\ref{T:refined-decomposition}, $\varphi^{(\partial)}_*V \otimes \varphi^{(\partial)}_*W$ has a submodule of dimension $(p-1)\dim V \dim W$ whose intrinsic $\partial'$-radius is strictly larger than $p^{-p/(p-1)}$.  By Lemma~\ref{L:frob-properties}(e), this can  happen only if $IR_\partial(V \otimes W) > p^{-1/(p-1)}$, which is what we need to prove in (b).
\end{proof}

\begin{remark}
\label{R:visible-instead-of-rational}
We remark that if we do not assume that $\partial$ is  of rational type but assume that $R_\partial(V)<|\partial|_K^{-1}$ instead, all the results in the subsection still hold (note that Frobenius antecedent in the visible case).
\end{remark}

\subsection{Multiple derivations}
\label{S:multi-derivations}

Having studied the situation of one single derivation, we now let multiple commutative derivations to interact.  This essentially amounts to putting the information from each derivation together.  To give the refined radii for multiple derivations a more canonical definition, we will represent the multiset of refined radii as a multiset of differential forms; for this, we need to check the naturality of such presentation.

\begin{notation}
In this subsection, we put $J = \{1, \dots, m\}$.
\end{notation}

\begin{definition}
\label{D:multi-differential-module}
Let $K$ be a differential ring of order $m$, that is a ring equipped with $m$ commuting derivations $\serie{\partial_}m$.   A \emph{$\partial_J$-differential module},
or simply a \emph{differential module},
is a finite projective $K$-module $V$ equipped with commuting 
actions of $\serie{\partial_}m$.  We will apply the results in previous subsections to each $\partial_j$ separately.
\end{definition}

\begin{definition}
Let $K$ and $V$ be as above, and let $R$ be a complete $K$-algebra.  For $\bbv \in V$ and $T_1, \dots, T_m \in R$, we define the \emph{$\partial_J$-Taylor series} to be
\[
\TT(\bbv; \partial_J; T_1, \dots, T_m) = \sum_{e_J = 0}^\infty \frac{\partial_J^{e_J}(\bbv)}{(e_J)!} T_J^{e_J} \in V \otimes_K R,
\]
if it converges.
\end{definition}

We will need the following tautological lemma later in the proof of Theorem~\ref{T:independence-of-basis}.

\begin{lemma}
Let $\partial = \alpha_1 \partial_1 +\cdots + \alpha_m \partial_m$ be another derivation, where $\alpha_1, \dots, \alpha_m \in K$.  To simplify notation, we formally write $\alpha_j = \partial(u_j)$ for any $j \in J$ (and one can check that the formula \eqref{E:Taylor-change-basis} can be written with no reference to $u_j$).  Then, for any $x \in V$, we have
\begin{equation}\label{E:Taylor-change-basis}
\TT\Big(x; \partial_J; \TT(u_1; \partial; \delta) - u_1, \dots, \TT(u_m; \partial; \delta) - u_m \Big) = \TT(x; \partial; \delta),
\end{equation}
as  formal power series in $V \otimes_K K\llbracket \delta \rrbracket$.
\end{lemma}
\begin{proof}
Since \eqref{E:Taylor-change-basis} is a tautological statement, we may assume that $K$ is $\ZZ$-torsion free.
It suffices to show that \eqref{E:Taylor-change-basis} is true modulo $\delta^n$ for any $\partial_J$-differential module $V$ and for any $x \in V$, by induction on $n$.  This is clear for $n=1$.  Assume that we have proved this claim for $n$ and we need to prove it for $n+1$.  It suffices to prove the equality
\[
\frac \partial{\partial \delta} \TT\Big(x; \partial_J; \TT(u_1; \partial; \delta) - u_1, \dots, \TT(u_m; \partial; \delta) - u_m \Big) = \frac \partial{\partial \delta}\TT(x; \partial; \delta) = \TT(\partial(x); \partial; \delta)
\]
modulo $\delta^n$ (note that the derivation reduces the exponents on $\delta$ by $1$).  We compute the left hand side as follows.
\begin{align*}
&
\frac \partial{\partial \delta} \TT\Big(x; \partial_J; \TT(u_1; \partial; \delta) - u_1, \dots, \TT(u_m; \partial; \delta) - u_m \Big) \\
=& \sum_{e_J=0}^\infty \frac{\partial_J^{e_J}(x)}{(e_J)!}
\frac \partial{\partial \delta}
\Big(
\big(\TT(u_1; \partial; \delta) - u_1\big)^{e_1}
\cdots \big(\TT(u_m; \partial; \delta) - u_m \big)^{e_m}
\Big) \\
=& \sum_{e_J=0}^\infty \frac{\partial_J^{e_J}(x)}{(e_J)!}\
\Big(
\sum_{j \in J} e_j \cdot
\big(\TT(u_1; \partial; \delta) - u_1\big)^{e_1}\cdots \big(\TT(u_j; \partial; \delta) - u_j \big)^{e_j-1}
\cdots \big(\TT(u_m; \partial; \delta) - u_m \big)^{e_m} \cdot \frac \partial{\partial \delta}\TT(u_j; \partial; \delta)
\Big) \\
=& \sum_{j \in J}
\sum_{e_J=0}^\infty \frac{\partial_J^{e_J}\big(\partial_j(x)\big)}{(e_J)!}\
\Big(
\big(\TT(u_1; \partial; \delta) - u_1\big)^{e_1}
\cdots \big(\TT(u_m; \partial; \delta) - u_m \big)^{e_m} \cdot \frac \partial{\partial \delta}\TT(u_j; \partial; \delta)
\Big)
\end{align*}
By the induction hypothesis, modulo $\delta^n$, this is congruent to
\begin{align*}
& \sum_{j \in J} \TT \big(\partial_j(x); \partial; \delta \big)\cdot \frac \partial{\partial \delta}\TT(u_j; \partial; \delta) = \sum_{j \in J} \TT \big(\partial_j(x); \partial; \delta \big)\cdot \TT(\partial(u_j); \partial; \delta)\\
=& \TT \big( \sum_{j \in J}\partial_j(x)\partial(u_j); \partial; \delta\big) = \TT \big( \partial(x); \partial; \delta\big).
\end{align*}
This finishes the induction and proves the lemma.
\end{proof}

\begin{definition}\label{D:multi-op-cvgt-radii}
Let $K$ be a complete nonarchimedean
differential field of order $m$ and characteristic zero,
and let $V$ be a nonzero $\partial_J$-differential module over $K$.
Define the \emph{intrinsic radius} of $V$ to be
$$
IR(V) = \min_{j \in J} \left\{IR_{\partial_j}(V) \right\} = \min_{j \in J} \left\{ |\partial_j|_{\sp, K} \big/ |\partial_j|_{\sp, V}\right\}.
$$
For $j \in J$, we say $\partial_j$ is \emph{dominant} for $V$ if 
$IR_{\partial_j}(V) = IR(V)$.  We define the multiset of \emph{intrinsic subsidiary radii}  $\gothI\gothR(V) = \{IR(V;1), \dots, IR(V;\dim V)\}$ by collecting and ordering 
intrinsic radii from the Jordan-H\"older constituents, as in
Definition~\ref{D:partial-radii}.
 We again say that $V$ has \emph{pure intrinsic
radii} if $\gothI\gothR(V)$ consists of $\dim V$ copies of  $IR(V)$.

We similarly define the \emph{extrinsic radius} $ER(V)$ to be the minimum of $R_{\partial_j}(V)$ and the multiset of \emph{extrinsic subsidiary radii} $\gothE\gothR (V)  = \{ ER(V; 1), \dots, ER(V; \dim V) \}$ by collecting and ordering extrinsic radii from the Jordan-H\"older constituents.
\end{definition}

\begin{definition} \label{D:rational type multi}
Let $K$ be a complete nonarchimedean differential field of order $m$ and
characteristic zero.
We say that $K$ is of \emph{rational type} with respect to a set of parameters
$\{u_j| j \in J\}$ if each $\partial_j$ is of rational type with respect to
$u_j$, and $\partial_i(u_j) = 0$ for $i \neq j$ in $J$.
\end{definition}

\begin{hypothesis}\label{H:K-multi-derivative}
For the rest of this subsection, let $K$ be a complete nonarchimedean field of characteristic zero, equipped with derivations $\partial_J$ of rational type with respect to parameters $u_J$.  Let $V$ be a $\partial_J$-differential module with pure $\partial_j$-radii for each $j \in J$.  We assume moreover that $IR(V) <1$.

\end{hypothesis}

\begin{notation}
For each $j$, put $s_j = -\log (\omega R_{\partial_j}(V)^{-1})$, $\lambda_j = \lambda(IR_{\partial_j}(V))$, and $r_j = r(IR_{\partial_j}(V))$.  By Theorem~\ref{T:decomp-over-field-complete}, we have $s_j \in \QQ \cdot \log |K^\times|$ for any $j$.
\end{notation}

\begin{definition}\label{D:multi-refined}
By Theorem~\ref{T:refined-decomposition}, we may replace $K$ by a finite tamely ramified extension such that $V$ admits a direct sum decomposition $V = \oplus V_{\theta_J}$, where each direct summand $V_{\theta_J}$ has pure refined $\partial_j$-radii $\theta_j$ for any $j \in J$.
Define the multiset of \emph{refined radii} of $V$, denoted by $\Theta(V)$, to be the collection of $\vartheta = \sum_{j \in J} \theta_j du_j$ with multiplicity $\dim V_{\theta_J}$, where $\vartheta$ is viewed as an element of $\oplus_{j \in J} \kappa_{K^\alg}^{(s_j)} du_j$.  The reason that we write the refined radii in the form of differentials will be justified later, in Theorem~\ref{T:independence-of-basis}.

We will also consider cases where the derivations with larger radii of convergence are ignored.
\begin{itemize}
\item [(i)] Let $\calI\Theta(V)$ be the multiset consisting of  elements $\sum \theta_j du_j$ with multiplicity $\dim V_{\theta_J}$, where the sum is only taken over those $j$ such that $IR_{\partial_j}(V_{\theta_J}) = IR(V_{\theta_J})$; this is called the multiset of \emph{refined intrinsic radii}.  Often, we view it as a multiset of elements in $\oplus_{j \in J} \kappa_{K^\alg}^{(\goths)} \frac{du_j}{u_j}$ for $\goths =-\log(\omega IR(V)^{-1})$.
We remark that this definition does not depend on the field extension of $K$ we made earlier.
\item [(ii)] Let $\calE\Theta(V)$ be the multiset consisting of  elements $\sum \theta_j du_j$ with multiplicity $\dim V_{\theta_J}$, where the sum is only taken over those $j$ such that $R_{\partial_j}(V_{\theta_J}) = R(V_{\theta_J})$.  We call it the \emph{refined extrinsic radii}.
\end{itemize}

\end{definition}

\begin{definition}
\label{D:multi-good-norm}
Let $(\bbb_1, \dots, \bbb_m) \in (0,1]^m$.  A norm $ |\cdot|_V$ on $V$ is \emph{$(\bbb_1, \dots, \bbb_m)$-good} (or simply \emph{good} if $\bbb_j = IR_{\partial_j}(V)$ for all $j\in J$), if it is $\bbb_j$-good with respect to $\partial_j$ for all $j \in J$.
\end{definition}

\begin{remark}
In contrast to the single derivation case, we do not know if a good norm exists in general, unless we assume that $K$ is discretely valued, in which case, Lemma~\ref{P:multi-good-norm-exists} below gives an affirmative answer.  This assumption may not be necessary for some of the results later, as one might get around this using some approximation process.  Since we will work with complete discrete valuation field in most applications, we restrict ourselves here to this case.
\end{remark}

\begin{hypothesis}\label{H:K-discrete}
For the rest of this subsection, we assume that $K$ is discretely valued.
\end{hypothesis}

\begin{lemma}\label{P:multi-good-norm-exists}
Assume that $\bbb_j \in (0, IR_{\partial_j}(V)] \cap |K^\times|^\QQ$ for any $j\in J$, and that $\bbb_j <1$ if $p>0$.
Then the differential module $V$ admits a $(\bbb_1, \dots, \bbb_m)$-good norm.
\end{lemma}
\begin{proof}
We first remark that if $IR_{\partial_j}(V) <1$, Theorem~\ref{T:decomp-over-field-complete} implies that $IR_{\partial_j}(V) \in |K^\times|^\QQ$.  To prove the lemma, we may assume $\bbb_j = IR_{\partial_j}(V)$.

By the same argument as in Lemma~\ref{L:good-norm-exists} using Frobenius antecedent, it suffices to prove the lemma under the assumption that $\bbb_j \leq \omega$ for any $j \in J$. Note that the $\partial_j$-Frobenius antecedent is compatible with  $\partial_{j'}$ for $j' \neq j$. 
Let $K'$ be the completion of $K(x_J)$ with respect to the $\ee^{-s_J}$-Gauss norm, where we set $\partial_j(x_{j'}) = 0$ for all $j,j' \in J$ and $s_j = -\log(\omega (\bbb|u|)^{-1})$. In particular, $K'$ is discretely valued since $\ee^{-s_j} \in |K^\times|^\QQ$ for any $j \in J$.

We first show that $V' :=V \otimes K'$ has a $(\bbb_1, \dots, \bbb_n)$-good norm.  For this, it suffices to show that given any norm $|\cdot|_{V'}$ with orthonormal basis $\serie{\bbe_}d$, the submodule $M'$ of $V'$ generated by 
\[
\big\{ x_J^{a_J}\partial_J^{a_J} \bbe_i \big| a_j \in \ZZ_{\geq 0} \textrm{ for any } j \in J \textrm{ and } i \in \{1, \dots, d\} \big\}
\]
over $\calO_{K'}$ is a finite $\calO_{K'}$-module; if so, $M'$ gives rise to a norm on $V'$, under which $|\partial_j| \leq |x_j| = \ee^{-s_j}$ for all $j$ verify the conditions of $(\bbb_1, \dots, \bbb_n)$-good norm in Definition~\ref{D:good-norm}. To prove that $M'$ is a finite $\calO_{K'}$-submodule, it suffices to prove that $|x_j^n\partial_j^n|_{V'}$ is bounded for each $j$ as $n \rar +\infty$ (we used here the fact that $K'$ is discretely valued, otherwise boundness may not imply finiteness.)  It is then enough to verify this boundness condition for any $K'$-norm on $V'$.  In particular, for each of $\partial_j$, we can choose a $\bbb_j$-good norm by Lemma~\ref{L:good-norm-exists}, for which $|x_j^n\partial_j^n|_{V'} \leq 1$.  Thus $M'$ is finite over $\calO_{K'}$ and hence we have a $(\bbb_1, \dots, \bbb_n)$-good norm on $V'$.

This norm restricts to a $K$-norm on $V$ satisfying all the norm conditions in Definition~\ref{D:good-norm}.  We use the following lattice lemma to show that it admits an orthogonal basis.
\end{proof}

\begin{lemma}\label{L:lattice}
Let $F$ be a complete \emph{discrete} valuation field and let $V$ be a finite dimensional vector space, equipped with a norm compatible with $F$.  Assume moreover that the valuation group $\log |V- \{0\}|_V$ of $V$ is also discrete.  Then $V$ admits an orthogonal basis.
\end{lemma}
\begin{proof}
The proof is almost the same as \cite[Lemma~1.3.7]{kedlaya-course}.  For completeness and the convenience of the reader, we reproduce it here.

We use induction on the dimension $n = \dim V$.  When $n = 1$, the statement is obvious; any nonzero vector forms an orthogonal basis.  Now assuming the statement for $n-1$,  we will prove it for an $n$-dimensional $F$-normed vector space $(V, |\cdot|_V)$ whose valuation group is discrete.  Pick a nonzero vector $v_1 \in V$ and denote $W = V / Fv_1$, provided with the quotient norm $|\cdot|_W$; this is again $F$-compatible and has discrete valuation group.  By the inductive hypothesis, $W$ admits an orthogonal basis $\bar v_2, \dots, \bar v_n$.  For $i = 2, \dots, n$, we pick $v_i \in V$ that lifts  $\bar v_i \in W$ such that $|v_i|_V =|\bar v_i|_W$ (this is possible because $V$ has discrete valuation group).  We claim that $v_1, \dots, v_n$ form an orthogonal basis of $V$.

They obviously form a basis of $V$.
We need to prove that for any $v = x_1v_1+\cdots +x_nv_n \in V$, $|v|_V = \max_i\{|x_i| |v_i|_V\}$.  It is clear that $|v|_V$ is less than or equal to the right hand side; we need to show $|v|_V \geq \max_i\{|x_i| |v_i|_V\}$.  We prove it the following two cases separately.

(i) If the maximum above is achieved by some $i\geq 2$, we have
\[
|v|_V \geq |v\mod Fv_1|_W = |x_2\bar v_2+ \cdots + x_n\bar v_n|_W = \max_{i=2}^n\{|x_i| |\bar v_i|_W\} = \max_{i=1}^n\{|x_i| |v_i|_V\}.
\]

(ii) We have $|x_1||v_1|>|x_i||v_i|$ for all $i = 2, \dots, n$.  In this case, we have  $|v| = |x_1||v_1| = \max_i\{|x_i| |v_i|_V\}$.

This shows that $v_1, \dots, v_n$ form an orthogonal basis of $V$ and finishes the proof of the lemma.
\end{proof}

\begin{remark}
One may hope to find an analog of Example~\ref{Ex:pure-refined} for $\partial_J$-differential modules. This, however, amounts to carefully choosing the element $x$ in Example~\ref{Ex:pure-refined} so that the actions of $\partial_J$ \emph{commutes}.  For this, we might need to restrict the possible intrinsic refined radii to a subset of $\oplus_{j \in J} \kappa_{K^\alg}^{(\goths)} \frac{du_j}{u_j}$, where $\goths = -\log( \omega IR(V)^{-1})$.  Unfortunately, we do not know how to identify this subset in general.  Proposition~\ref{P:refined-is-boundary-p=0} below partly answers this question.

It would be interesting to know, when $p>0$, whether any element in $\oplus_{j \in J} \kappa_{K^\alg}^{(\goths)} \frac{du_j}{u_j}$ can appear in the multiset of refined intrinsic radii of some differential module.
The referee also pointed out that the reduction of $\partial_j$ may give rise to a $\calD$-module in characteristic $p$.  We do not know if this construction is independent of the choice of good norms.  But we suspect that this is related to the reduction of some arithmetic $\mathscr D$-module when the differential module comes from one.
\end{remark}

\begin{proposition}\label{P:refined-is-boundary-p=0}
Assume that $IR(V) < \omega$ and that $p=0$ or $d = \rank V = 1$.  Let $\goths = -\log( \omega IR(V)^{-1})$. Note that the action of $u_j \partial_j$ on $K$ induces a derivation on $\kappa_{K^\unr}^{(\goths)}$.  If $\vartheta = \sum_{j \in J} \theta_j \frac{du_j}{u_j} \in \calI\Theta(V)$, then for $i, j \in J$, we have $u_i\partial_i \theta_j = u_j\partial_j \theta_i$ in $\kappa_{K^\unr}^{(\goths)}$.
\end{proposition}
\begin{proof}
By possibly replacing $K$ by a finite tamely ramified extension, we reduce to the case when $V$ is irreducible with a good norm given by an orthonormal basis, and when $V$ has pure refined intrinsic radii $\sum_{j \in J} \theta_j \frac{du_j}{u_j}$.  The $u_j\partial_j$-action with respect to this basis is given by a matrix $N_j \in \Mat_{d \times d}(\gothm_K^{(\goths)})$.  Since $\partial_i$ and $\partial_j$ commute with each other for any $i, j \in J$, we have
\begin{equation}\label{E:refined=boundary}
N_iN_j + u_i \partial_i(N_j) = N_j N_i + u_j \partial_j(N_i).
\end{equation}
Taking the trace of \eqref{E:refined=boundary} gives $d\cdot u_i\partial_i \theta_j = d \cdot u_j\partial_j \theta_i$, which yields the proposition because $d$ is invertible in $\kappa_K$.
\end{proof}

Before proceeding, we need some notation to use in Theorem~\ref{T:independence-of-basis} below.

\begin{notation}
If $p>0$, we can write an integer $n \in \NN$ as $n = a_0 + pa_1 +\cdots +p^k a_k$ with $a_1, \dots, a_k \in \{0, \dots, p-1\}$. Put $\sigma_p(n) = a_0+ \cdots + a_k$ if $p>0$, and $\sigma_p(n) = 0$ if $p=0$.  It is straightforward to check that $\sigma_p(n_1) + \sigma_p(n_2) \geq \sigma_p(n_1+n_2)$ for $n_1, n_2 \in \NN$, and that $|n!| = \omega^{n - \sigma_p(n)}$ for $n \in \NN$.
\end{notation}

The following theorem explains how refined radii change when we consider a different set of derivations, and hence justifies the reason we wrote refined radii in the form of differentials in Definition~\ref{D:multi-refined}.

\begin{theorem}\label{T:independence-of-basis}
Assume that $V$ has pure refined $\partial_j$-radii $\theta_j \in \kappa_{K^\alg}^{(s_j)}$ for any $j \in J$. Let $ K'$ be a complete discrete valuation field containing $K$.
Let $\partial$ be a derivation on $K'$, extending the action of $\alpha_1 \partial_1 + \cdots + \alpha_m \partial_m$ on $K$ to $K'$, where $\alpha_1, \dots, \alpha_m \in K'$. In fact, we have $\alpha_j = \partial(u_j)$ for any $j \in J$.  We assume that $ \partial$ is a derivation of rational type on $K'$. Set $s = \min_{j \in J} \{s_j-\log|\alpha_j|\}$ and let  
$J_0$ be a subset of $J$ consisting of $j$ for which $s = s_j - \log|\alpha_j|$.  Assume moreover that $IR_j(V) <1$ if $j \in J_0$.
Let $\theta = \sum_{j \in J_0}\alpha_j \theta_j \in \kappa_{ K'^\alg}^{(s)}$.

Then $R_\partial(V \otimes_K K') \leq \omega \ee^s$, and the equality is achieved if and only if $\theta \neq 0$ in $\kappa_{K'^\alg}^{(s)}$.  Moreover, when equivalent statement is verified, $V \otimes_K K'$ has pure $\partial$-radii $\omega \ee^s$ and pure refined $\partial$-radii $\theta$.
\end{theorem}
\begin{proof}
For $j \in J$, the equality $\alpha_j = \partial(u_j)$ follows from applying $\partial$ to $u_j$.

By Proposition~\ref{P:multi-good-norm-exists} and by possibly enlarging $K$ and $K'$, we may assume that $V$ admits a norm given by some orthonormal basis $\underline \bbe$ such that, for any $j \in J$,
\begin{itemize}
\item[(i)] if $IR_j(V)<1$, the norm is good with respect to $\partial_j$;

\item[(ii)] if $IR_j(V)=1$, the norm is $\bbb_j$-good with respect to $\partial_j$ for some  $\bbb_j$ in $(|\alpha_j| e^{s-s_j}, 1) \cap |K^\times|^\QQ$.  In this case, instead of taking the usual definitions of $r_j$, $\lambda_j$, and $s_j$, we set $r_j=r(\bbb_j)$, $\lambda_j=\lambda(\bbb_j)$, and $s_j=s-\log(\bbb_j|\alpha_j|^{-1})$.   Note that $s_j - \log |\alpha_j|>s$ still holds.
\end{itemize}
Similarly to Notation~\ref{N:dichotomy}, we define integers $r$ and $\lambda$ as follows.
\begin{itemize}
\item[(x)] When $|\partial|_{K'}\omega \ee^s< \omega$ we denote $\lambda= 0$ and $r = 1$.
\item[(xx)] When $|\partial|_{K'}\omega \ee^s \in [\omega, 1)$ and $p>0$, let $\lambda$ denote the unique nonnegative integer such that  $|\partial|_{K'}\omega \ee^s \in [p^{-1/p^{\lambda-1}(p-1)}, p^{-1/p^\lambda(p-1)})$, and put $r = p^{\lambda}$.  In this case, we have $(|\partial|_{K'}\omega\ee^s)^{p^k} \leq \omega$ for $k< \lambda$ and hence $(|\partial|_{K'}\omega\ee^s)^{i} \leq \omega^{\sigma_p(i)}$ for $i = 1, \dots, r-1$.
\end{itemize}
For each $j \in J$, we have
\[
\Big|\frac{\partial^i_j}{i!} \Big|_V \leq |\partial_j|^i_K, \textrm{ for }i = 1, \dots, r_j-1, \quad \textrm{and }
|\partial_j^{r_j} |_V \leq |u_j|^{-r_j} \ee^{-r_js_j}.
\]
For $i = 1, \dots, r$, the action of $\partial^i$ on an element $x$ of $\underline e$ can be expressed in terms of the actions of $\partial_J$, according to the coefficients of $\delta^i$ on the left hand side of \eqref{E:Taylor-change-basis}, applied to $x$. More precisely, for any $j \in J$ and any $i \in \NN$, the coefficient of $\delta^i$ in $\TT(u_j; \partial; \delta) -u_j$ has norm $\leq |\partial(u_j)| |\partial|^{i-1}_{K'} = |\alpha_j| |\partial|^{i-1}_{K'}$.  For any coefficient that arises in the $\partial_J$-Taylor series expansion, if we put $e_j = c_j + d_jr_j $ with $c_j \in \{0,\dots, r_j-1\}$ and $d_j \in \ZZ_{\geq 0}$   for any $j \in J$, then we have
\[
\Big| \frac{\partial_J^{e_J}(x)}{(e_J)!} \Big|_V \leq \prod_{j \in J} \Big| \frac{\partial_j^{d_jr_j}}{(d_jr_j)!} \Big|_V \cdot \prod_{j \in J} \Big| \frac{\partial_j^{c_j}(x)}{(c_j)!} \Big|_V   
\leq |x|_V \cdot \prod_{j \in J}|\partial_j|_K^{c_j}\cdot
\prod_{j \in J} \big( \ee^{-d_jr_js_j} \omega^{-d_jr_j+\sigma_p(d_jr_j)} \big),
\]
Putting these two bounds together, we see that if a $\delta^i$-term on the left hand side of \eqref{E:Taylor-change-basis} arises in a term that includes $\frac{\partial_J^{e_J}(x)}{(e_J)!}$ (which particularly implies that $i \geq e_1+ \cdots+ e_m$), then its norm is smaller than or equal to
\begin{eqnarray*}
&& |x||\partial|_{K'}^{i-e_1-\cdots -e_m}
\prod_{j \in J}|\alpha_j|^{e_j} \prod_{j \in J}|\partial_j|_K^{c_j}\cdot
\prod_{j \in J} \big( \ee^{-d_jr_js_j} \omega^{-d_jr_j+\sigma_p(d_jr_j)} \big) \\
&=& |x||\partial|_{K'}^{i-e_1-\cdots -e_m}
\prod_{j \in J}\big(|\partial_j|_K |\alpha_j|\big)^{c_j}\cdot
\prod_{j \in J} \big( \big(|\alpha_j|_K \ee^{-s_j} \big)^{d_jr_j} \omega^{-d_jr_j+\sigma_p(d_jr_j)} \big) \\
&\leq& |x||\partial|_{K'}^{i-e_1-\cdots -e_m}
\prod_{j \in J}|\partial|_{K'}^{c_j}\cdot
\prod_{j \in J} \big( \ee^{-d_jr_js} \omega^{-d_jr_j+\sigma_p(d_jr_j)} \big) 
\quad \ \textrm{(note } |\partial|_{K'} \geq |\partial(u_j)| |u_j|^{-1} = |\alpha_j||\partial_j|_K)
\\
&\leq& |x||\partial|_{K'}^i
(|\partial|_{K'}\omega \ee^s)^{-\sum_jd_jr_j} \omega^{\sigma_p(\sum_jd_jr_j)}.
\end{eqnarray*}
When $i=1, \dots, r-1$, the coefficient of this $\delta^i$-term has norm $\leq |\partial|^i_{K'} |x|$ by condition (xx).  When $i = r$, this $\delta^i$-term has norm $\leq |\partial|_{K'}^r \big((|\partial|_{K'}\omega \ee^s)^{-r} \omega \big)|x| = \omega^{-r+1} \ee^{-rs}|x|$; the equality can happen only when $\sum_j d_jr_j = r$ and $\sigma_p(\sum_j d_jr_j) = \sum_j \sigma_p(d_jr_j)$, which together yield
$e_j = r$ for some $j \in J_0$ and $e_{j'} = 0$ for $j' \neq j$.  When equality of norms is achieved, the corresponding $\delta^i$-term is 
$\alpha_j^r \partial_j^r(x) / r!$.
Therefore, modulo the elements with norm smaller than $\ee^{-rs}$, the matrix of $\partial^r$ with respect to $\underline \bbe$ is congruent to $\sum_{j \in J_0}
\alpha_j^r \partial_j^r$; this is a sum of matrices with single eigenvalues $\alpha_j^r\theta_j^r$ for $j \in J_0$ (note that, again, $IR_{\partial_j}(V) <1$ for all $j \in J_0$).
By Lemma~\ref{L:converse-good-norm}, we have $R_\partial(V) \leq \omega \ee^s$ and this is an equality if and only if $\sum_{j \in J_0} \alpha_j^r \theta_j^r \neq 0$ in $\kappa_{K'^\alg}^{(rs)}$, which is equivalent to $\sum_{j \in J_0} \alpha_j \theta_j \neq 0$ in $\kappa_{K'^\alg}^{(s)}$; note that $r$ is always 1 or a power of $p$.  Moreover, if the equivalent condition is satisfied, $V$ has pure  refined $\partial$-radii
\[
\big(\sum_{j \in J_0} \theta_j^r \alpha_j^r \big)^{1/r} = \sum_{j \in J_0} \theta_j \alpha_j = \theta \in \kappa_{K'^\alg}^{(s)}.
\]
\end{proof}

\begin{corollary}\label{C:refined-radii-for-generic-point}
Let $V$ be a $\partial$-differential module over $K$ and let $f =\TT(-; \partial, T) : K \rar K\llbracket T/u\rrbracket_0$ and $f^*V$ be as in Lemma~\ref{L:basic-IR-proposition}(d).  For $\eta \in [0,|u|)$, let $F_\eta$ denote the completion of $K(T)$ with respect to the $\eta$-Gauss norm.
\begin{itemize}
\item[\emph{(a)}]  If $\eta \in (0, R_\partial(V)]$, $f^*V \otimes F_\eta$ has pure intrinsic $\partial_T$-radius $1$; if $\eta \in ( R_\partial(V), |u|)$, $f^*V \otimes F_\eta$ has (extrinsic) $\partial_T$-radius $R_\partial(V)$.
\item[\emph{(b)}]  When $\eta \in (R_\partial(V), |u|)$, we have $\Theta_{\partial_T}(f^*V \otimes F_\eta) = \Theta_\partial(V)$.
\end{itemize}
\end{corollary}
\begin{proof}
For any $x \in V$, $f^*(\partial(x)) = \partial_T(f^*(x))$.  The first statement follows from this immediately, and the second statement follows from Theorem~\ref{T:independence-of-basis}.  (When $IR_\partial(V) = 1$, (b) is void.)
\end{proof}

\begin{remark}
Similar to Remark~\ref{R:visible-instead-of-rational}, if  we do not  assume that $\partial_1, \dots, \partial_n$ are of rational type (but only commutative),  the results from this subsection still hold if, for any $\partial_j$ for which the refined $\partial_j$-radii are relevant, we have $R_{\partial_j}(V) \leq |\partial_j|^{-1}$.
\end{remark}

\subsection{One-dimensional variation of refined radii}
\label{S:one-dim}

Having established the results for differential modules over a field, we now study the case of a differential module over a rigid analytic annulus or a rigid analytic disc.  It is particularly interesting to study how the multisets of radii of the differential module with respect to different Gauss norms vary as we change the radii which define the Gauss norm. Kedlaya and the author had proved various results on this in \cite[Chapter~11]{kedlaya-course} and \cite[Section~2]{kedlaya-xiao}, essentially stating that the radii of convergence are piecewise log-affine functions in the radii of the annulus.  In this subsection, we will characterize how the refined radii change as we change the radii for the Gauss norm, in the case when the functions given by the radii of convergence are in fact log-affine.

\begin{hypothesis}\label{H:K-multi}
Throughout this subsection, we assume that $K$ is a 
complete nonarchimedean field of characteristic zero and residual characteristic $p$. We also assume that $K$ is equipped with derivations $\partial_1, \dots, \partial_m$ of rational type with respect to $u_1, \dots, u_m$. \end{hypothesis}

\begin{notation}
Put $J = \{\serie{}m\}$ and $J^+ = J \cup \{0\}$.
For $\eta > 0$, let $F_\eta$ denote the completion of $K(t)$ under the 
$\eta$-Gauss norm $|\cdot|_\eta$. Set $\partial_0 = \frac d {dt}$ on $K[t]$; it extends by continuity to $F_\eta$ and ring of functions on discs or annuli.  The derivations
$\partial_{J^+}$ are of rational type on
$F_\eta$.
\end{notation}

\begin{notation}\label{N:subsidiary-radii-and-sums}
Fix $j \in J^+$ and an interval $I \subseteq [0, \infty)$.  We say that $I$ is an open interval in $[0, \infty)$ if it is of the form $[0, \beta)$ or $(\alpha, \beta)$, where $0 < \alpha < \beta$. Put $\dot I=I \bs \{0\}$.  For $M$ a $\partial_j$-differential module of rank $d$
over $A_K^1(I)$,  $r \in - \log \,\dot I$, and $i \in \{\serie{}d\}$, we put 
$$
f_i^{(j)}(M,r) = -\log R_{\partial_j}(M \otimes F_{\ee^{-r}}; i), \quad\textrm{and} \quad F_i^{(j)}(M,r) = f_1^{(j)}(M,r) + \cdots + f_i^{(j)}(M,r).
$$
\end{notation}

\begin{theorem} \label{T:variation-j}
Fix $j \in J^+$ and an interval $I \subseteq [0, +\infty)$.  Let $M$ be a $\partial_j$-differential module of rank $d$
over $A_K^1(I)$.  Then we have the following.
\begin{enumerate}
\item[\emph{(a)}]
(Linearity)
For $i=\serie{}d$, the functions $f_i^{(j)}(M,r)$ and $F_i^{(j)}(M,r)$ are continuous.  They are piecewise affine on the locus where $f_i^{(j)}(M, r) > -\log |u_j|$ if $j\in J$; and they are piecewise affine on whole $-\log\, \dot I$ if $j =0$.

\item[\emph{(b)}]
(Weak integrality)
\begin{enumerate}
\item[\emph{(b1)}] Suppose $p=0$ or $j = 0$. If $i=d$ or $f_{i+1}^{(j)}(M, r_0) < f_i^{(j)}(M,r_0)$, the slopes of $F_i^{(j)}(M,r)$ in some neighborhood of $r=r_0$ belong to $\ZZ$.  Consequently, the slopes of each $f_i^{(j)}(M,r)$ and $F_i^{(j)}(M,r)$ belong to $\frac 11\ZZ \cup \cdots \cup \frac 1d\ZZ$.
\item[\emph{(b2)}] Suppose $p>0$ and $j \in J$.
If $f_i^{(j)} (M, r_0) > \frac 1 {p^n(p-1)} \log p - \log |u_j|$ for some $n \in \ZZ_{\geq 0}$, then the slopes of each $f_i^{(j)}(M,r)$ and $F_i^{(j)}(M,r)$ in some neighborhood of $r_0$ belong to $\frac 1 {p^n d!}\ZZ$.
\end{enumerate}

\item[\emph{(c)}]
(Monotonicity)
Suppose $0 \in I$ and suppose either $j \in J$, or $j =0$ and $f_i^{(0)}(M, r_0) > r_0$.
Then the slopes of $F_i^{(j)}(M,r_0)$
are nonpositive in a neighborhood of $r_0$.

\item[\emph{(d)}]
(Convexity)
For $i=1, \dots, d$, the function $F_i^{(j)}(M,r)$ is convex.

\item[\emph{(e)}] (Decomposition)
Assume  that $I$ is an open interval in $(0, +\infty)$.  Suppose that for some $i \in \{1, \dots,d\}$, $F_i^{(j)}(M, r)$ is affine and $f_i^{(j)}(M, r) > f_{i+1}^{(j)}(M, r)$ for $r \in -\log\, \dot I$.  Then we can write $M$ uniquely as the direct sum of two $\partial_j$-differential submodules $M_1$ and $M_2$, such that, for any $\eta \in \dot I$, the multiset of $\partial_j$-radii of $M_1 \otimes F_\eta$ exactly consists of the smallest $i$ elements in the multiset of $\partial_j$-radii of $M \otimes F_\eta$.
\end{enumerate}
\end{theorem}
\begin{proof}
This is \cite[Theorems~2.2.5, 2.2.6, and 2.3.5]{kedlaya-xiao}.
\end{proof}

\begin{notation} \label{N:subsidiary-radii2}
Let $I \subseteq [0, +\infty)$ be an interval and let $M$ be a $\partial_{J^+}$-differential module of rank $d$ 
on $A_K^1(I)$.  
For $r \in - \log \, \dot I$ and $i \in \{\serie{}d\}$, we put 
\[
f_i(M,r) = -\log IR(M \otimes F_{\ee^{-r}}; i), \quad\textrm{and}\quad F_i(M,r) = f_1(M,r) + \cdots + f_i(M,r).
\]
Suppose that $I \subseteq [0,1)$ and that $|u_j| = 1$ for any $j \in J$, we put
\[
\hat f_i(M, r) = -\log ER(M \otimes F_{\ee^{-r}}; i), \quad \textrm{and}\quad\hat F_i(M, r) = \hat f_1(M, r) + \cdots + \hat f_i(M, r). 
\]
\end{notation}

\begin{theorem}
\label{T:variation}
Fix an interval $I \subseteq [0, +\infty)$.  Let $M$ be a $\partial_{J^+}$-differential module of rank $d$ over $A_K^1(I)$.
\begin{enumerate}
\item[\emph{(a)}]
(Linearity)
For $i=\serie{}d$, the functions $f_i(M,r)$ and $F_i(M,r)$ are 
continuous and piecewise affine.

\item[\emph{(b)}]
(Integrality)
If $i=d$ or $f_i(M,r_0) > f_{i+1}(M,r_0)$, then 
the slopes of $F_i(M,r)$ in some neighborhood of $r_0$ belong to $\ZZ$. Consequently, the slopes of each $f_i(M,r)$ and $F_i(M,r)$ belong to
$\frac{1}{1} \ZZ \cup \cdots \cup \frac{1}{d} \ZZ$.

\item[\emph{(c)}]
(Monotonicity)
Suppose that $0 \in I$.
Then the slopes of $F_i(M,r)$
are nonpositive, and each $F_i(M,r)$ is constant for $r$ sufficiently large.

\item[\emph{(d)}]
(Convexity)
For $i=1, \dots, d$, the function $F_i(M,r)$ is convex.

\item[\emph{(e)}] (Decomposition)
Suppose that $I$ is an open interval in $(0, +\infty)$, and suppose that, for some $i \in \{\serie{}d-1\}$, the function 
$F_i(M,r)$ is affine and  $f_i(M,r) > f_{i+1}(M,r)$ for $r \in -\log \dot I$.
Then $M$ can be uniquely written as the direct sum of two $\partial_{J^+}$-differential submodules $M_1$ and $M_2$ such that, for any $\eta \in \dot I$, the multiset of intrinsic radii of $M_1 \otimes F_\eta$ exactly consists of the smallest $i$ elements in the multiset of intrinsic radii of $M \otimes F_\eta$.

\item[\emph{(f)}] (Dichotomy)  Suppose that $I$ is an open interval in $(0, +\infty)$ and that $M$ is not the direct sum of two nonzero $\partial_{J^+}$-differential submodules.  If $f_1(M, r)$ is affine for $r \in -\log \dot I$, then, for each $j \in J^+$,
\begin{itemize}
\item[\emph{(1)}] either $M \otimes F_\eta$ has pure intrinsic $\partial_j$-radii which equal $IR(M \otimes F_\eta)$ for all $\eta \in \dot I$, or 
\item[\emph{(2)}] we have $IR_{\partial_j}(M \otimes F_\eta)>IR(M \otimes F_\eta)$ for all $\eta \in \dot I$.
\end{itemize}
\end{enumerate}

Moreover, if $|u_j| = 1$ for any $j \in J$ and if $I \subseteq [0, 1)$, then the same statements above except (c) hold for $\hat f_i(M, r)$ and $\hat F_i(M, r)$ in place of $f_i(M, r)$ and $F_i(M, r)$, respectively.  In this case, the following statement holds.
\begin{itemize}
\item [\emph{(c')}] (Monotonicity)
Suppose that $0 \in I$.  For $i=1, \dots, d$, for any point $r_0$ where $\hat f_i(M, r_0) > r_0$, the slopes of $\hat F_i(M, r)$ are nonpositive in some neighborhood of $r_0$.  We also have $\hat f_i(M, r) = r$ for $r$ sufficiently large.
\end{itemize}
\end{theorem}
\begin{proof}
Statements (a)-(e) for $f_i(M, r)$ and $F_i(M, r)$ are proved in \cite[Theorems~2.4.4~and~2.5.1]{kedlaya-xiao}.  Statements (a), (b), (c'), (d), and (e) for $\hat f_i(M, r)$ and $\hat F_i(M, r)$ can be proved similarly as follows.

Let $\widetilde K$ denote the completion of $K(x_J)$ with respect to the $(1, \dots, 1)$-Gauss norm.  For $I = [\alpha, \beta) \subseteq [0,1)$, (as in \cite[Notation~2.4.1]{kedlaya-xiao},) Taylor series defines an injective continuous homomorphism $\tilde f^*: K \langle \alpha / t, t/ \beta \}\} \rar \widetilde K \langle \alpha / t, t / \beta \}\}$ such that $\tilde f^*(u_j) = u_j + x_j t$.  
For $\eta \in (\alpha, \beta)$, we use $\widetilde F_\eta$ to denote the completion of $\widetilde K(t)$ with respect to the $\eta$-Gauss norm.  Then $\tilde f^*$ extends to an injective isometric homomorphism $\tilde f^*: F_\eta \inj \widetilde F_\eta$.

We view $\tilde f^*M$ as a $\partial_0$-differential module on $A_{\widetilde K}^1[\alpha, \beta)$.  Since $
\partial_0|_{\tilde f^* M} = \partial_0 |_M + \sum_{j \in J} x_j \partial_j|_M$, we have
\[
R_{\partial_0} (M \otimes \widetilde F_\eta) = \min_{j \in J^+} \big\{R_{\partial_j}(M \otimes F_\eta)\big\} = ER(M \otimes F_\eta), \textrm{ for any } \eta \in [\alpha, \beta).
\]
In other words, $f_i^{(0)}(\tilde f^* M, r) = \hat f_i(M, r)$ for $r \in (-\log \beta, -\log \alpha)$.  The theorem follows from Theorem~\ref{T:variation-j}; to obtain the decomposition in (e), we use Lemma~\ref{L:proj-intersect} and Remark~\ref{R:proj-intersect} to glue the decompositions over $A^1_{\widetilde K}[\alpha, \beta)$ and over $F_\eta$ for some $\eta \in (\alpha, \beta)$.

We now prove (f) for the intrinsic radii and the proof for  the extrinsic radii is similar.

Assume that we are not in case (2).  Then $IR_{\partial_j}(M \otimes F_\eta)=IR(M \otimes F_\eta)$ for some $\eta \in \dot I$.  By Theorem~\ref{T:variation-j}(d), the convexity of $f_1^{(j)}(M, r)$ forces $IR_{\partial_j}(M \otimes F_\eta)=IR(M \otimes F_\eta)$ for all $\eta \in \dot I$.  Now, if $IR_{\partial_j}(M \otimes F_\eta; 2) > IR( M \otimes F_\eta)$ for all $\eta \in (\alpha, \beta)$, the decomposition (e) would imply that $M$ is decomposable, which contradicts the assumption.  Therefore, we have $IR_{\partial_j}(M \otimes F_\eta; 2) = IR( M \otimes F_\eta)$ for some $\eta \in \dot I$.  By Theorem~\ref{T:variation-j}(d) again, we have the equality for all $\eta \in \dot I$.  Continuing this argument for the third smallest and other $\partial_j$-radii leads us to case (1).
\end{proof}

Next, we discuss how the multiset of refined $\partial_j$-radii of the $\partial_j$-differential module $M$ changes when we base change the $\partial_j$-differential module $M$ to the completions with respect to different Gauss norms, in the case when $f_1^{(j)}(M, r) = \cdots = f_{\rank M}^{(j)}(M, r)$ is \emph{affine}.
Before proving general results, we first look at an example of $\partial_j$-differential module with pure refined $\partial_j$-radii when base changed to any completion with respect to the Gauss norm.  It is a 1-dimensional family analog of Example~\ref{Ex:pure-refined}.

\begin{example}\label{Ex:pure-IR-1-dim}
Let $j\in J^+$ and let $(\alpha, \beta) \subseteq (0, \infty)$ be an open interval. 
Fix $b \in \QQ$ and $\theta \in \kappa_{K^\alg}^{(a)}$, where $a \in -\log |K^\times|^\QQ$.  Assume that
\begin{equation}\label{E:pure-IR-1-dim}
\ee^{a}\alpha^{b}, \ee^{a} \beta^{b} < \left\{
\begin{array}{ll}
1 & \textrm{if } p=0,\\
p^{1/p} & \textrm{if } p>0.
\end{array}
\right.
\end{equation}
We will see that this \emph{includes some non-visible radii}.
As noted in Remark~\ref{R:pure-refined-local}, we cannot improve the restriction from $p^{-1/p}$ to $p^{-1/(p-1)}$.

Let $e$ be the prime-to-$p$ part of the denominator of $b$.
We have the following.
\begin{itemize}
\item[(i)] If $p=0$, then $a \in -\log |(K')^\times|$ and $\theta \in \kappa_{K'}^{(a)}$ for some finite \emph{tamely ramified} extension $K'/K$.  Let $x \in \gothm_{K'}^{(a)} $ be a lift of $\theta$. We set $n=0$ and $d = 1$ in this case.
\item[(ii)] If $p>0$ and $j =0$, there exists $n \in\NN$ such that $\theta^{p^n} \in \kappa_{K'}^{(p^na)}$ with $p^na \in -\log|(K')^\times|$ and $p^neb \in p\ZZ$, for some finite \emph{tamely ramified} extension $K'/K$.  Let $x \in \gothm_{K'}^{(p^na)}$ be a lift of $\theta^{p^n}$.  We set $d=  p^n$.
\item[(ii')] If $p>0$ and $j \in J$, there exists $n \in\NN$ such that $\theta^{p^n} \in (\kappa_{K'}^{(p^{n-1}a)})^p$ and $p^neb \in \ZZ$ with $p^{n-1}a \in -\log|(K')^\times|$ for some finite \emph{tamely ramified} extension $K'/K$.  Let $x \in \gothm_{K'^{(\partial_j)}}^{(p^na)}$ be a lift of $\theta^{p^n}$; this is possible by Lemma~\ref{L:Frob-does-to-kappa}.
\end{itemize}

Let $A_{K'}^1(\alpha^{1/e}, \beta^{1/e})$ be the open annulus with coordinate $t^{1/e}$.
Let $\calL_{x, b, (n)}^{(j)}$ denote the  $\partial_j$-differential module over $A_{K'}^1(\alpha^{1/e}, \beta^{1/e})$ of rank $d$ with basis $\{\bbe_1, \dots, \bbe_d\}$, on which $\partial_j$ acts per description
\[
\partial_j \bbe_i = \bbe_{i+1} \textrm{ for }i = 1, \dots, d-1, \quad \textrm{and}\quad\partial_j \bbe_d = \left\{
\begin{array}{ll}
xt^{-db}u_j^{-d}\bbe_1, & \textrm{if }j \in J\\
xt^{-d(b+1)}\bbe_1, & \textrm{if }j =0.
\end{array}\right.
\]
We added $u_j^{-d}$ and $t^{-d}$ in the definition to balance the different normalizations on intrinsic $\partial_j$-radii.
\end{example}

\begin{lemma}\label{L:pure-IR-1-dim}
Keep the notation as in Example~\ref{Ex:pure-IR-1-dim}.  If we set $F'_{\ee^{-r}} = F_{\ee^{-r}}(t^{1/e})$, then for any $r \in (-\log \beta, -\log \alpha)$, $\calL_{x, b, (n)}^{(j)} \otimes F'_{\ee^{-r}}$ has pure intrinsic $\partial_j$-radii $\omega \ee^{a-br}$ and pure refined $\partial_j$-radii $\theta t^{-b}$.
\end{lemma}
\begin{proof}
Comparing this with Example~\ref{Ex:pure-refined} shows that for any $r$, $\calL^{(j)}_{x, b, (n)} \otimes F'_{\ee^{-r}}$ is isomorphic to $\calL_{xt^{-db}u_j^{-d}, (n)}$ if $j \in J$, and to $\calL_{xt^{-d(b+1)}, (n)}$ if $j=0$. Applying Lemma~\ref{L:pure-refined} to this $\partial_j$-differential module yields the result; note that the condition~\eqref{E:pure-IR-1-dim} corresponds to the condition on $\goths$ in Example~\ref{Ex:pure-refined}.
\end{proof}

\begin{theorem}
\label{T:var-refined}
Fix $j \in J^+$.  Let $M$ be a $\partial_j$-differential module over an \emph{open} annulus $A_K^1(\alpha, \beta)$ such that $M \otimes F_{\ee^{-r}}$ has pure intrinsic $\partial_j$-radii $\omega \ee^{a-br}< 1$ for any $r \in (-\log \beta, -\log \alpha)$ (this implies that $f_1^{(j)}(M, r) = \cdots = f_{\dim M}^{(j)}(M, r)$ is an affine function with slope $b$).  
Let $e$ be the prime-to-$p$ part of the denominator of $b$.
Then there exists a unique  direct sum decomposition
$
M = \bigoplus_{\{\mu_e\theta\} \subseteq \kappa_{K^\alg}^{(a)}} M_{\{\mu_e\theta\}}
$
of $\partial_j$-differential modules
over $A_{K}^1(\alpha, \beta)$, where the sum is taken over all $\mu_e \rtimes \Gal(K^\alg/K)$-orbits of $\kappa_{K^\alg}^{(a)}$, and the refined $\partial_j$-radii of $M_{\{\mu_e\theta\}} \otimes F_\eta$ for any $\eta \in (\alpha, \beta)$ is a multiset consisting of the $\mu_e \rtimes \Gal(K^\alg/K)$-orbit $\{t^{-b}\theta \}$ with an appropriate multiplicity.  

Moreover, if $K'$ is a finite tamely ramified tension of $K$ such that all the $\theta$'s in the above decomposition belong to $\cup_n(\kappa_{K'}^{(p^ns)} )^{1/p^n}$, then we have a unique direct sum decomposition \[
M \otimes_{K\{\{\alpha/t, t/\beta\}\}} K'\{\{\alpha^{1/e}/t^{1/e}, t^{1/e}/\beta^{1/e}\}\} = \bigoplus_{\theta \in  \kappa_{K^\alg}^{(a)}} M_{\theta}
\]
of $\partial_j$-differential modules
over $A_{K'}^1(\alpha^{1/e}, \beta^{1/e})$ such that $M_{\theta} \otimes K'F'_\eta$ has pure refined $\partial_j$-radii $t^{-b}\theta $ for any $\eta \in (\alpha, \beta)$.  
\end{theorem}
\begin{proof}
First of all, since defining a $\partial_j$-differential module only needs finite data, we may assume that $\QQ \cdot \log|K^\times| \neq \RR$.

The decomposition as stated in the theorem if exists is determined by the decomposition of $M \otimes F_{\ee^{-r}}$ for each $r \in (-\log \beta, -\log \alpha)$; it is hence unique.  We may always replace $M$ by $M \otimes_{K\{\{\alpha/t, t/\beta\}\}} K'\{\{\alpha^{1/e}/t^{1/e}, t^{1/e}/\beta^{1/e}\}\}$ for $e$ and any finite tamely ramified extension $K'$ of $K$, and we may recover the result for $M$ using Galois descent.  In particular, we may assume that $e=1$.  Moreover, using Lemma~\ref{L:proj-intersect} and Remark~\ref{R:proj-intersect}, it suffices to obtain the decomposition in a neighborhood of each radius in $(\alpha, \beta)$ and we can glue the decompositions over the overlaps.

Let $r_0 \in (-\log \beta, -\log \alpha)$ be a point. We first assume that $IR_{\partial_j}(M \otimes F_{\ee^{-r_0}}) < 1$ when $p=0$, and $IR_{\partial_j}(M \otimes F_{\ee^{-r_0}}) < p^{-1/p(p-1)}$ when $p>0$ (note that this restriction still allows some non-visible radii).  By shrinking the interval $(\alpha, \beta)$ to a smaller neighborhood of $r_0$, we may assume that the  condition above at $r_0$ holds for all points in $(-\log \beta, -\log \alpha)$.  Pick a point $r_1 \in (-\log \beta, -\log \alpha)$ which \emph{does not belong to} $\QQ \cdot \log |K^\times|$.  

Let $\theta t^{-b} \in \calI\Theta_{\partial_j}(M \otimes F_{\ee^{-r_1}})$ be an element in the multiset of refined intrinsic $\partial_j$-radii, with multiplicity $\mu$.  Since $M \otimes F_{\ee^{-r_1}}$ has pure intrinsic $\partial_j$-radii $\omega \ee^{a-br}$, we have $\theta t^{-b}\in \kappa_{F_{\ee^{-r}}^\alg}^{(a-br)} \cong t^{-b}\kappa_{K^\alg}^{(a)} $; here the latter isomorphism follows from our choice $r_1 \notin \QQ \cdot \log|K^\times|$.  we may replace $K$ by a finite tamely ramified extension so that $\theta \in \cup_n( \kappa_{K}^{(p^na)})^{1/p^n}$.  The construction in Example~\ref{Ex:pure-IR-1-dim} gives a $\partial_j$-differential module $\calL_{x, b, (n)}^{(j)}$ over $A_{K}^1(\alpha, \beta)$ such that $\calL_{x, b, (n)}^{(j)} \otimes F_{\ee^{-r}}$ has pure $\partial_j$-radii $\omega \ee^{a-br}$ and pure intrinsic $\partial_j$-radii $\theta t^{-b}$ for any $r \in (-\log \beta, -\log \alpha)$.

If we set $N = M \otimes (\calL_{x,b, (n)}^{(j)})^\dual $, then we have $IR_{\partial_j}(N \otimes F_{\ee^{-r}}) \leq \omega\ee^{a-br}$ for any $r \in (-\log \beta, -\log \alpha)$.
Moreover, Proposition~\ref{P:refined-properties} and Theorem~\ref{T:refined-decomposition} together implies that
\[
f^{(j)}_1(M, r_1) = f^{(j)}_1(N, r_1) = f^{(j)}_{(\dim M -\mu)d}(N, r_1) > f^{(j)}_{(\dim M -\mu)d+1}(N, r_1).
\]
By Theorem~\ref{T:variation}(d), the same inequality holds for all $r \in (-\log \beta, -\log \alpha)$ in place of $r_1$ because a convex function below a linear function is same as the linear function if and only if the two functions touch at some point.  By Theorem~\ref{T:variation-j}(e), we have a unique decomposition of $\partial_j$-differential module $N = N_0 \oplus N_1$ such that, for any $r \in (-\log \beta, -\log \alpha)$,  $N_0\otimes F_{\ee^{-r}}$ has pure intrinsic $\partial_j$-radii $\omega \ee^{a-br}$ and $IR_{\partial_j}(N_1 \otimes F_{\ee^{-r}}) > \omega \ee^{a-br'}$.  By the same argument as in Theorem~\ref{T:refined-decomposition}, this implies that $M$ admits a decomposition of $\partial_j$-differential modules $M  = M_\theta \oplus M'$ over $A^1_{K}(\alpha, \beta)$
such that $M_\theta \otimes (\calL_{x, b, (n)}^{(j)})^\dual = N_1$ and $M' \otimes (\calL_{x, b, (n)}^{(j)})^\dual = N_0$.  By Proposition~\ref{P:refined-properties-general} and Lemma~\ref{L:pure-IR-1-dim},  for any $r \in (-\log \beta, - \log \alpha)$, $M_\theta \otimes F_{\ee^{-r}}$ has pure refined intrinsic $\partial_j$-radii $\theta t^{-b}$, and the multiset of refined intrinsic $\partial_j$-radii of $M'\otimes F_{\ee^{-r}}$ does not contain $\theta t^{-b}$.  We obtain the decomposition asked in the theorem by applying the above argument to every $\theta$.

To finish the proof,  it suffices to consider the case when $p>0$ and $IR_{\partial_j}(M \otimes F_{\ee^{-r}}) \in [p^{-1/p(p-1)}, 1)$.   But in this case, the $\partial_j$-Frobenius antecedent of $M$ exists over the annulus with radii in a neighborhood of $r$.  The decomposition follows from the decomposition of the $\partial_j$-Frobenius antecedents of $M$ (applied iteratively until the intrinsic $\partial_j$-radii fall in the range above).
\end{proof}

\begin{remark}\label{R:r-not-in-valuation}
The artificial reduction to the case $\QQ \cdot \log |K^\times| \neq \RR$ is to deduce $\theta \in \kappa_{K^\alg}^{(a)}$.  This fact can also be proved using Newton polygons if the $f_1^{(j)}(M,r)$ is not constantly $p^{-1/(p-1)}$, in which case, one may alternatively use Frobenius pushforward to reduce to the visible case.
\end{remark}

\begin{theorem}
\label{T:var-refined-multi}  Let $I$ be an open interval of $[0, +\infty)$ and let $M$ be a $\partial_{J^+}$-differential module over $A_K^1(I)$ such that $M \otimes F_{\ee^{-r}}$ has pure intrinsic radii $\omega \ee^{a-br} < 1$ for $r \in -\log(\dot I)$. Let $e$ denote the prime-to-$p$ part of the denominator of $b$. 
Then there exists a unique direct sum decomposition
$
M = \bigoplus_{\{\mu_e\vartheta\}} M_{\{\mu_e\vartheta\}}
$
of $\partial_{J^+}$-differential modules
over $A_{K}^1(I)$, where the sum is taken over all $\mu_e \rtimes \Gal(K^\alg/K)$-orbits of $\oplus_{j \in J} \kappa_{K^\alg}^{(a)} \frac{du_j}{u_j} \oplus \kappa_{K^\alg}^{(a)} \frac{dt}t$, and the refined intrinsic radii of $M_{\{\mu_e\vartheta\}} \otimes F_\eta$ for any $\eta \in -\log \dot I$ is a multiset consisting of the $\mu_e \rtimes \Gal(K^\alg/K)$-orbit $\{ t^{-b}\vartheta\}$ with an appropriate multiplicity.  

Moreover, there exists a finite tamely ramified tension $K'$ of $K$ such that we have a unique direct sum decomposition
\begin{equation}\label{E:var-refined-all}
M \otimes_{K[t]} K'[t^{1/e}] = \bigoplus_{\vartheta \in  \oplus_{j \in J} \kappa_{K^\alg}^{(a)} \frac{du_j}{u_j} \oplus \kappa_{K^\alg}^{(a)} \frac{dt}t} M_{\vartheta}
\end{equation}
of $\partial_{J^+}$-differential modules
over $A_{K'}^1(I^{1/e})$ such that $M_{\vartheta} \otimes K'F'_\eta$ has pure refined intrinsic radii $t^{-b}\vartheta$ for any $\eta \in -\log \dot I$.
\end{theorem}
\begin{proof}
We first treat the case when $0 \notin I$.
Without loss of generality, we assume that $M$ is not a direct sum of two nonzero sub-$\partial_{J^+}$-modules, which implies the dichotomy given by Theorem~\ref{T:variation}(f).  We may apply Theorem~\ref{T:var-refined} to the $\partial_j$ for which case (f1) of Theorem~\ref{T:variation} holds for $M$ and note that the decompositions for different $\partial_j$'s given by Theorem~\ref{T:var-refined} are compatible.  This gives rise to the desired decomposition.

Now, we consider the case when $I = [0, \beta)$.  Since we have already proved the theorem over $(\alpha, \beta)$ for any $\alpha>0$, it suffices to find the decomposition for $I = [0, \alpha)$ for some $\alpha \in (0,1)$.  Note that when $\alpha$ is sufficiently small, $M \otimes A^1_K[0, \alpha)$ is  trivial as a $\partial_0$-differential module and hence is the pullback of a $\partial_J$-differential module $M_0$ over $K$ along the natural morphism $K \to K\{\{t/\alpha\}\}$.  The decomposition~\ref{E:var-refined-all} follows from the decomposition of $M_0$ given by Theorem~\ref{T:refined-decomposition}.
\end{proof}

We have the similar result for refined extrinsic radii, but only over $A^1_K(I)$; this is because adjoining $t^{1/e}$ would change the extrinsic radii.  This is also subtlety when considering differential modules over discs (as oppose to annuli) and trying to extend the decomposition into the center of the disc;  this is only possible if the functions defined by the extrinsic radii are  ``constant".

\begin{theorem}
\label{T:nonlog-refined-decomposition}
Assume that $|u_j|=1$ for all $j \in J$.
Let $M$ be a $\partial_{J^+}$-differential module over an open \emph{annulus} $A_K^1(I)$ with $I \subseteq (0, 1)$.  Assume that $M \otimes F_{\ee^{-r}}$ has pure extrinsic radii $\omega \ee^{a-br} < \ee^{-r}$ for $r \in -\log (\dot I)$. Let $e$ denote the prime-to-$p$ part of the denominator of $b$.  Then there exists a unique  direct sum decomposition 
\begin{equation}
\label{E:var-refined-all-extrinsic}
M = \bigoplus_{\{\mu_e \hat\vartheta\}} M_{\{\mu_e \hat\vartheta\}}
\end{equation}
of $\partial_{J^+}$-differential modules over $A_K^1(I)$,
 where the direct sum is taken over all $\mu_e \rtimes\Gal(K^\alg/K)$-orbits $\{\mu_e\hat \vartheta\}$ in $\oplus_{j \in J} \kappa_{K^\alg}^{(a)} du_j \oplus \kappa_{K^\alg}^{(a)} dt$, and the multiset of refined extrinsic radii of $M_{\{\mu_e\hat\vartheta\}} \otimes F_\eta$ exactly consists of the  $\mu_e \rtimes\Gal(K^\alg/K)$-orbit $\{t^{-b}\mu_e\hat \vartheta\}$ with an appropriate multiplicity, for any $\eta \in \dot I$.
\end{theorem}
\begin{proof}
The proof is the same as Theorem~\ref{T:var-refined-multi}.
\end{proof}

\begin{proposition}
\label{P:refined-disc-j}
Fix $j \in J^+$.  Let $M$ be a $\partial_j$-differential module over an open disc $A_K^1[0, \alpha)$ such that $M \otimes F_\eta$ for any $\eta$ in a neighborhood of $\eta = \alpha$ has pure $\partial_j$-radii $\omega \ee^s$, where $\omega \ee^s$ is independent of $\eta$, and is strictly less than $|u_j|$ if $j \in J$ and less than $\alpha$ if $j=0$.  Then there exists a unique direct sum decomposition $M  = \bigoplus_{\{\theta\} \subset \kappa_K^{(s)}} M_{\{\theta\}}$ of $\partial_j$-differential modules over $A_{K}^1[0, \alpha)$, where the direct sum is taken over all $\Gal(K^\alg/K)$-orbits $\{\theta\}$ of $\kappa_K^{(s)}$, and the multiset of refined $\partial_j$-radii of $M_{\{\theta\}} \otimes F_\eta$ consists of the $\Gal(K^\alg/K)$-orbits $\{\theta\}$ with appropriate multiplicity, for any $\eta \in (0, \alpha)$ if $j \in J$ and for any $\eta \in (\omega \ee^s, \alpha)$ if $j=0$.
\end{proposition}
\begin{proof}
 Theorem~\ref{T:variation-j}(c) implies that $M \otimes F_\eta$ has pure $\partial_j$-radii $\omega \ee^s$, for any $\eta \in (0, \alpha]$ if $j \in J$ and for any $\eta \in (\omega \ee^s, \alpha]$ if $j=0$.
The proposition then follows from the same argument as in Theorem~\ref{T:var-refined}, but invoking \cite[Theorem~2.3.10]{kedlaya-xiao} in place of Theorem~\ref{T:variation-j}(e) when making the the decomposition by extrinsic radii.  Note also that we will only make use of the $\partial_j$-differential module $\calL_{x,0,(n)}^{(j)}$ in the proof which is defined over the entire disc $A_K^1[0, \alpha)$.
\end{proof}

\begin{proposition}
\label{P:nonlog-refined-disc}
Assume that $|u_j|=1$ for any $j \in J$.
Let $M$ be a $\partial_{J^+}$-differential module over an open disc $A_K^1[0, \alpha)$ with $\alpha <1$.  Assume that $M \otimes F_{\ee^{-r}}$ has pure extrinsic radii $\min\{\omega \ee^{s}, \ee^{-r}\}$ for any $r >-\log \alpha$, where $\omega \ee^s < \alpha$.   Then there exists a unique direct sum decomposition 
$
M = \bigoplus_{\{ \hat\vartheta\}} M_{\{\hat\vartheta\}}
$
of $\partial_{J^+}$-differential modules over $A_K^1[0, \alpha)$,
 where the direct sum is taken over all $\Gal(K^\sep/K)$-orbits  $\{\hat \vartheta\}$ in $\oplus_{j \in J} \kappa_{K^\alg}^{(s)} du_j \oplus \kappa_{K^\alg}^{(s)} dt$, such that the multiset of refined extrinsic radii of $M_{\{\hat\vartheta\}} \otimes F_\eta$ exactly consists of the $\Gal(K^\sep/K)$-orbits $\{\hat\vartheta\}$ with appropriate multiplicity, for any $\eta > \omega \ee^s$.
\end{proposition}
\begin{proof}
Without loss of generality, we assume that $M$ is not a direct sum of two nonzero $\partial_{J^+}$-differential modules.
We first show a dichotomy (similar to Theorem~\ref{T:variation}) that for each $\partial_j$, either $M \otimes F_\eta$ has pure $\partial_j$-radii $\omega \ee^s$ for all $\eta > \omega \ee^s$, or $R_{\partial_j}(M \otimes F_\eta)< \omega \ee^s$ for all $\eta > \omega \ee^s$.  Assume that we are not in the latter case.  Then we have $R_{\partial_j}(M \otimes F_\eta) = ER(M \otimes F_\eta)$ for some $\eta \in (\omega \ee^s, \alpha)$.  By Theorem~\ref{T:variation-j}(c)(d), the monotonicity and the convexity of $f_1^{(j)}(M, r)$ forces $R_{\partial_j}(M \otimes F_\eta) = ER(M \otimes F_\eta)$ for all $\eta \in (0, \alpha)$.  Now, if $R_{\partial_j}(M \otimes F_\eta, 2) > ER(M \otimes F_\eta)$ for all $\eta \in (\omega \ee^s, \alpha)$, we may use \cite[Theorem~2.3.10]{kedlaya-xiao} to decompose $M$ to split off the smallest $\partial_j$-radii, which contradicts the indecomposability assumption on $M$.  Therefore, $R_{\partial_j}(M \otimes F_\eta, 2) = ER(M \otimes F_\eta)$ for some $\eta \in (\omega \ee^s, \alpha)$.  Continuing this argument for the third and other $\partial$-radii leads us to the former case of the claim.  The proposition now follows from applying Proposition~\ref{P:refined-disc-j} to each $\partial_j$ that satisfies the former condition of the claim.
\end{proof}

\begin{remark}
We do not expect a decomposition theorem analogous to Proposition~\ref{P:nonlog-refined-disc} in the case when the functions for extrinsic radii are linear with negative slopes.  One of the reason is that we do not know how to construct modules $\calL_{x, b, (n)}^{(j)}$ over the open disc.  A more serious reason is that, when $\eta$ is sufficiently close to $0$, $M \otimes F_\eta$ is always the same as $\eta$, and hence no information of the $\partial_j$-radii of $M \otimes F_\eta$ is reflected in the extrinsic radii; in contrast, if the functions of extrinsic radii stay constant before the they become equal to $-\log \eta$, all dominant $\partial_j$ must have constant $\partial_j$-radii by the monotonicity (Theorem~\ref{T:variation-j}(c)).
\end{remark}

\subsection{Refined differential conductors}
\label{S:differential-conductor}

Differential modules defined over an open annulus with outer radius 1 are historically considered very important, in particular those whose intrinsic radii approach to 1, as we base change to the completion with respect to the Gauss norms with radii approaching to 1; this is known as the solvable case.  In particular, the rate of the such change of intrinsic radii is related to the Swan conductors if the differential modules come from a Galois representation of $G_{\FF_p((t))}$.  In this subsection, we focus on this situation and define differential conductors, as well as refined differential conductors if the differential module has pure differential conductors.

We continue to assume Hypothesis~\ref{H:K-multi}.  Moreover, we assume $p>0$ in this subsection.

\begin{definition}
\label{D:solvable-module}
Let $M$ be a $\partial_{J^+}$-differential module of rank $d$ over $A^1_K(\eta_0, 1)$ for some $\eta_0 \in (0, 1)$.  We say that $M$ is \emph{solvable} if $IR(M \otimes F_{\eta}) \to 1$ as $\eta \rar 1^-$.
\end{definition}

\begin{theorem} \label{T:swan1}
Let $M$ be a solvable $\partial_{J^+}$-differential module of rank $d$ over $A^1_K(\eta_0,1)$
for some $\eta_0 \in (0,1)$.
Then by making $\eta_0$ sufficiently close to $1$, there exists a unique direct sum decomposition
$M = M_1 \oplus \cdots \oplus M_\gamma$ over $A^1_K(\eta_0,1)$
and  nonnegative distinct
rational numbers $b_1,\dots,b_\gamma$ with $b_i \cdot \rank(M_i) \in \ZZ$,
such that $M_i \otimes F_\eta$ has pure intrinsic radii $\eta^{b_i}$ for any $i = 1, \dots, \gamma$ and any $\eta \in (\eta_0, 1))$.

Keep the same hypothesis and assume moreover that $|u_j| = 1$ for all $j \in J$.  Then by making $\eta_0$ sufficiently close to $1$, there exists a unique direct sum decomposition
$M = \hat M_1 \oplus \cdots \oplus \hat M_{\hat \gamma}$ over $A^1_K(\eta_0,1)$
and nonnegative distinct
rational numbers $\hat b_1,\dots,\hat b_{\hat \gamma}$ with $\hat b_i \cdot \rank(\hat M_i) \in \ZZ$,
such that $\hat M_i \otimes F_\eta$ has pure extrinsic radii $\eta^{\hat b_i}$ for any $i=1,\dots,\hat \gamma$ and any $\eta \in (\eta_0, 1))$.
\end{theorem}
\begin{proof}
By Theorem~\ref{T:variation}(a)(b)(d),
for $l=1,\dots,d$, the functions $d! F_l(M, r)$ and $d! \hat F_l(M, r)$ on $(0, -\log \eta_0)$
are continuous, convex, and piecewise affine with integer slopes.
The assumption $d! F_l(M,r) \to 0$ also implies that $d! \hat F_l(M, r) \rar 0$ as $r \to 0^+$; because of this
and the fact that $d! F_l(M,r) \geq 0$ and $d! \hat F_l(M, r) \geq 0$ for all $r$, the slopes of
$F_l(M,r)$ and $\hat F_l(M, r)$ are forced to be nonnegative. Hence there is a least such
slope; that is, $d! F_l(M,r)$ and $d! \hat F_l(M, r)$ are linear in a right neighborhood of $r=0$.

We can thus choose $\eta_0 \rar 1^-$ so that $d! F_l(M,r)$ and $d! \hat F_l(M, r)$ are linear on
$(0, -\log  \eta_0)$ for $l=1,\dots,d$.  The desired decompositions is constructed in  Theorem~\ref{T:variation}(e) and 
the integrality of $b_i \cdot \rank(M_i)$ and $\hat b_i \cdot \rank (\hat M_i)$ follows from the fact that
$F_{\dim M_i}(M_i,r)$ and $\hat F_{\dim \hat M_i} (\hat M_i, r)$ have integral slopes, again by Theorem~\ref{T:variation}(b).
\end{proof}

\begin{definition}
\label{D:differential-conductor}
Let $M$ be a solvable $\partial_{J^+}$-differential module of rank $d$ over $A_K^1(\eta_0, 1)$ for some $\eta_0 \in (0, 1)$.  Define the multiset of \emph{differential log-breaks} of $M$ to be the multiset consisting of $b_i$ from Theorem~\ref{T:swan1}  with multiplicity $\rank M_i$; we use $b_\log(M; 1) \geq \cdots \geq b_\log(M; d)$ to denote the differential log-breaks in decreasing order.  We define the \emph{differential Swan conductor} of $M$ to be the sum of the differential log-breaks, that is $\Swan(M) = \sum_{i=1}^r b_i\cdot \rank(M_i)$; it is a nonnegative integer by Theorem~\ref{T:swan1}.  We say that $M$ has \emph{pure differential log-breaks} if all differential log-breaks are equal.

When $M$ has pure differential log-breaks, we define the multiset of \emph{refined Swan conductors} of $M$, denoted by $\calI\Theta(M)$, to be the multiset consisting of $\vartheta$ in \eqref{E:var-refined-all} with multiplicity $\rank M_\vartheta$.

Similarly, when $|u_j| = 1$ for all $j \in J$, we define multiset of \emph{differential nonlog-breaks} to be the multiset consisting of $\hat b_i$ from Theorem~\ref{T:swan1}  with multiplicity $\rank \hat M_i$; we use $b_\nlog(M; 1) \geq \cdots \geq b_\nlog(M; d)$ to denote the differential nonlog-breaks in decreasing order.  We define the \emph{differential Artin conductor} of $M$ to be the sum of the differential nonlog-breaks; it is also a nonnegative integer by Theorem~\ref{T:swan1}.   We say that $M$ has \emph{pure differential nonlog-breaks} if all differential nonlog-breaks are equal. 

When $M$ has pure differential nonlog-breaks, we define the multiset of \emph{refined Artin conductors} of $M$, denoted by $\calE\Theta(M)$, to be the multiset of $\mu_e \rtimes\Gal(K^\sep/K)$-orbits $\{\mu_e\hat \vartheta\}$ in \eqref{E:var-refined-all-extrinsic} with  multiplicity equal to the multiplicities of $\{t^{-b}\mu_e\hat \vartheta\}$ in $M_{\{\mu_e\hat \vartheta\}} \otimes F_{\eta}$ for any $\eta \in (\eta_0, 1)$.
\end{definition}

\section{Refined differential conductors for Galois representations}

One of the most important application of $p$-adic differential modules is to provide an interpretation of the Swan conductors of representations of $G_k$, where $k$ is a complete discrete valuation field of equal characteristic $p>0$ with perfect residue field.  This idea was later generalized by Kedlaya \cite{kedlaya-swan1} to the case when the residue field of $k$ need not to be perfect, and by the author \cite{xiao1} to relate the differential modules to the Swan conductors in the sense of Abbes and Saito \cite{abbes-saito1}.  In this section, we further develop the theory on the differential module side to incorporate the study of refined differential conductors, which will be related to Saito's definition \cite{saito-wild-ram} of refined Swan conductors, as proved in the next section. 

Throughout this section, we assume that $p>0$ is a prime number.

\subsection{Construction of differential modules}
\label{S:construction-diff-mod}

This subsection is dedicated to the construction of the differential modules associated to representations of $G_k$, where $k$ is a complete discrete valuation field of equal characteristic $p>0$.

\begin{definition}
\label{D:p-basis}
For a field $\kappa$ of characteristic $p>0$, a \emph{$p$-basis} of $\kappa$ is a set $(b_j)_{j \in J} \subset \kappa$ such that  the products $b_J^{e_J}$, where $e_j \in \{0, 1, \dots, p-1\}$ for all $j \in J$ and $e_j = 0$ for all but finitely many $j$, form a basis of the vector space $\kappa$ over $\kappa^p$.
\end{definition}

\begin{notation}\label{N:k}
Let $k$ be a complete discrete valuation field of characteristic $p>0$.  Let $\pi_k$ be a uniformizer of $k$, generating the maximal ideal $\gothm_k$ in the ring of integers $\calO_k$. Let $\kappa= \kappa_k$ denote the residue field.  Let $\bar \kappa = \kappa^\alg$ denote an algebraic closure of $\kappa$.   We choose and fix a non-canonical isomorphism $k \simeq \kappa((\pi_k))$.  We fix a $p$-basis $\bar b_J$ of $\kappa$ and let $b_J$ be the preimage of them via the isomorphism above.  Then $\{b_J, \pi_k\}$ form a $p$-basis of $k$, which  we refer to as a \emph{lifted $p$-basis}.
Let $k_0 = \cap_{n \in \NN} \kappa^{p^n} = \cap_{n \in \NN} k^{p^n}$.  We know that $d\pi_k$ and $db_J$ form a basis of $\Omega_{\calO_k}^1$ over $\calO_k$.

Let $\calO_K$ denote the Cohen ring of $\kappa$ with respect to $\bar b_J$ and let $B_J \subset \OK$ be the canonical lifts of the $p$-basis.  Put $K = \Frac \calO_K$. We use $\calO_{K_0}$ to denote the ring of Witt vectors of $k_0$, viewed as a subring of $\calO_K$ and we put $K_0 = \calO_{K_0} [\frac 1p]$.
\end{notation}

\begin{notation}
For an extension $k'/k$ of complete discrete valuation field, the \emph{(na\"ive) ramification degree} of $k'/k$ is simply the index of the valuation of $k$ in that of $k'$.

We say that $k'/k$ is \emph{tamely ramified} if $p \nmid e$ and the residue field extension $\kappa_{k'}/\kappa_k$ is separable, that is $\kappa_{k'}$ is algebraic and separable over $\kappa_k(x_\alpha; \alpha \in \Lambda)$ for some transcendental elements $x_\alpha$ and an index set $\Lambda$.  If moreover, $e=1$, we say $k'/k$ is \emph{unramified}.
\end{notation}

\begin{notation} \label{N:representation}
By a \emph{representation} of $G_k$, we mean a continuous homomorphism $\rho: G_k \rar \GL(V_\rho)$, where $V_\rho$ is a vector space over a (topological) field $F$ of characteristic zero.  We say that $\rho$ is a \emph{$p$-adic} if $F$ is a finite extension of $\Qp$.

Let $F$ be a finite extension of $\Qp$, let $\calO$ denote its ring of integers, and let $\Fq$ denote the residue field of $\calO$, where $q$ is a power of $p$. Put $\ZZ_q=W(\Fq)$  and $\QQ_q = \ZZ_q[\frac{1}{p}]$.  By an \emph{$\calO$-representation} of $G_k$, we mean a continuous homomorphism $\rho: G_k \rar \GL(\Lambda_\rho)$ with $\Lambda_\rho$ a finite free $\calO$-module.

For $\rho$ a $p$-adic representation or an $\calO$-representation, we say that $\rho$ has finite local monodromy if the image of the inertia group $I_k$ is finite.

We assume that $\Fq \subseteq k_0$.  Put $K' = KF$.  Since $F/\QQ_q$ is totally ramified, we have $\calO_{K'} \cong \calO_K \otimes_{\ZZ_q} \calO$.
\end{notation}

\begin{notation}\label{N:K}
We put $\calR_{K'}^\eta = K' \langle \eta/T,T \}\}$ for $\eta \in (0,1)$ and put
$\calR_{K'} = \cup_{\eta \in (0,1)} \calR_{K'}^\eta$; the latter ring is commonly called the \emph{Robba ring} over $K'$.  Let $\calR_{K'}^\inte$ be the subring of $\calR_{K'}$ consisting of elements whose $1$-Gauss norm is bounded by 1; it is a Henselian discrete valuation ring, with residue field $k$ if we identify the reduction of $T$ with $\pi_k$.  For $\eta \in (0,1)$, we use $F'_\eta$ to denote the completion of $K'(T)$ with respect to the $\eta$-Gauss norm.

A \emph{Frobenius lift} $\phi$ is an endomorphism of $\calR_{K'}^\inte$ which lifts the natural $q$-th power Frobenius on $k$. Any Frobenius lift extends by continuity to an action on $\calR_{K'}$.  A \emph{standard Frobenius lift} is a Frobenius lift which sends $T$ to $T^p$ and $B_j$ to $B_j^p$ for any $j \in J$.

The differentials $\Omega^1_{\calR_{K'}^\inte}$, $\Omega^1_{\calR_{K'}}$ and $\Omega^1_{\calR_{K'}^\eta}$ for any $\eta \in (0,1)$ admit a basis given by $dB_J$ and $dT$.  We set $\partial_0 = \partial / \partial T, \partial_j = \partial / \partial B_j$ with $j \in J$ for the dual basis.  Then a $\nabla$-module over $\calR_{K'}$ is just a $\partial_{J^+}$-differential module.
\end{notation}

\begin{definition}
Let $\phi$ be a Frobenius lift. Let $R = \calR_{K'}$, $\calR_{K'}^\eta$, or $\calR_{K'}^\mathrm{int}$.
A \emph{$(\phi, \nabla)$-module} $M$ over $R$ is a $\partial_{J^+}$-differential module together with an isomorphism $\Phi: \phi^*M \rar M$ of $\partial_{J^+}$-differential modules.
\end{definition}

\begin{theorem}\label{T:equivalence-categories}
For any Frobenius lift $\phi$, we have an equivalence of categories between the category of $\calO$-representations with finite local monodromy and the category of $(\phi, \nabla)$-modules over $\calR_{K'}^\inte$.  Moreover, all $(\phi, \nabla)$-modules can be realized over $\calR_{K'}^\eta$ for some $\eta \in (0, 1)$.  This $(\phi, \nabla)$-module is independent of the choice of the $p$-basis.
\end{theorem}
\begin{proof}
The functor is constructed in \cite[Section~3]{kedlaya-swan1} or \cite[Subsection~2.2]{xiao1}.
\end{proof}

\begin{definition}
For a $p$-adic representation $\rho$ of $G_k$ with finite local monodromy, we choose an $\calO$-lattice $\Lambda_\rho$ of $V_\rho$, stable under the action of $G_k$; this gives an $\calO$-representation of $G_k$.   Theorem~\ref{T:equivalence-categories} then produces a $(\phi, \nabla)$-module over $\calR_{K'}^\inte$, whose base change to $\calR_{K'}$ is called the \emph{differential module associated to} $\rho$, denoted by $\calE_\rho$.  This $\calE_\rho$ does not depend on the choice of the lattice $\Lambda_\rho$.
\end{definition}

For the rest of this subsection, we assume the following.
\begin{hypothesis}
\label{H:J-finite-set}
The residue field $\kappa$ has a finite $p$-basis $\bar b_J$, where $J = \{1, \dots, m\}$.  We put $J^+=J \cup \{0\}$.
\end{hypothesis}

\begin{proposition}
\label{P:standard-Frob}
Let $\phi$ be the standard Frobenius lift on $\calR_{K'}^\inte$.  Then the Frobenius $\phi: F'_{\eta^q} \rar F'_{\eta}$ is the same as the iterative Frobenius $\varphi^{(\partial_0, \lambda)} \circ \cdots \circ\varphi^{(\partial_m, \lambda)}$ in Construction~\ref{Cstr:j-frobenius}, where $q = p^\lambda$.
\end{proposition}
\begin{proof}
We may assume that $K'$ contains $\zeta_q$, a $q$-th root of unity.
It suffices to show that the image $\phi(F'_{\eta^q})$ is stable under the action of $(\ZZ / q \ZZ)^{m+1}$ in the sense of Construction~\ref{Cstr:j-frobenius}, where each $\partial_j$-Frobenius corresponds to a factor $\ZZ / q \ZZ$, and that the degree of $F'_\eta$ over $\phi(F'_{\eta^q})$ is $q^{m+1}$.

For $\underline i = (i_0, \dots, i_m) \in (\ZZ/q\ZZ)^{m+1}$, we have $T^{(\underline i)} = \zeta_q^{i_0}T$ and $(B_j)^{(\underline i)} = \zeta_q^{i_j}B_j$ for any $j \in J$.
Hence $(\cdot)^{(\underline i)} \circ \phi$ for all $\underline i$ are continuous homomorphisms from $\calO_K\llbracket T \rrbracket$ to itself, sending $B_j$ to $B_j^q$ and $T$ to $T^q$. By the functoriality of Cohen rings (e.g., \cite[Proposition~2.1.8]{xiao1}), these homomorphisms are all the same.  Hence the image of $\phi$ is stable under the $(\ZZ / q \ZZ)^{m+1}$-action.
It is evident that $F'_\eta$ has rank $q^{m+1}$ over $\phi(F'_{\eta^q})$; this forces the two homomorphisms to be the same.
\end{proof}

\begin{proposition}
\label{P:calE-solvable}
Let $\phi$ be the standard  Frobenius lift on $\calR_{K'}^\inte$ and let
$\calE$ be a $(\phi, \nabla)$-module over $A_{K'}^1[\eta_0, 1)$ for some $\eta_0 \in (0,1)$.  Then $\calE$ is solvable.
\end{proposition}
\begin{proof}
This is well-known to the experts; we include a proof for the convenience of the reader.
By Lemma~\ref{L:frob-properties}(a), we have
$$
f_i(\phi^*M, r) = \max \big \{ p^{-\lambda} f_i(M, qr), p^{1-\lambda} (f_i(M, qr) - \log p), \dots, f_i(M, qr) - \lambda\log p \big\},
$$
 where $\lambda = \log_p q$.  Since $\phi^*M \isom M$, the function $g_i(M) = \limsup_{r \rar 0^+} f_i(M ,r)$ satisfies
\[
g_i(M) = \max \big \{ p^{-\lambda} g_i(M), p^{1-\lambda} (g_i(M) - \log p), \dots, g_i(M) - \lambda\log p \big\}.
\]
This forces $g_i(M)$ to be zero.  By the continuity of $f_i(M, r)$ and the convexity of $F_i(M, r)$ in Theorem~\ref{T:variation}, $\lim_{r \rar 0^+} f_i(M, r) = 0$.  In other words, $\calE$ is solvable.
\end{proof}

\begin{proposition}
\label{P:Frobenius-equivalent}
Let $\phi$ be the standard Frobenius lift and let $\phi'$ be another Frobenius lift on $\calR_{K'}^\inte$.  Assume that $\calE$ is a $(\phi, \nabla)$-module over $A_{K'}^1[\eta_0, 1)$ for some $\eta_0 \in (0,1)$.  Then the restriction of $\calE$ to $A_{K'}^1[\eta, 1)$ for some $\eta \in [\eta_0, 1)$ is naturally equipped with a $(\phi', \nabla)$-module structure.
\end{proposition}
\begin{proof}
Define the Frobenius structure for $\phi'$ by Taylor series as follows.  For $\bbv \in \calE$,
\[
\phi'(\bbv) = \sum_{e_{J^+}=0}^\infty \frac{(\phi'(T) - \phi(T))^{e_0} \prod_{j \in J} (\phi'(B_j) - \phi(B_j))^{e_j}}{(e_{J^+})!}  \phi \Big( \frac{\partial^{e_0}}{\partial T^{e_0}} \frac{\partial^{e_1}}{\partial B_1^{e_1}} \cdots \frac{\partial^{e_m}}{\partial B_m^{e_m}} (\bbv) \Big).
\]

Since $|\phi'(T) - \phi(T)|_1 <1$ and $|\phi'(B_j) - \phi(B_j)|_1 < 1$ for all $j \in J$, we have the same inequality using $\eta$-Gauss norm when $\eta \in [\eta'_0, 1]$ for some $\eta'_0$ sufficiently close to 1.  Hence the expression for $\phi'$ converges on $A_{K'}^1[\eta'_0, 1)$ and gives the restriction of $\calE$ to $A_{K'}^1[\eta'_0, 1)$ a structure of $(\varphi', \nabla)$-module.
\end{proof}

\begin{remark}
One may also approach the results of this subsection without referring to the standard Frobenius but instead using a generalized version of Lemma~\ref{L:frob-properties}(a) for non-centered Frobenius.  This point of view is taken in \cite[Chap.~17]{kedlaya-course}.
\end{remark}

\subsection{Differential conductors}
\label{S:diff-conductors}

Combining the results from Subsection~\ref{S:differential-conductor} and Proposition~\ref{P:calE-solvable}, we can define differential conductors for a representation of $G_k$ with finite local monodromy.  To make this definition more robust, we will introduce the break with respect to each element of the $p$-basis, and the break of the differential module is just the maximum among all breaks for each element of the $p$-basis, after appropriate normalization. This point of view is in particular useful when we try to understand how the conductors change when restricting a Galois representation to $G_l$ for some (explicit) finite extension $l$ of $k$.

\begin{definition}\label{D:conductors-repn}
For a $p$-adic representation $\rho$ of $G_k$ with finite local monodromy, let $l$ be the extension of $k$ corresponding to $\Ker \rho$ via Galois theory.  We may choose a $p$-basis $\{c_J, \pi_l\}$ of $l$ such that $\pi_l$ is an uniformizer and $c_J \subset \calO_l^\times$, and such that $c_{J \bs J_0} \subset \calO_k$ for some finite subset $J_0 \subset J$.  If we use $k^\wedge$ to denote the completion of $k(c_{J \bs J_0}^{1/p^n}; n \in \NN)$, then $k^\wedge$ verifies Hypothesis~\ref{H:J-finite-set}.  We define the \emph{nonlog-breaks} (resp. \emph{log-breaks}) of $\rho$ to be those of $\rho|_{G_{k^\wedge}}$.  Their sums are called the \emph{Artin (resp. Swan) conductors} of $\rho$, denoted by $\Art(\rho)$ (resp. $\Swan(\rho)$).  These do not depend on the choice of the $p$-basis or of $J_0$, by \cite[Proposition~2.6.6]{kedlaya-swan1}.
\end{definition}

\begin{definition}
Put $\Fil^0G_k = G_k$ and $\Fil^aG_k = I_k$ for $a \in (0, 1]$.  For $a > 1$, let $R_a$ be the set of finite image representations $\rho$ with nonlog-break strictly less than $a$.  Put $\Fil^aG_k = \bigcap _{\rho \in R_a} \big(I_k \cap \ker (\rho)\big)$ and set $\Fil^{a+}G_k$ to be the closure of $\cup_{b > a} \Fil^b G_k$. This defines a  filtration on $G_k$ such that for any representation $\rho$ with finite image, $\rho(\Fil^aG_k)$ is trivial if and only if $\rho \in R_a$.

Similarly, put $\Fil_\log^0G_k = G_k$.  For $a > 0$, let $R_{a, \log}$ be the set of finite image representations $\rho$ with log-break less than $a$.  Put $\Fil_\log^aG_k = \bigcap _{\rho \in R_{a, \log}} \big(I_k \cap \ker (\rho)\big)$ and set $\Fil_\log^{a+}G_k$ to be the closure of $\cup_{b > a} \Fil_\log^bG_k$. This defines a   filtration on $G_k$ such that for any representation $\rho$ with finite image, $\rho(\Fil_\log^aG_k)$ is trivial if and only if $\rho \in R_{a, \log}$.

For a finite Galois extension $l$ of $k$, the above filtrations induce filtrations on the Galois group $G_{l/k}$ by $G_{l/k, (\log)}^a =  G_l\Fil^a_{(\log)}G_k / G_l$ and $G_{l/k, (\log)}^{a+} =  G_l \Fil^{a+}_{(\log)}G_k/ G_l$, for $a \geq 0$.  We define the \emph{(log-)ramification breaks} of the extension $l/k$ to be the numbers $b$ for which $G_{l/k (,\log)}^b \neq G_{l/k (,\log)}^{b+}$.  We order them as $b_{\mathrm{(n)log}}(l/k) = b_{\mathrm{(n)log}}(l/k; 1) \geq b_{\mathrm{(n)log}}(l/k; 2) \geq \cdots$.  In particular, if $\rho$ is a \emph{faithful} representation of $G_{l/k}$, we have $b_{\mathrm{(n)log}}(\rho) = b_{\mathrm{(n)log}}(l/k)$. 
\end{definition}

\begin{theorem}\label{T:properties-diff-conductor}
The differential conductors satisfy the following properties:

\emph{(a)} For any representation $\rho$ of finite local monodromy, 
\begin{eqnarray*}
\Art(\rho) & = & \sum_{a \in \QQ_{\geq 0}} a \cdot \dim \big( V_\rho^{\Fil^{a+} G_k} \big/ V_\rho^{\Fil^a G_k} \big) \in \ZZ_{\geq 0}, \\
\Swan(\rho) &=& \sum_{a \in \QQ_{\geq 0}} a \cdot \dim \big( V_\rho^{\Fil_\log^{a+} G_k} \big/ V_\rho^{\Fil_\log^a G_k}  \big)\in \ZZ_{\geq 0}.
\end{eqnarray*}

\emph{(b)} Let $k'/k$ be a (not necessarily finite) extension of complete discretely valued fields.  If $k'/k$ is unramified, then $\Fil^a G_{k'} = \Fil^a G_k$ for $a>0$. If $k'/k$ is tamely ramified with na\"ive ramification index $e < \infty$, then $\Fil_\log^{ea}G_{k'} = \Fil_\log^a G_k$ for $a>0$.

\emph{(c)}
For $a >0$, we have $\Fil^{a+1} G_k \subseteq \Fil_\log^a G_k \subseteq \Fil^a G_k$.

\emph{(d)}  For graded pieces, we have
\begin{eqnarray*}
\textrm{for } a > 1 \textrm{, } \Fil^aG_k / \Fil^{a+}G_k & = & \left\{
\begin{array}{ll}
0 & a \notin \QQ, \\
\textrm{an abelian group killed by $p$} & a \in \QQ,
\end{array}
\right. \\
\textrm{for } a > 0 \textrm{, } \Fil_\log^aG_k / \Fil_{ \log}^{a+} G_k& = & \left\{
\begin{array}{ll}
0 & a \notin \QQ, \\
\textrm{an abelian group killed by $p$} & a \in \QQ.
\end{array}
\right.
\end{eqnarray*}

\emph{(e)} These filtrations on $G_k$ agree with the ones defined by Abbes and Saito in \cite{abbes-saito1, abbes-saito2}.
\end{theorem}
\begin{proof}
Using the comparison \cite[Theorem~4.4.1]{xiao1} of the arithmetic and differential conductors, this follows from their basic properties as stated in \cite[Theorem~2.4.1 and Proposition~4.1.7]{xiao1}.  We refer to \cite{xiao1} and \cite{abbes-saito1} for the definition of Abbes and Saito's filtrations.
\end{proof}

We assume Hypothesis~\ref{H:J-finite-set} for the rest of the subsection. 

\begin{definition}
Let $\rho$ be a representation of $G_k$ with finite local monodromy.  The \emph{log-breaks} of $\rho$ are defined to be the differential log-breaks of $\calE_\rho$, as a solvable $\partial_{J^+}$-differential module.  Put $b_\log(\rho; l) = b_\log(\calE_\rho; l)$ for $l = 1, \dots, \dim \rho$.
Similarly, the \emph{nonlog-breaks} of $\rho$ are defined to be the differential nonlog-breaks of $\calE_{\rho / \rho^{I_k}}$ together with the element 0 with multiplicity $\dim \rho^{I_k}$, where $\rho^{I_k}$ is the maximal subrepresentation  of $\rho$ on which $I_k$ acts trivially. Put $b_\nlog(\rho; l) = b_\nlog(\calE_{\rho/ \rho^{I_k}}; l)$ for $l = 1, \dots, \dim (\rho / \rho^{I_k})$, and $b_\nlog(\rho; \dim (\rho / \rho^{I_k}) +1) = \cdots = b_\nlog(\rho; \dim \rho) = 0$.

For simplicity, we also put $b_\nlog(\rho) = b_\nlog(\rho; 1)$ and $b_\log(\rho)= b_\log(\rho; 1)$; they are called the \emph{highest nonlog-break} and the \emph{highest log-break}, respectively.
\end{definition}

\begin{proposition}
For each $j \in J^+$, there is a ramification break $b_j(\rho)$ associated to  $b_j$ ($j \in J$) or $\pi_k$ ($j = 0$), such that $R_{\partial_j}(\calE_\rho \otimes F'_\eta) = \eta^{b_j(\rho)}$ for any $\eta \in (\eta_0,1)$ with some $\eta_0<1$.  Moreover, 
\[
b_\nlog(\rho) =  \max_{j \in J^+} \{b_j(\rho)\}, \quad b_{\log}(\rho) = \max \{b_0(\rho)-1; b_j(\rho) \textrm{ for } j \in J\}.
\]
\end{proposition}
\begin{proof}
By applying the same argument of Proposition~\ref{P:calE-solvable} to intrinsic $\partial_j$-radii, we know that $IR_{\partial_j}(\calE_\rho \otimes F'_{ \eta^q}) = IR_{\partial_j}(\calE_\rho \otimes F'_{\eta})^q$ as $\eta \rar 1^-$.  Therefore, by the convexity given by Theorem~\ref{T:variation-j}(d), $f_1^{(j)}(\calE_\rho, r)$ is affine as $r\rar 0^+$.  The proposition follows.
\end{proof}

\begin{definition}
We call $b_{J^+}(\rho)$ the \emph{breaks by $p$-basis} of $\rho$ with respect to the lifted $p$-basis $b_J$ and the uniformizer $\pi_k$.
\end{definition}

When we change the choices of $p$-basis of $k$, the breaks by basis $b_j(\rho)$ may change accordingly.

\begin{lemma}
Fix $j_0 \in J$.  Let $b'_{J^+}(\rho)$ be the breaks by $p$-basis of $\rho$ with respect to the lifted $p$-basis $\{ b_{J \bs \{j_0\}}, b_{j_0} + \pi_k\}$ and the uniformizer $\pi_k$.  Then we have $b'_j(\rho) = b_j(\rho)$ for $j \in J$ and
\[
b'_0 (\rho) \left\{
\begin{array}{ll}
= \max\{b_0(\rho), b_{j_0}(\rho)\} & \textrm{ if } b_0(\rho) \neq b_{j_0}(\rho),\\
\leq b_0(\rho) & \textrm{ if } b_0(\rho) = b_{j_0}(\rho).
\end{array}
\right.
\]
\end{lemma}
\begin{proof}
Let $\partial'_{J^+}$ denote the derivations dual to the basis $dB_{J \bs \{j_0\}}, dT, d(B_{j_0}+T)$ of $\Omega^1_{\calR^\inte_{K'}}$.  Then we have $\partial'_J = \partial_J$ and $\partial'_0 = \partial_0 - \partial_{j_0}$.  The lemma follows immediately.
\end{proof}

\begin{remark}
This lemma is in fact much stronger than it appears.  Applying the same argument to $b_{j_0} + \alpha\pi_k$ for all $\alpha \in k_0$ implies that, for \emph{all but possibly one} $\alpha \in k_0$, $b'_0(\rho) \geq b_{j_0}(\rho)$.  So, vaguely speaking, the equality $b_0(\rho) = b(\rho)$ holds ``generically"; this motivates the following lemma.
\end{remark}

\begin{lemma}\label{L:generic-rotation-break}
Fix $j_0 \in J$.
Let $\tilde k$ be the completion of $k(x)$ with respect to the $1$-Gauss norm, equipped with the lifted $p$-basis $\{b_{J \bs \{j_0\}}, b_{j_0} + x \pi_k, x\}$.  Let $\tilde \rho$ be the representation $G_{\tilde k} \rar G_k \stackrel \rho \rar GL(V_\rho)$.  Let $\tilde b_{J^+ \cup\{m+1\}}(\tilde\rho)$ denote the breaks by $p$-basis with respect to the aforementioned  lifted $p$-basis and the uniformizer $\pi$, where $\tilde b_{J \bs \{j_0\}}(\tilde \rho)$ corresponds to $b_{J \bs \{j_0\}}$, $\tilde b_{j_0}(\tilde\rho)$ corresponds to $b_{j_0} + x\pi_k$, $\tilde b_0(\tilde\rho)$ corresponds to $\pi_k$, and $\tilde b_{m+1}(\tilde\rho)$ corresponds to $x$.  Then we have $
 \tilde b_j(\rho') = b_j(\rho)$  for $j \in J$, $\tilde b_{m+1}(\tilde \rho) = b_{j_0}(\rho)-1$, $
\tilde b_0(\tilde\rho) = \max \{ b_0(\rho)$, and $ b_{j_0}(\rho)\}$.  In particular, $\tilde b_\nlog(\tilde\rho) = b_\nlog(\rho)$.
 \end{lemma}
\begin{proof}
Let $\widetilde K'$ denote the completion of $K'(X)$ with respect to the 1-Gauss norm, where $X$ is the canonical lift of $x$.  Let $f: A_{\widetilde K'}^1[\eta_0, 1) \rar A_{K'}^1[\eta_0, 1)$ be the natural morphism. Then $f^*\calE_\rho$ is the differential module associated to $\rho'$.
Let $\tilde \partial_{J^+\cup \{m+1\}}$ be the differential operators corresponding to the $p$-basis $(b_{J \bs\{j_0\}}, b_{j_0} + x \pi_k, \pi_k)$.  Then under the identification by $f^*$, we have
\begin{equation}
\label{E:generic-rotation-differential}
\tilde \partial_J = \partial_J, \quad
 \tilde \partial_{m+1} = T\partial_{j_0}, \quad \textrm{and}\ \tilde \partial_0 = \partial_0 - X \partial_{j_0}.
\end{equation}
The lemma follows from this because $X$ is transcendental over $K'$.
\end{proof}

\begin{lemma}\label{L:break-pth-root}
Fix $j_0 \in J$.  Set $k' = k(b_{j_0}^{1/p})$, equipped with the lifted $p$-basis $\{ b_{J \bs \{j_0\}}, b_{j_0}^{1/p}\}$.  Let $b'_{J^+}(\rho|_{G_{k'}})$ be the breaks by $p$-basis of $\rho|_{G_{k'}}$ with respect to the aforementioned $p$-basis and uniformizer $\pi_k$.  Then we have $b'_j(\rho|_{G_{k'}}) = b_j(\rho)$ for $j \in J^+ \bs \{j_0\}$ and $b'_{j_0}(\rho|_{G_{k'}}) = \frac 1p b_{j_0}(\rho)$.
\end{lemma}
\begin{proof}
Replacing $k$ by $k'$ is equivalent to pulling back the differential module $\calE_\rho$ along $\varphi^{(\partial_j)}$.  The lemma follows from applying Lemma~\ref{L:frob-properties}(a) to $\calE \otimes F'_\eta$ when $\eta \rar 1^-$.
\end{proof}

\begin{lemma}\label{L:break-p-infty-root}
Fix $j_0 \in J$.  Let $k'$ denote the completion of $k(b_{j_0}^{1/p^n}; n \in \NN)$ equipped with lifted $p$-basis $ b_{J \bs \{j_0\}}$.  Let $b'_{J^+}(\rho|_{G_{k'}})$ be the breaks by $p$-basis of $\rho|_{G_{k'}}$ with respect to this $p$-basis and the uniformizer $\pi_k$.  Then we have $b'_j(\rho|_{G_{k'}}) = b_j(\rho)$ for $j \in J^+ \bs \{j_0\}$.
\end{lemma}
\begin{proof}
Replacing $k$ to $k'$ is equivalent to simply forgetting the $j_0$-direction.
\end{proof}

\begin{situation}
\label{Sit:saito-base-change}
Now, we study a particular case of base change, which will be useful in the comparison Theorem~\ref{T:comparison}.  This type of base change was first considered by Saito in \cite{saito-wild-ram}.

Fix $e \in \NN$ possibly divisible by $p$.  Let $k$ be as above, and let $k'$ be the completion of $k(x)$ with respect to the 1-Gauss norm, with uniformizer $\pi_{k'} = \pi_k$.  Put $\tilde k = k'[u] / (u^e - x^{-1}\pi_k)$.
The residue field of $\tilde k$ is $\kappa(\bar x)$; we consider the $p$-basis $(b_J, \bar x)$ and the uniformizer $\pi_{\tilde k} = u$ of $\tilde k$.  We choose the unique isomorphism $ \kappa (\bar x) ((u)) \simeq \tilde k$ that is compatible with the chosen isomorphism $\kappa(\pi_k) \simeq k$ in Notation~\ref{N:k} and that sends $\bar x$ to $x$. This gives rise to the lifted $p$-basis $(b_J, x, u)$ of $\tilde k$.
\end{situation}

\begin{proposition}\label{P:saito-base-change}

The natural homomorphism $G_{\tilde k} \rar G_k$ induces a homomorphism $\Fil^{ea}_\log G_{\tilde k} \rar \Fil^a_\log G_k$ for any $a \in \QQ_{\geq 0}$.  Moreover, the induced homomorphism $\Fil^{ea}_\log G_{\tilde k} / \Fil^{ea+}_\log G_{\tilde k} \rar \Fil^a_\log G_k / \Fil^{a+}_\log G_k$ is surjective for any $a \in \QQ_{>0}$.
\end{proposition}
\begin{proof}
It suffices to show that, for a $p$-adic representation of $G_k$ with finite local monodromy and pure log-break $b_\log(\rho)$, the induced representation $\tilde \rho: G_{\tilde k} \rar G_k \rar GL(V_\rho)$ also has the same log-break.
 Let $\widetilde K'$ be the completion of $K'(X)$ with respect to the 1-Gauss norm, where $X$ is the canonical lift of $\bar x$.  We then  have a natural map $f: A^1_{\widetilde K'}[ \eta^{1/e}, 1) \rar A^1_{K'}[\eta, 1)$ for $\eta \rar 1^-$, sending $T$ to $XU^e$, where $U$ is the coordinate of the former annulus.

Let $\tilde b_0(\tilde \rho), \dots, \tilde b_{m+1}(\tilde \rho)$ be the breaks by $p$-basis with respect to $b_J$, $x$ and the uniformizer $\pi_{\tilde k } = u$.  Then $f^* \calE_\rho$ is the differential module associated to $\tilde \rho$, with the actions of $\tilde \partial_0 = \partial / \partial U$, $\tilde \partial_J = \partial / \partial B_J$, and $\tilde \partial_{m+1} = \partial / \partial X$.  We have
\begin{equation}\label{E:saito-base-change-differential}
\tilde \partial_J = \partial_J, \quad \tilde \partial_0 = eXU^{e-1} \partial_0, \quad \textrm{and }\ \tilde \partial_{m+1} = U^e \partial_0.
\end{equation}
By Theorem~\ref{T:independence-of-basis}, we have $\tilde b_J(\tilde \rho) = e b_J(\rho)$, $\tilde b_0(\tilde \rho) \leq e b_0(\rho) - (e-1)$, and $\tilde b_{m+1} (\tilde \rho) = e b_0(\rho) -e$  (when $e$ is prime to $p$, the inequality becomes an equality).  In particular, we have $\tilde b_{m+1}(\tilde \rho) \geq \tilde b_0(\tilde \rho) -1$.  Hence we conclude that
\[
b_\log(\tilde \rho) = \max\{\tilde b_0(\tilde \rho) - 1, \tilde b_J(\tilde \rho), \tilde b_{m+1}(\tilde \rho)\} = \max \{e b_J(\rho), e b_0(\rho) - e\} = e b_\log(\rho).
\]
This proves the proposition.
\end{proof}

\subsection{Refined differential conductors}
\label{S:refined}

In this subsection, we define the refined differential conductors, which provides additional information the subquotient $\Fil^a_{(\log)}G_k / \Fil^{a+}_{(\log)}G_k$ of the ramification filtrations.  This definition makes use of the refined Swan conductors we discussed in Subsection~\ref{S:differential-conductor}.

We keep the notation as in previous subsections but we drop Hypothesis~\ref{H:J-finite-set}.

\begin{notation}\label{N:Dwork-pi}
Fix a Dwork pi $\boldsymbol{\pi} = (-p)^{1/(p-1)}$.
\end{notation}

\begin{notation}
We put $\Omega^1_{\calO_k}(\log) = \Omega^1_{\calO_k} + \calO_k \frac{d\pi_k}{\pi_k} \subset \Omega^1_k$.  If we choose a $p$-basis $\bar b_J$ of $\kappa$ as in Notation~\ref{N:k}, we have $\Omega^1_{\calO_k}(\log) = \calO_k \frac{d\pi_k}{\pi_k} \oplus \bigoplus_{j \in J} \calO_k db_j$.
\end{notation}

\begin{construction}
\label{C:refined-conductor}
Let $\rho$ be a $p$-adic representation of $G_k$ with finite local monodromy and with pure break $b = b_\nlog(\rho)$ (resp. log-break $b = b_{ \log}(\rho)$).  We may replace $k$ by the completion of an inseparable extension as in Definition~\ref{D:conductors-repn} and then assume Hypothesis~\ref{H:J-finite-set}.  Let $\calE_\rho$ denote the $(\phi, \nabla)$-module associated to $\rho$.  By Theorem~\ref{T:variation}(e), there exists $\eta_0 \in (0, 1)$ such that $\calE_\rho \otimes F'_\eta$ has pure extrinsic (resp. intrinsic) radii $\eta^b$ for any $\eta \in [\eta_0, 1)$.

We define the multiset of \emph{refined Artin conductors} of $\rho$ to be
\[
\mathrm{rar}(\rho) = \big\{\frac 1{\boldsymbol \pi}\vartheta \pi_k^{-b} \; | \; \vartheta \in \calI\Theta(\calE_\rho) \big\} \subset  \Omega_{\calO_k}^1 \otimes_{\calO_k} \pi_k^{-b}\bar \kappa.
\]
Similarly, we define the multiset of \emph{refined Swan conductors} of $\rho$ to be
\[
\rsw(\rho) = \big\{\frac 1{\boldsymbol \pi}\vartheta \pi_k^{-b} \; | \; \vartheta \in \calI\Theta(\calE_\rho) \big\} \subset  \Omega_{\calO_k}^1(\log) \otimes_{\calO_k} \pi_k^{-b}\bar \kappa.
\]
\end{construction}

\begin{remark} \label{R:Dwork-pi-vs-pth-root}
There is a unique primitive $p$-th root of unity $\zeta_p$ such that $\boldsymbol \pi \equiv (\zeta_p-1) \mod (\zeta_p-1)^2$.  The definition of refined conductors above is unchanged if we replace $\boldsymbol \pi$ by this $\zeta_p - 1$.
\end{remark}

\begin{lemma}
In Construction~\ref{C:refined-conductor}, the definition of the refined Artin and Swan conductors does not depend on the choices of the lifted $p$-basis of $k$ and the uniformizer $\pi_k$.
\end{lemma}
\begin{proof}
We may assume Hypothesis~\ref{H:J-finite-set} since only finitely many elements in the $p$-basis appear in the refined Artin and Swan conductors.

For another choice of lifted $p$-bases and uniformizers, we will consider another set of differential operators: $\partial'_j = \partial / \partial B'_j$ for $j \in J$ and $\partial'_0 = \partial/\partial T'$.  We put
\[
dB_j = \sum_{j' \in J} \alpha_{j, j'} dB'_{j'} + \alpha_{j, 0}dT' \textrm{ for } j \in J, \quad\textrm{and }\
dT = \sum_{j' \in J} \alpha_{0, j'} dB_{j'} + \alpha_{0, 0} dT',
\]
where $\alpha_{j, j'} \in \calO_{K'}\llbracket T\rrbracket$ for $j, j' \in J^+$.  Moreover, we have $\alpha_{0, j} \in T\cdot \calO_{K'}\llbracket T\rrbracket$.

We may assume that $\calE_\rho$ has pure differential nonlog-break (resp. log-break) and has pure extrinsic (resp. intrinsic) radii $\eta^b$ for $\eta \in [\eta_0, 1)$ for some $\eta_0 \in (0, 1)$.

 Theorem~\ref{T:independence-of-basis} then implies that, for $\eta \in (\eta_0, 1) \cap p^\QQ$ and
for any $j \in J^+$ such that $R_{\partial'_j}(V \otimes F'_\eta) = ER(V \otimes F'_\eta)$ (resp. $IR_{\partial'_j}(V \otimes F'_\eta) = IR(V \otimes F'_\eta)$), we have
\begin{align*}
\Theta_{\partial'_j}(\calE_\rho \otimes F'_\eta) &= \big\{\bbpi T^{-b}(\alpha_{0,j}\theta_0 + \cdots +\alpha_{m, j} \theta_m) \big| \bbpi T^{-b}(\theta_0dT + \theta_1 dB_1 +\cdots + \theta_m dB_m) \in \calE \Theta(\calE_\rho \otimes F'_\eta)\big\} \\
\textrm{(resp. }
\Theta_{\partial'_j}(\calE_\rho \otimes F'_\eta) &= \big\{\bbpi T^{-b}(\frac{\alpha_{0,j}}T\theta_0 + \cdots +\alpha_{m,j} \theta_m) \big| \bbpi T^{-b}(\theta_0\frac {dT}T + \theta_1 dB_1 +\cdots + \theta_m dB_m) \in \calI \Theta(\calE_\rho \otimes F'_\eta)\big\}\ )
\end{align*}
Note also that
\begin{align*}
&\big(\alpha_{0, 0}\theta_0 + \cdots +\alpha_{m,0} \theta_m\big) dT' +
\sum_{j \in J} \big(\alpha_{0,j}\theta_0 + \cdots +\alpha_{m,j} \theta_m\big) dB'_j \\
&= \theta_0 dT + \theta_1 dB_1 + \cdots +\theta_m dB_m.
\end{align*}
Combining these two formulas, we conclude that $\calE\Theta(V)$ (resp. $\calI\Theta(V)$) for $\partial_{J^+}$ is the same as that for $\partial'_{J^+}$.  Hence the refined Artin (resp. Swan) conductors are well-defined.
\end{proof}

\begin{lemma}
\label{L:tame-refined}
Let $k'/k$ be a tamely ramified extension of ramification degree $e=e_{k'/k}$ and let $\rho$ be a $p$-adic representation of $G_k$ with finite local monodromy and with pure log-break $b=b_\log(\rho)$.  Then  $\rho|_{G_{k'}}$ has pure log-break $eb$.  Moreover, if we identify $\Omega_{\calO_k}^1(\log) \otimes_{\calO_k} \pi_k^{-b} \bar\kappa$ with $\Omega_{\calO_{k'}}^1(\log) \otimes_{\calO_{k'}} \pi_{k'}^{-eb} \bar\kappa$, then $\rsw(\rho) $ is the same as $\rsw(\rho|_{g_{k'}})$.
\end{lemma}
\begin{proof}
This follows immediately from the fact that $\calE_{\rho|_{G_{k'}}}$ is just the base change of $\calE_{\rho}$ along $A^1_{K'}[\eta^{1/e}, 1) \to A^1_{K'}[\eta, 1)$, where the coordinate for the first annulus is $t^{1/e}$.
\end{proof}

\begin{theorem}
\label{T:refined-homomorphism}
Let $k$ be a complete discrete valuation field of equal characteristic $p>0$.
\begin{itemize}
\item[\emph{(a)}] Let $\rho$ be a $p$-adic representation of $G_k$ with finite local monodromy and with pure log-break $b = b_{\log}(\rho)>0$.  Then there exists a unique direct sum  decomposition of $\rho$ as $\rho \cong \oplus_{\{\vartheta\} \subset \rsw(\rho)} \rho_{\{\vartheta\}}$, where the direct sum is taken over all $\mu_e \rtimes G_k$-orbits $\{\vartheta\}$ in $\rsw(\rho)$, and $\rsw(\rho_{\{\vartheta\}})$ consists of the Galois orbits $\{\vartheta\}$ with appropriate multiplicity.  Moreover, there exists a finite tamely ramified extension $k'/k$ of na\"ive ramification degree $e$ such that  we have a unique  direct sum decomposition of representations of $G_{k'}$ over some finite extension $F'$ of $F$:
$
\rho|_{G_{k'}} \otimes F' \cong \bigoplus_{\vartheta \in \rsw(\rho)} \rho_\vartheta,
$
such that $\rho_\vartheta$ has pure refined Swan conductors $\vartheta \in \Omega_{\calO_{k'}}^1(\log) \otimes_{\calO_{k'}} \pi_{k'}^{-eb} \bar\kappa \cong \Omega_{\calO_k}^1(\log) \otimes_{\calO_k} \pi_k^{-b} \bar\kappa$.
\item[\emph{(b)}]  Choose the $p$-th root of unity $\zeta_p$ as in Remark~\ref{R:Dwork-pi-vs-pth-root}.  Then there exists an \emph{injective} homomorphism for any $b \in \QQ_{> 0}$,
\begin{equation}\label{E:rsw}
\rsw = \rsw_k: \Hom(\Fil_{\log}^b G_k / \Fil_{\log}^{b+} G_k, \Fp) \rar \Omega^1_{\calO_k}(\log) \otimes_{\calO_k} \pi_k^{-b}\bar \kappa,
\end{equation}
such that, when viewing the left hand side as a subset of $\Hom(\Fil_{\log}^b G_k / \Fil_{\log}^{b+} G_k, \Qp(\zeta_p)^\times)$ via the identification of $1 \in \Fp$ with $\zeta_p$, we have, for any $p$-adic representation $\rho$ of $G_k$ with finite local monodromy and with pure log-break $b$, the images of the summands of $\rho|_{\Fil_{\log}^b G_k}$ under the homomorphism $\rsw$ exactly form the multiset of refined Swan conductors of $\rho$.
Moreover, 
the homomorphism \eqref{E:rsw} does not depend on the choices of the Dwork pi.
\end{itemize}
\end{theorem}
\begin{proof}
For both (a) and (b), we may assume that Hypothesis~\ref{H:J-finite-set} holds, since only finitely many elements in a $p$-basis matter.
 
(a)
Using the identification given in Lemma~\ref{L:tame-refined}, we may first replace $k$ and $\Frac \calO$ by a tamely ramified extension of $k$ and a finite extension of $\Frac\calO$, respectively, so that the decomposition of the $\nabla$-module $\calE_\rho$ given by \eqref{E:var-refined-all} of $\calE_\rho$ can be realized over $\calR_{K'}$, and that $\Fq \subseteq k_0$.  
Since this decomposition is canonical, it is also a decomposition of $(\phi, \nabla)$-modules.  By the slope filtration \cite[Theorem~3.4.6]{kedlaya-swan1}, the Frobenius action on each direct summand of $\calE_\rho$ is \'etale, yielding the decomposition of the representation via the equivalence of categories in Theorem~\ref{T:equivalence-categories}.

(b) The following are immediate corollaries of Proposition~\ref{P:refined-properties}.
\begin{itemize}

\item[(i)] For any $p$-adic representations $\rho$ and $\rho'$ of $G_k$ with finite local monodromy, same pure log-break $b$, and same pure refined Swan conductor $\vartheta$, the log-break of $\rho \otimes \rho'^\dual$ is strictly smaller than $b$.

\item[(ii)] For any $p$-adic representations $\rho$ and $\rho'$ of $G_k$ with finite local monodromy, same pure log-break $b $, but different pure refined Swan conductor $\vartheta \neq \vartheta'$, respectively,  $\rho \otimes \rho'^\dual$ has pure log-break $b$ and pure refined Swan conductor $\vartheta -\vartheta'$.
\end{itemize}

We also need the following easy fact about Galois representations.

\begin{itemize}
\item[(iii)] For any homomorphism $\chi: \Fil_{\log}^b G_k / \Fil_{\log}^{b+} G_k \rar \Fp$, there exist a finite tamely ramified extension $k'$ of $k$ with na\"ive ramification degree $e$ and a representation $\rho_\chi$ of $G_{k'}$ with finite local monodromy,  pure log-break $eb$, and  pure refined Swan conductor, such that $\rho_\chi|_{\Fil_{\log}^b G_k / \Fil_{\log}^{b+} G_k}$ contains $\chi$ as a direct summand.
\end{itemize}
Proof of (iii):  The chosen $p$-th root of unity $\zeta_p$ in Remark~\ref{R:Dwork-pi-vs-pth-root} promotes $\chi$ to the homomorphism $\chi: \Fil_{\log}^b G_k / \Fil_{\log}^{b+} G_k \rar \Fp \to \Qp(\zeta_p)^\times$ by identifying $1$ with $\zeta_p$.  
Since $G_k / \Fil_\log^{b+}G_k$ is a pro-finite group, there exists a normal subgroup $H$ of $G_k$ of finite index containing $\Fil_\log^{b+}G_k$, such that $\chi$ factors through $I = \Fil_\log^b G_k / (H \cap \Fil_\log^b G_k)$.  Put $\rho' = \Ind^{G_k/H}_{I} \chi$; then $\rho'|_{\Fil^bG_k}$ contains $\chi$ as a direct summand.  We may use (a) to write $\rho'|_{G_{k'}}$ for some finite tamely ramified extension $k'$ of $k$ as the the direct sum of representations with pure refined Swan conductors.  Then $\chi$ appears in at least one of the direct summand, which we take to be our chosen $\rho_\chi$.

Having established (iii), we define $\rsw$ to be the morphism sending $\chi$ to the unique refined Swan conductor of $\rho_\chi$, which is an element of 
$
\Omega^1_{\calO_{k'}}(\log) \otimes_{\calO_{k'}} \pi_{k'}^{-eb} \bar \kappa \cong \Omega^1_{\calO_{k}}(\log) \otimes_{\calO_{k}} \pi_{k}^{-b} \bar \kappa,
$
via the identification in Lemma~\ref{L:tame-refined}.
This map is well-defined by (iv) below and it is clearly a homomorphism. Its injectivity will follow from (v).

\begin{itemize}
\item[(iv)]  For any two representations $\rho_\chi$ and $\rho'_{\chi}$ satisfying (iii), they must have the same refined Swan conductor.
\end{itemize}
Suppose the contrary, that is $\rho_\chi$ and $\rho'_\chi$ have distinct pure refined Swan conductors $\vartheta$ and $\vartheta'$.  This particular implies that $\rho_\chi \otimes \rho'^\vee_\chi$ has pure Swan conductor $b$ by (ii).  However, the construction of $\rho_\chi$ and $\rho'_\chi$ implies that $\rho_\chi \otimes \rho'^\vee_{\chi} |_{G_{k'}}$ contains a direct summand trivial on $\Fil_\log^{eb}G_{k'}$; this is a contradiction.

\begin{itemize}
\item[(v)]  For two distinct homomorphisms $\chi, \chi': \Fil_{\log}^b G_k / \Fil_{\log}^{b+} G_k \rar \Fp$, the representations $\rho_\chi$ and $\rho_{\chi'}$ given by (iii) have distinct refined Swan conductors.
\end{itemize}
Suppose the contrary.  Then (i) implies that $\rho_\chi \otimes \rho^\vee_{\chi'}$ would have log-break strictly less than $eb$.  However,  $\rho_\chi \otimes \rho^\vee_{\chi'}$, when restricted to $\Fil_{\log}^b G_k / \Fil_{\log}^{b+} G_k = \Fil_{\log}^{eb} G_{k'} / \Fil_{\log}^{eb+} G_{k'}$, has a direct summand isomorphic to $\chi \otimes \chi'^\vee$, which is nontrivial.  This is a contradiction.

We now prove the independence on the choice of the Dwork pi.
If we choose another Dwork pi, we would need to use another primitive $p$-th root of unity $\zeta_p^i$ for some $i \in {1, \dots, p-1}$.  On one hand, the refined Swan conductor is multiplied by $\frac{\zeta_p^i - 1}{\zeta_p-1} \equiv i \mod (\zeta_p -1)$. On the other hand, the $p$-adic representation $\Fil^b_\log G_k / \Fil^{b+}_\log G_k \rar \Qp(\zeta_p)^\times$ becomes $\chi^i$.  Hence we need to take $\rho^{\otimes i}_chi$ as our $p$-adic representation of $G_{k'}$ to define the homomorphism $\rsw$.  This representation has refined Swan conductor $\rsw(\rho_\chi^{\otimes i}) = i \cdot \rsw(\rho_\chi)$, which is the same as the refined Swan conductor of $\rho$ computed using the old Dwork pi.
\end{proof}

\begin{remark}
It is interesting to point out that the choice  of a Dwork pi is related to the choice of the Artin-Scheier $\ell$-adic sheaf in \cite{saito-wild-ram}; they both amount to choosing a primitive $p$-th root of unity.  The difference is that we consider it as an element in $\overline \QQ_p$ whereas Saito viewed it as an element in $\overline \QQ_l$.
\end{remark}

\begin{proposition}
\label{P:decomposition=>conjugation}
Let $k$ be a complete discrete valuation field of equal characteristic $p>0$.  Then for $b \in \QQ_{>0}$, the conjugation action of $\Fil_\log^{0+}G_k/\Fil_\log^b G_k$ on $\Fil_\log^b G_k / \Fil_\log^{b+} G_k$ is trivial.  In other words, $\Fil_\log^b G_k / \Fil_\log^{b+} G_k$ lies in the center of $\Fil_\log^{0+}G_k/\Fil_\log^{b+} G_k$.
\end{proposition}
\begin{proof}
This proposition is proved in \cite[Theorem~1]{abbes-saito2}.  We hereby give an alternative proof using differential modules.

It suffices to prove the following: for a $p$-adic representation $\rho$ of $G_k$ with finite local monodromy and with pure log-break $b$, if it is absolutely irreducible under any tamely ramified extension, then $\rho|_{\Fil^b_{\log} G_k / \Fil^{b+}_\log G_k}$ is a direct sum of a \emph{single} character $\chi: \Fil^b_\log G_k/ \Fil^{b+}_\log G_k \rar \calO^\times$.   This is equivalent to showing that the action of $\Fil^b_\log G_k$ on $\rho \otimes \rho^\dual$ is trivial, and hence to showing that the log-break of $\rho \otimes \rho^\dual$ is strictly smaller than $b$.

As usual, we may assume Hypothesis~\ref{H:J-finite-set}.
By Theorem~\ref{T:refined-homomorphism}(a), the irreducibility condition on $\rho$ implies that $\rho$ must have pure refined Swan conductor and hence the log-break $\rho \otimes \rho^\dual$ must be strictly less than $b$.  We are done.
\end{proof}

\begin{proposition}
Keep the notation as in Situation~\ref{Sit:saito-base-change}.  Then the refined Swan conductor homomorphism $\rsw_k$ for $k$ factors as
\begin{equation}\label{E:saito-base-change}
\Hom( \Fil^b_\log G_k / \Fil^{b+}_\log G_k, \Fp) \rar \Hom( \Fil^{e_{\tilde k/ k}b}_\log G_{\tilde k} / \Fil^{e_{\tilde k / k}b+}_\log G_{\tilde k}, \Fp)
\stackrel {\rsw_{\tilde k}}\lrar \Omega^1_{\calO_{\tilde k}}(\log) \otimes_{\calO_{\tilde k}} \pi_{\tilde k}^{-eb} \kappa_{\tilde k^\alg}.
\end{equation}
\end{proposition}
\begin{proof}
Keep the notation as in Proposition~\ref{P:saito-base-change}, let $\widetilde F'_\eta$ be the completion of $\widetilde K'(U)$ with respect to the $\eta^{1/e}$-Gauss norm in $U$.  Fix $\eta_0 \in (0, 1)$ such that $IR(\calE_\rho \otimes F'_\eta) = \eta^b$ for $\eta \in [\eta_0, 1)$. Then  \eqref{E:saito-base-change-differential} implies that, for any $\eta \in [\eta_0, 1) \cap p^\QQ$ and for any $j \in \{0, \dots, m+1\}$ such that $IR_{\partial_j}(f^*\calE_\rho \otimes \widetilde F'_\eta) = IR(\calE_\rho \otimes \widetilde F'_\eta)$, we have
\[
\Theta_{\partial_j}(f^*\calE_\rho \otimes \widetilde F'_\eta) = \left\{
\begin{array}{ll}
\Theta_{\partial_j}(\calE \otimes F'_\eta) & j \in J,\\
eXU^{e-1} \Theta_{\partial_0}(\calE \otimes F'_\eta) & j = 0 \textrm{ and hence } p \nmid e,\\
U^e \Theta_{\partial_{m+1}}(\calE \otimes F'_\eta) & j = m+1.
\end{array}\right.
\]
We here used Theorem~\ref{T:independence-of-basis} to compute the refined radii.  The proposition follows.
\end{proof}

One may want to prove analogs of Theorem~\ref{T:refined-homomorphism} and Proposition~\ref{P:decomposition=>conjugation} for refined Artin conductors.  This however needs to take a bit more effort because there may not be a representation of $G_k$ with pure refined Artin conductor.  Instead, we reduce to the classical case, where the results for refined Artin conductors follows from those for refined Swan conductors.

\begin{theorem}
Let $k$ be a complete discrete valuation field of equal characteristic $p>0$.
\begin{itemize}
\item[\emph{(a)}] 
Choose the $p$-th root of unity $\zeta_p$ as in Remark~\ref{R:Dwork-pi-vs-pth-root}.  Then there exists an \emph{injective} homomorphism for any $b \in \QQ_{> 1}$,
\begin{equation}\label{E:rar}
\mathrm{rar} = \mathrm{rar}_k: \Hom(\Fil^b G_k / \Fil^{b+} G_k, \Fp) \rar \Omega^1_{\calO_k} \otimes_{\calO_k} \pi_k^{-b}\bar \kappa,
\end{equation}
such that, when viewing the left hand side as a subset of $\Hom(\Fil^b G_k / \Fil^{b+} G_k, \Qp(\zeta_p)^\times)$ via the identification of $1 \in \Fp$ with $\zeta_p$, we have, for any $p$-adic representation $\rho$ of $G_k$ with finite local monodromy and with pure nonlog-break $b$, the images of the summands of $\rho|_{\Fil^b G_k}$ under $\mathrm{rar}$ exactly form the multiset of refined Artin conductors of $\rho$.
Moreover, 
this homomorphism does \emph{not} depend on the choices of the Dwork pi.

\item[\emph{(b)}]  For any $b \in \QQ_{>1}$, the conjugation action of $\Fil^{1+}G_k/\Fil^b G_k$ on $\Fil^b G_k / \Fil^{b+} G_k$ is trivial.  In other words, $\Fil^b G_k / \Fil^{b+} G_k$ lies in the center of $\Fil^{1+}G_k/\Fil^{b+} G_k$.
\end{itemize}
\end{theorem}
\begin{proof}
For both (a) and (b), we may assume Hypothesis~\ref{H:J-finite-set}.  Moreover, we assume that $J$ is not empty because otherwise we are in the classical case, and both (a) and (b) follow from their log-version counterpart: Theorem~\ref{T:refined-homomorphism} and Proposition~\ref{P:decomposition=>conjugation}, respectively.

 We perform a base change similar to the one in Lemma~\ref{L:generic-rotation-break}.
Let $k'$ be the completion of $k(x_1, \dots, x_m)$ with respect to the $(1, \dots, 1)$-Gauss norm and let $\tilde k$ be the completion of $k'\big((b_j+x_j\pi_k)^{1/p^n}, x_j^{1/p^n}; n \in \NN; j \in J\big)$, equipped with the uniformizer $\pi_{\tilde k} = \pi_k$.  It is in fact a complete discrete valuation field with \emph{perfect} residue field.   By Lemmas~\ref{L:generic-rotation-break} and \ref{L:break-pth-root}, the natural homomorphism $G_{\tilde{k}} \to G_k$ induces a \emph{surjective} homomorphism $\Fil^a G_{\tilde k} / \Fil^{a+1} G_{\tilde k} \to \Fil^a G_k \to \Fil^{a+} G_k$.  Dualizing this gives an \emph{injective} homomorphism $\mu: \Hom(\Fil^a G_k/\Fil^{a+} G_k, \FF_p) \to \Hom(\Fil^{a} G_{\tilde k} / \Fil^{a+} G_{\tilde k}, \FF_p)$.

For $\rho$ a representation of $G_k$ with finite local monodromy and with pure nonlog-break $b$ we let $\tilde \rho$ denote the representation $G_{\tilde k} \rar G_k\stackrel \rho \rar GL(V_\rho)$.  Let $K''$ denote the completion of $K'(X_J)$ with respect to the $(1, \dots, 1)$-Gauss norm, where $X_j$ is a lift of $x_j$ for $j \in J$.  Let $\widetilde K$ denote the completion of $K''\big( (B_j+X_jT)^{1/p^n}, X_j^{1/p^n}; n \in \NN, j \in J \big)$.  Let $f: A^1_{\widetilde K}[\eta_0, 1) \rar A_{K'}^1[\eta_0, 1)$ denote the natural morphism.  Then $f^*\calE_\rho$ is the differential module associated to $\tilde \rho$.  Let $\tilde \partial$ denote the differential operator on $f^* \calE_\rho$ dual to the basis $dT$.  Similar to \eqref{E:generic-rotation-differential}, we have
\[
\tilde \partial = \partial_0 - X_1\partial_1 - \cdots -X_m \partial_m.
\]
If we let $\widetilde F_\eta$ denote the completion of $\widetilde K(T)$ with respect to the $\eta$-Gauss norm, we have 
\[
R_{\tilde \partial}(f^*\calE \otimes \widetilde F_\eta) = \min_{j \in J^+} \big\{ R_{\partial_j}(\calE \otimes F'_\eta)\big\}.
\]
Hence $\tilde \rho$ has pure nonlog-break $b$ and, by Theorem~\ref{T:independence-of-basis}, its multiset of refined Artin conductors is
\[
\mathrm{rar}(\tilde \rho) = \big\{ (\theta_0 - X_1\theta_1 - \cdots - X_m \theta_m) d\pi_k \big| \theta_0 d\pi_k + \theta_1 db_1 +\cdots + \theta_m db_m \in \mathrm{rar}(\rho) \big\}.
\]
In other words, if we use $\lambda$ denote the $\bar \kappa$-linear injective homomorphism $\Omega^1_{\calO_k} \otimes_{\calO_k} \pi_k^{-b} \overline \kappa \to \pi_k^{-b} \kappa_{\tilde k^\alg} d\pi_k$ given by $\lambda(db_j) = -X_j d\pi_k$ and $\lambda(d\pi_k) = d\pi_k$, then $\mathrm{rar}(\tilde \rho) = \lambda (\mathrm{rar}(\rho))$. This together with the injectivity of $\mu$ reduce (a) and (b) for $G_{k}$ to that of $G_{\tilde k}$, which is already known as we explained earlier.  In particular, we have $\lambda \circ \rsw_k = \rsw_{\tilde k} \circ \mu$.
\end{proof}

\subsection{Multi-indexed ramification filtrations for higher local fields}
\label{S:higher-local-fields}

When $k$ is a $n$-dimensional local field, the refined Artin and Swan conductors give more refined filtrations on the Galois group $G_k$, indexed by $\QQ^n$ with lexicographic order.
  We restrict ourselves to the equal characteristic $p>0$ case.

\begin{definition}\label{D:higher-local-field}
We say that a complete discrete valuation field $k$ of characteristic $p>0$ is an \emph{$(m+1)$-dimensional local field} if there is a chain of fields $k = k_{m+1}, k_{m}, \dots, k_0$, where $k_{i+1}$ is a complete discrete valuation field with residue field $k_i$ for $i = 0, \dots, m$.  Contrary to most literature, we do \emph{not} assume that $k_0$ is a perfect field. Let $\{b_j\}_{j \in J}$ be a set of lifts of a $p$-basis of $k_0$ to $\calO_k$.

An $(m+1)$-tuple of elements $t_0, \dots, t_m \in k$ is called a \emph{system of local parameters} of $k$ if $t_i \in\calO_k$ is a lift of a uniformizer of $k_{m+1-i}$ all the way up to $k$.  Such a choice gives a (non-canonical) isomorphism $k \simeq k_0((t_m)) \cdots ((t_0))$.  In this case, we have
\[
\Omega_{\calO_{k'}}^1(\log) = \bigoplus_{i = 0}^m \calO_{k'} \frac{dt_i}{t_i} \oplus \bigoplus_{j \in J} \calO_{k'} \frac{db_j}{b_j}, \quad \textrm{and }\ \Omega_{\calO_{k'}}^1 = \bigoplus_{i = 0}^m \calO_{k'} dt_i \oplus \bigoplus_{j \in J} \calO_{k'} db_j.
\]

Equip $\QQ^{m+1}$ with the lexicographic order: $\bbi = (i_1, \dots, i_{m+1}) < \bbj = (j_1, \dots, j_{m+1})$ if and only if 
\[
i_l < j_l, i_{l+1} = j_{l+1}, \dots, \textrm{ and }i_{m+1}= j_{m+1} \textrm{ for some }l \leq m+1.
\]
For $a \in\QQ$, we also use $\QQ^{m+1}_{>a}$ to denote the subset of $\QQ^{m+1}$ consisting of $\bbi = (i_1, \dots, i_{m+1}) $ such that $i_{m+1}>a$.

Given a system of local parameters, we define a multi-indexed valuation as follows, denoted by $\bbv = (v_1, \dots, v_{m+1}): k^\times \rar \ZZ^{m+1} \subset \QQ^{m+1}$, where $v_{m+1} = v_{k_{m+1}}$ and recursively we have, downwards from $i=m+1$ to $i=1$, that $v_{i-1}(\alpha) = v_{k_{i-1}}(\alpha_{i-1})$ with $\alpha_{i-1} $ equal to the reduction of $\alpha_i t_{m+1-i}^{-v_i(\alpha_i)}$ in $k_{i-1}$.  Note that the definition of $\bbv$ depends on the choice of local parameters $t_0, \dots, t_m$.
\end{definition}

\begin{definition}
\label{D:multi-index-filtration}
For $\lambda = \sum_{i = 0}^m\alpha_i dt_i  + \sum_{j \in J} \beta_j db_j\in \Omega^1_{\calO_k} \otimes_{\calO_k} k$, we set 
\[
\bbv_\nlog(\lambda) = \min\{\bbv(\alpha_0), \dots, \bbv(\alpha_m), \bbv(\beta_j); j \in J\}.
\]
This gives a multi-indexed valuation on $\Omega^1_{\calO_k} \otimes_{\calO_k} t_0^{-i_{m+1}}\overline \kappa$ for $i_{m+1} \in \QQ$.

For $\lambda = \sum_{i = 0}^m \alpha_i \frac{dt_i}{t_i} + \sum_{j \in J} \beta_j \frac{db_j}{b_j}\in \Omega^1_{\calO_k}(\log) \otimes_{\calO_k} k$, we set
\[
\bbv_\log(\lambda) = \min\{\bbv(\alpha_0), \dots, \bbv(\alpha_m), \bbv(\beta_j); j \in J\}.
\]
This gives a multi-indexed valuation on $\Omega^1_{\calO_k}(\log) \otimes_{\calO_k} t_0^{-i_{m+1}} \overline \kappa$ for $i_{m+1} \in \QQ$.

For $\bbi = (\serie{i_}{m+1}) \in \QQ^{m+1}_{>1}$, we define $\Fil^\bbi G_k$ to be the inverse image along the homomorphism 
\[
\Fil^{i_{m+1}} G_k \rar \Fil^{i_{m+1}} G_k / \Fil^{i_{m+1}+}G_k \stackrel {\mathrm{rar}} \lrar \Omega_{\calO_k}^1 \otimes_{\calO_k} t_0^{-i_{m+1}} \overline \kappa
\]
of the elements whose image under $-\bbv_\nlog$ is greater than or equal to $\bbi$.

For $\bbi = (\serie{i_}{m+1}) \in \QQ^{m+1}_{>0}$, we define $\Fil^\bbi_\log G_k$ to be the inverse image along the homomorphism 
\[
\Fil^{i_{m+1}}_\log G_k \rar \Fil^{i_{m+1}}_\log G_k / \Fil^{i_{m+1}+}_\log G_k \stackrel {\mathrm{rsw}} \lrar \Omega_{\calO_k}^1(\log) \otimes_{\calO_k} t_0^{-i_{m+1}} \overline \kappa
\]
of the elements whose image under $-\bbv_\log$ is greater than or equal to $\bbi$.
\end{definition}

\begin{remark}
The abstract filtrations do not depend on the choices of local parameters, but the indexings do.
Set $O_K = \{ x \in K | \bbv(x)\geq (0, \dots, 0) \}$.
It might be more natural to index the above filtrations by ``rational powers of fractional ideals of $K$" of the form $I^{1/n}$, where $I$ is an $O_K$-submodule of $K$ containing $O_K$, $n$ is an integer, and $I^{1/n}$ is equivalent to $I'^{1/n'}$ if $I^{n'} = I'^{n}$ as $O_K$-submodules of $K$.
\end{remark}

\begin{remark}
When $k_0$ is a finite field, this filtration is expected to be compatible with an easily defined filtration on the Milnor $K$-groups via class field theory for higher local fields. This may be verified by comparing the filtration on the Milnor K-groups with Kato's refined Swan conductors, which is equivalent to Saito's definition by \cite[Theorem~9.1.1]{abbes-saito-micro-local} and hence to our definition by Theorem~\ref{T:comparison} proved later. For more along this line, the reader may refer to the recipe in Kato's masterpiece \cite{kato}.
\end{remark}

\section{Comparison with Saito's definition}

In this section, we compare our definition of the refined Swan conductor homomorphism with the one given by Saito in \cite{saito-wild-ram}.  Since the reader who is only interested in one side of the story may use this result (Theorem~\ref{T:comparison}) as a black box, we present the proof assuming that the reader is familiar with the definition of arithmetic ramification filtrations (see e.g, \cite[Section~1]{saito-wild-ram} and \cite{xiao1}).  

The proof of the comparison theorem is of geometric natural.  We explain the rough idea here.  We first realize the given finite extension $l$ of $k$ as the corresponding extension of function fields of a finite \'etale extension of smooth affine varieties $Y \to X$.  Our main object is some version of infinitesimal neighborhood of the generic fiber over $k$ of the diagonal embedding of $Y$ into $Y \times Y$, viewed as a rigid analytic space over $k$.  The refined Swan conductor homomorphism defined by Saito makes use of the stable formal model of such an object, whereas our definition using differential modules is close related to some object over the generic point of a smooth model over $\calO_K$ lifting the aforementioned rigid space.  The crucial calculation we performed in Subsection~\ref{S:Dwork-isoc} relates these objects, in which case, it boils down to some explicit computation on a higher dimensional analog of the Artin-Scheier cover, and on the associated $\ell$-adic sheaves and overconvergent $F$-isocrystals.

We assume $p>0$ is a prime number throughout this section.

\subsection{Review of Saito's definition}
\label{S:saito-refined}

In this subsection, we review the definition of the ramification filtrations and the refined Swan conductors defined by Abbes and Saito in \cite{abbes-saito1, abbes-saito2, saito-wild-ram}.  Instead of introducing the general construction, we will focus on a special case which is used in the comparison theorem. 
For more details and a complete treatment, one may consult \cite{saito-wild-ram}.

\begin{construction}\label{C:saito}

Let $l$ be a finite Galois extension of $k$. We consider a closed immersion $\Spec \calO_l \rar P$ into a smooth (affine) scheme $P$ over $\Spec \calO_k$.  Put $\calI = \Ker(\calO_P \rar \calO_l)$.

For $r =a/b \in \QQ_{>0}$ with $a, b >0$, let $P^{[a/b]}_{\calO_k} \rar P$ be the blowup at the ideal $\calI^b + \gothm_k^a \calO_P$ and let $P^{(a/b)}_{\calO_k} \subset P^{[a/b]}_{\calO_k}$ be the complement of the support of $(\calI^a \calO_{P^{[a/b]}_{\calO_k}} + \gothm_k^b\calO_{P^{[a/b]}_{\calO_k}}) / \gothm_k^b\calO_{P^{[a/b]}_{\calO_k}}$.  Let $P^{(r)}_{\calO_k}$ be the normalization of $P^{(a/b)}_{\calO_k}$; it does not depend on $a$ and $b$ but only their ratio.  Let $P^{(r)}_k$ and $P^{(r)}_\kappa$ denote the generic fiber and the special fiber of $P^{(r)}_{\calO_k}$, respectively. 
Let $\widehat{P^{(r)}_k}$ denote the generic fiber of completing $P^{(r)}_{\calO_k}$ along $P^{(r)}_\kappa$.
The immersion $\Spec \calO_l \rar P$ is uniquely lifted to an immersion $\Spec \calO_l \rar P^{(r)}_{\calO_k}$.

By the finiteness theorem of Grauert-Remmert cited in \cite[THEOREM~1.10]{abbes-saito2}, there exists a finite \emph{separable} extension $k'/k$ of na\"ive ramification degree $e = e_{k'/k}$ such that the normalization $P^{(er)}_{\calO_{k'}}$ of $P^{(r)}_{\calO_k} \times_{\calO_k} \calO_{k'}$ has reduced geometric fibers over $\Spec \calO_{k'}$, which we call a \emph{stable model} of $P^{(r)}_{\calO_k}$. We put $P^{(r)}_{\bar \kappa} = P^{(er)}_{\calO_{k'}} \times_{\calO_{k'}} \bar\kappa$; it is called the \emph{stable special fiber} of $P^{(r)}_{\calO_k}$ and it does not depend on the choice of $k'$.
\end{construction}

We defer the discussion of the properties of this construction until later when we have a concrete example at hand.

For the rest of this section, we assume the following geometric assumption.

\begin{hypothesis}[Geom]
There exists an affine smooth variety $X$ over $k_0$ and an irreducible divisor $D$, smooth over $k_0$ with generic point $\xi$, such that $\calO_k \cong \calO_{X, \xi}^\wedge$, where the latter is the completion of the local ring at $\xi$.  In particular, Hypothesis~\ref{H:J-finite-set} is fulfilled.
\end{hypothesis}

\begin{remark}
This Hypothesis (Geom) is essentially the same as the hypothesis (Geom) in \cite[P.786]{saito-wild-ram}, except that our $k$ is the completion of the Henselian local field considered in Saito's paper.
\end{remark}


\begin{construction}
 After replacing $X$ (and hence $D$) by an \'etale neighborhood of $\xi$ if necessary, there exists a finite flat morphism $f: Y \rar X$ of smooth schemes over $k_0$ such that $V = Y \times_X U \rar U = X \bs D$ is finite \'etale with Galois group $G_{l/k}$ and that $Y \times_X \Spec \calO_{X, \xi}^\wedge = \Spec \calO_l$.

Let $(X \times X)'$ be the blowup of $X \times_{k_0} X$ along $(X \times_{k_0} D) \cup (D \times_{k_0} X)$, and let $(X \times X)^\sim$ denote the complement of the proper transforms of $X \times_{k_0} D$ and $ D \times_{k_0} X$ in $(X \times X)'$.  The diagonal embedding $\Delta_X: X \rar X \times_{k_0} X$ naturally lifts to an embedding $\tilde \Delta_X: X \rar (X \times X)^\sim$.  Now, pulling back the whole picture along $f: Y \rar X$ gives the following commutative diagram
\begin{equation}\label{E:geometric-over-k_0}
\xymatrix{
&&& (Y \times X)^\sim \ar[r]^{f \times 1} \ar[dd]_{\pi_Y} & (X \times X)^\sim \ar[dd]^{\pi_X} \\
Y \ar[rr]^{\ \ f} \ar[drrr]_{\Delta_Y} \ar[urrr]^{\tilde \Delta_Y} && X  \ar[drr]^{\Delta_X} \ar[urr]_{\tilde \Delta_X}\\
&& & Y \times_{k_0} X \ar[r]_{f \times 1} \ar[d]_{p_1} & X \times_{k_0} X \ar[d]^{p_1} \ar[r]^-{p_2} & X\\
& && Y \ar[r]^f & X
}
\end{equation}
where $(Y \times X)^\sim$ is the fiber product of the big square, and all parallelograms are Cartesian.

 Put $P = (X \times X)^\sim \times_{p_2\circ \pi_X, X} \Spec \calO_{X, \xi}^\wedge$ and $Q = (Y \times X)^\sim \times_{p_2\circ \pi_X\circ (f \times 1), X} \Spec \calO_{X, \xi}^\wedge$. Taking the Cartesian product of the top part of \eqref{E:geometric-over-k_0} with $\Spec \calO_{X, \xi}^\wedge = \Spec \calO_k$ over $X \times_{k_0} X$ along $p_2$ then gives the following commutative diagram.
\begin{equation}\label{E:geometric-over-Ok}
\xymatrix{
\Spec \calO_l \ar[d]_f \ar[r]^-{\tilde \Delta_Y} & Q \ar[d]^{f \times 1} \\
\Spec \calO_k \ar[r]^-{\tilde \Delta_X} & P \ar[r]^-{p_2} & \Spec \calO_k
}
\end{equation}
Let $\calI$ denote the ideal of the immersion $\tilde \Delta_X$.  We will view $P$ and $Q$ as  schemes over $\calO_k$ via $p_2$.

We can now apply Construction~\ref{C:saito} to the embeddings $\tilde \Delta_X$ and $\tilde \Delta_Y$ to define $P_{\calO_k'}^{(er)}$, $P_{k'}^{(er)}$, $\widehat{P_{k'}^{(er)}}$, $P_{\bar \kappa}^{(er)}$ and $Q_{\calO_k'}^{(er)}$, $Q_{k'}^{(er)}$, $\widehat{Q_{k'}^{(er)}}$, $Q_{\bar \kappa}^{(er)}$, respectively, where $k'/k$ is a finite separable extension of na\"ive ramification degree $e$.  
We still use $p_1$ to denote the morphism $P^{(er)}_{\calO_{k'}} \rar P \stackrel {p_1} \lrar \Spec \calO_k$.  By functoriality of Construction~\ref{C:saito}, we have a morphism $f^{(r)}: Q^{(er)}_{\calO_{k'}} \rar P^{(er)}_{\calO_{k'}}$.
\end{construction}

\begin{remark}
The field extension $k'$ serves as the role of a ``coefficient field"; we only use it to provide reasonable integral structures of our spaces over $\calO_{k'}$, and also to make $er$ an integer.  We can make $k'$ as large as we need.

In contrast, the extension $l/k$ pulled back from $p_1: X \times_{k_0} X \to X$ encodes the arithmetic information.
\end{remark}

We collect together some properties of these spaces.

\begin{proposition}\label{P:saito-properties}
Let $k'/k$ be a finite separable extension of na\"ive ramification degree $e$.
\begin{itemize}
\item[\emph{(a)}] 
When $er$ is an integer, the space $P^{(er)}_{\calO_{k'}}$ is defined to be 
$\sum_{i \geq 0} \pi_{k'}^{-ier} \cdot \calI^i \subset \calO_P \otimes_{\calO_k} k'$.  It is smooth over $\calO_{k'}$, and its
closed fiber $P^{(er)}_{\kappa_{k'}}$ can be canonically identified with the $\kappa_{k'}$-vector space $\Omega^1_{\calO_k}(\log) \otimes_{\calO_k} \pi_{k'}^{-er} \kappa_{k'}$.  The rigid space $\widehat{P^{(er)}_{k'}}$ is isomorphic to $\Sp \big(k'\langle \pi_{k'}^{-er-e}\delta_0, \pi_{k'}^{-er} \delta_J\rangle \big)$, where $\delta_0, \dots, \delta_m$ form a dual basis of $\Omega^1_{\calO_k}$.
\item[\emph{(b)}]
The generic fiber $Q^{(er)}_{k'}$ of $Q^{(er)}_{\calO_{k'}}$ is isomorphic to $ P^{(er)}_{k'} \otimes_{p_1, k} l$. In particular, $Q^{(er)}_{k'}$ is finite and \'etale over $P^{(er)}_{k'}$ with Galois group $G_{l/k}$, and the same is true for $\widehat{Q^{(er)}_{k'}}$ over $\widehat{P^{(er)}_{k'}}$.
\item[\emph{(c)}] Let $\Spf \calO_{Q^\wedge}$ be the completion of $Q$ along $\Spec \calO_l$.  If $er$ is an integer, then $\widehat{Q^{(er)}_{k'}}$ is the affinoid variety $X_\log^j(\calO_{Q^\wedge} \rar \calO_l)_{k'}$ defined in \cite[Section~4.2]{abbes-saito2} for $j = r$.
\item[\emph{(d)}] If the highest log-break $b_\log(l/k)$ is less than or equal to $r$, then $Q^{(r)}_{\bar \kappa}$ is an element of the category $(\FE/ P^{(r)}_{\bar\kappa})^\alg$, defined below in Definition~\ref{D:FE-alg}.
\item[\emph{(e)}] The highest log-break $b_\log(l/k)$ is strictly less than $ r$ if and only if the number of connected components of $Q^{(r)}_{\bar  \kappa}$ is $[l:k]$.
\end{itemize}
\end{proposition}
\begin{proof}
For (a), see \cite[Lemma~1.10]{saito-wild-ram}. The claim (b) follows from the fact that $f: V \rar U$ is finite and \'etale with Galois group $G_{l/k}$. For (c), see \cite[Example~1.21]{saito-wild-ram}.  The statements (d) and (e) follow from \cite[Lemma~1.13 and Theorem~1.24]{saito-wild-ram}.
 \end{proof}

\begin{definition}\label{D:FE-alg}
For an $\bar\kappa$-vector space $W$ of finite dimensional, let $(\FE/W)^\alg$ be the full subcategory of $(\FE/W)$ whose objects are finite \'etale morphisms $g: Z \rar W$ such that $Z$ admits a structure of algebraic group scheme and such that $g$ is a morphism of algebraic groups.
\end{definition}

\begin{remark}\label{R:FE-alg}
By the argument just before \cite[Lemma~1.23]{saito-wild-ram}, the category $(\FE/W)^\alg$ is a Galois category associated to the Galois group $\pi_1^\alg(W)$, which is a quotient of the fundamental group $\pi_1(W)$.  This group can be identified with the Pontrjagin dual of the extension group $\Ext^1(W, \Fp)$ in the category of smooth algebraic groups  over $\bar\kappa$.
 The map $W^\dual  = \Hom_{\bar\kappa}(W, \bar\kappa) \rar \Ext^1(W, \Fp)$ sending a linear form $f: W\rar \AAA^1_{\bar\kappa}$ to the pullback along $f$ of the Artin-Scheier sequence $0 \rar \Fp \rar \AAA^1_{\bar\kappa} \stackrel {t \rar t^p-t} \lrar \AAA^1_{\bar\kappa} \rar 0$ is an isomorphism.
\end{remark}

\begin{proposition}\label{P:rsw'}
We have a surjective homomorphism $\pi_1^\alg(P^{(b)}_{\bar \kappa}) \surj \Fil^b_\log G_k / \Fil^{b+}_\log G_k$; it induces an injective homomorphism
\[
\rsw':
\Hom(\Fil^b_\log G_k / \Fil^{b+}_\log G_k, \Fp) \lrar \Omega^1_{\calO_k}(\log) \otimes_{\calO_k} \pi_k^{-b} \bar\kappa.
\]
\end{proposition}
\begin{proof}
For the first half of the proposition, see \cite[Theorem~1.24]{saito-wild-ram}.  The second half
follows from Remark~\ref{R:FE-alg}.
\end{proof}

In the following special case, we give a more detailed description of these spaces.

\begin{situation}
Let $l/k$ be a finite totally ramified Galois extension, which is \emph{not} tamely ramified.  Assume that the highest log-break $b=b_\log(l/k)$ is a positive integer.  Assume moreover that $\Fil^{b-1}_\log G_k /(\Fil^{b-1}_\log G_k \cap G_l) \simeq \Fp$; in particular, the second highest log-break $b_\log(l/k, 2)$ is strictly less than $ b_\log(l/k)-1$.  By Proposition~\ref{P:rsw'}, $Q_{\bar \kappa}^{(b)}$ consists of $[l:k]/p$ copies of the \emph{same} Artin-Scheier cover of $P_{\bar \kappa}^{(b)}$, at least if we forget about the algebraic group structure.  Assume that this cover is given by 
\begin{equation}\label{E:z-Q-kappa}
\bar z^p - \bar z + \big(\bar \alpha_0 \pi_k^{-b-1}\delta_0 + \bar \alpha_1 \pi_k^{-b} \delta_1+ \cdots + \bar \alpha_m \pi_k^{-b}\delta_m \big)=0
\end{equation}
for some $\bar \alpha_{J^+} \in \bar \kappa$,
where the coordinates of $P_{\bar \kappa}^{(b)}$ are given by $\pi_k^{-b-1}\delta_0$ and $ \pi_k^{-b}\delta_J$. These elements $\bar \alpha_0, \dots, \bar \alpha_m$ are determined up to multiplication by $i \in \Fp^\times$, in accordance with the choice of $\bar z$ up to multiplication by the same $i\in \Fp^\times$.

Let $k'/k$ be a finite separable extension of ramification degree $e>1$, such that $Q^{(eb)}_{\calO_{k'}}$ is a stable model.  By possibly enlarging $k'$, we may assume that $\bar \alpha_{J^+} \in \kappa_{k'}$ and that $Q_{\kappa_{k'}}^{(eb)}$ is the disjoint union of $[l:k]/p$ copies of the aforementioned Artin-Scheier cover of $P_{\kappa_{k'}}^{(eb)}$.
\end{situation}

\begin{lemma}
\label{L:Rkeb-1 definition}
The space $Q^{(eb)}_{\calO_{k'}}$ is the disjoint union of $[l:k]/p$ copies of the same space $R^{(eb)}_{\calO_{k'}}$.  Let $\widehat {R^{(eb)}_{\calO_{k'}}}$ denote the completion of $R^{(eb)}_{\calO_{k'}}$ along its special fiber and let $\widehat {R^{(eb)}_{k'}}$ denote the generic fiber, viewed as a rigid analytic space.  Then $\widehat{Q^{(eb-1)}_{k'}}$ is the disjoint union of $[l:k]/p$ copies of a same space $\widehat{R^{(eb-1)}_{k'}}$, which is the normal closure of $\widehat{P^{(eb-1)}_{k'}}$ in $\widehat{R^{(eb)}_{k'}}$ and is finite and \'etale over $\widehat{P^{(eb-1)}_{k'}}$.
\end{lemma}
\begin{proof}
There is a $G_{l/k}$-equivariant one-to-one correspondence between the connected components of $Q_{\kappa_{k'}}^{(eb)}$ and the connected components of $Q^{(eb)}_{\calO_{k'}}$.

Since the second highest log-break $b_\log(l/k; 2)$ is strictly less than $ b_\log(l/k) - 1$, by \cite[Remark~3.13]{abbes-saito1}, the number of connected components of $\widehat{Q^{(eb-1)}_{k'}}$ is $[l:k]/p$. Note that each connected component of $\widehat{Q^{(eb-1)}_{k'}}$, which is automatically finite and \'etale over $\widehat{P^{(eb-1)}_{k'}}$, can be also characterized as the normal closure of $\widehat{P^{(eb-1)}_{k'}}$ in $\widehat{R^{(eb)}_{k'}}$; this normal closure is the space $\widehat{R^{(eb-1)}_{k'}}$ we sought for.
\end{proof}

\begin{proposition}\label{P:form-z}
Let $\alpha_{J^+} \subset \calO_{k'}$ lift $\bar \alpha_{J^+} \subset \kappa_{k'}$.
We can choose a lift $z$ of $\bar z$ to $\widehat{R^{(eb)}_{\calO_{k'}}}$ such that its minimal polynomial over $\widehat{P^{(eb)}_{\calO_{k'}}} = \Spf \calO_{k'}\langle \pi_{k'}^{-eb-e} \delta_0, \pi_{k'}^{-eb}\delta_J \rangle$ is
\begin{equation}\label{E:z-Q-Ok}
z^p - z + \big(\alpha_0 \pi_{k'}^{-eb-e}\delta_0 + \alpha_1 \pi_{k'}^{-eb} \delta_1+ \cdots + \alpha_m \pi_{k'}^{-eb}\delta_m \big)=0.
\end{equation}
Then the element $z$ generates $\widehat{R^{(eb)}_{\calO_{k'}}}$ over $\widehat{P^{(eb)}_{\calO_{k'}}}$.  Moreover, the element $z$ extends to a section over $\widehat{R^{(eb-1)}_{k'}}$ and it generates $\widehat{R^{(eb-1)}_{k'}}$ over $\widehat{P^{(eb-1)}_{k'}}$.
\end{proposition}
\begin{proof}
We first pick any lift $z'$ of $\bar z$ to $\widehat{R^{(eb)}_{\calO_{k'}}}$; it must satisfy an equation of the form
$
z'^p + a_1 z'^{p-1} + \cdots + a_p = 0
$, where $a_1, \dots, a_p \in \calO_{k'}\langle \pi_{k'}^{-eb-e} \delta_0, \pi_{k'}^{-eb}\delta_J \rangle$ and the reduction of this equation is exactly \eqref{E:z-Q-kappa}. For the given $\alpha_{J^+} \subset \calO_{k'}$, we have 
\[
\epsilon = z'^p - z' +\big(\alpha_0 \pi_{k'}^{-eb-e}\delta_0 + \alpha_1 \pi_{k'}^{-eb} \delta_1+ \cdots + \alpha_m \pi_{k'}^{-eb}\delta_m \big) \in  \pi_{k'} \calO_{\widehat {R^{(eb)}_{\calO_{k'}}}}.
\] 
Now, $z = z' + \epsilon + \epsilon^p + \epsilon^{p^2} + \cdots$ converges and satisfies \eqref{E:z-Q-Ok}.  

Since $z$ generates a subalgebra of $\calO_{\widehat{R^{(eb)}_{ \calO_{k'}}}}$ which is finite and \'etale over $\calO_{\widehat{P^{(eb)}_{\calO_{k'}}}}$ of the same degree $p$, this subalgebra has to equal   $\calO_{\widehat{R^{(eb)}_{\calO_{k'}}}}$.

For the similar statement for $eb-1$ in place of $eb$, we argue as follows. Since $\widehat{R^{(eb-1)}_{k'}}$ is the normal closure of $\widehat{P^{(eb-1)}_{k'}}$ in $\widehat{R^{(eb)}_{k'}}$ by Lemma~\ref{L:Rkeb-1 definition}, the element $z$ extends to a section over $\widehat{R^{(eb-1)}_{k'}}$ with the \emph{same} minimal equation \eqref{E:z-Q-Ok}.  Again, since $z$ generates a subalgebra of $\calO_{\widehat{R^{(eb-1)}_{k'}}}$ which is finite and \'etale over $\calO_{\widehat{P^{(eb-1)}_{k'}}}$ of same degree, it has to generate the whole ring.  This finishes the proof.
\end{proof}

\subsection{Lifting rigid spaces}
\label{S:lifting-spaces}

The definition of the refined Swan conductor homomorphism using differential modules needs to work with spaces and modules over the field $K$.  Following the idea of \cite{xiao1}, we formally lift the picture from $k$ to some annulus $A_K^1[\eta, 1)$.  This construction is a local version of Berthelot's definition of rigid cohomology.

\begin{construction}
Replacing $X$ by an open Zariski neighborhood of $\xi$ if necessary, there exists a finite morphism $\bbf: \bbY \rar \bbX$ between two affine smooth formal schemes of topologically finite type over $\calO_{K_0}$, such that $\bbf$ reduces to $f$ modulo $p$ and such that the induced map $\bbY \bs f^{-1}(D) \rar \bbX \bs D$ is finite \'etale with Galois group $G_{l/k}$.  In particular, the special fibers of $\bbX$ and $\bbY$ are $X$ and $Y$, respectively.

Let $\blacktriangle_X: \bbX \rar \bbX \times_{\Spf \calO_{K_0}} \bbX$ be the diagonal embedding, and put $\blacktriangle_Y = (\mathrm{id}, \bbf): \bbY \rar \bbY \times_{\Spf \calO_{K_0}} \bbX$.  Let $p_1$ and $p_2$ denote the projections from $\bbX \times_{\Spf \calO_{K_0}} \bbX$ to the first and the second factors, respectively.

Let $\bbX^\wedge$ denote the completion of $\bbX \times_{\Spf \calO_{K_0}} \bbX$ along the diagonal embedding $\blacktriangle_X$; it can be identified with the completion of the cotangent bundle of $\bbX$ along its zero section.  Set $\bbY^\wedge = \bbX^\wedge \otimes_{p_1, \bbX} \bbY$; it is the same as the completion of $\bbY \times_{\Spf \calO_{K_0}} \bbX$ along the embedding $\blacktriangle_Y$.

For $\eta \in (0,1)$, we set $\calR_{K, \eta}^\inte$ to be the subring of $\calR_K^\eta$ consisting of elements having $1$-Gauss norm $\leq 1$; it is complete with respect to the $\eta'$-Gauss norm for $\eta' \in [\eta,1]$.  On one hand, this ring does not give rise to a formal scheme; on the other hand, it is good to keep the geometric intuition.  Hence we introduce the \textit{geometric incarnation} $\Sp \calR^\inte_{K, \eta}$, which is just a symbol.  Any morphism between geometric incarnations should be thought of as ring homomorphisms; in particular, the fiber product is simply the (completed) tensor product.  We also point out that we will only consider affine schemes and there is no question of gluing.

We may compare the following commutative diagram with \eqref{E:geometric-over-k_0}.
\begin{equation}\label{E:geometric-over-OK_0}
\xymatrix{
\bbY \ar[r]^\bbf \ar[d]_{\blacktriangle_Y} & \bbX \ar[d]^{\blacktriangle_X} \\
\bbY^\wedge \ar[r]^-{\bbf \times 1} \ar[d] & \bbX^\wedge \ar[d]^{p_1} \ar[r]^-{p_2} & \bbX & \Sp \calR_{K, \eta}^\inte \ar[l]_-{i}\\
\bbY \ar[r]^\bbf& \bbX
}
\end{equation}
where $i:\Sp \calR_{K, \eta}^\inte \rar \bbX$ is the geometric incarnation of the natural homomorphism $\calO_{\bbX, \xi}^\wedge \rar \calR^\inte_{K, \eta}$, for some $\eta \in (0, 1)\cap p^\QQ$.  We have $\Sp\calR_{K, \eta}^\inte \times_\bbX \bbY = \Sp\calR_{L, \eta^{1/e_{l/k}}}^\inte$ for $\eta$ sufficiently close to $1^-$.
Put
\[
\bbP_\eta = \bbX^\wedge \times_{p_2, \bbX, i} \Sp \calR_{K, \eta}^\inte\textrm{ and } \bbQ_\eta =  \bbY^\wedge \times_{p_2\circ (\bbf \times 1), \bbX, i} \Sp\calR_{K, \eta}^\inte.
\]
Again, both $\bbP_\eta$ and $\bbQ_\eta$ should be thought of as geometric incarnations of $\calO_{\bbP_\eta}$ and $\calO_{\bbQ_\eta}$, the completed tensor products of corresponding rings of functions.
We then have the following Cartesian diagram
\begin{equation}\label{E:geometric-over-RK}
\xymatrix{
\Sp\calR_{L, \eta^{1/e_{l/k}}}^\inte \ar[d]_\bbf \ar[r]^-{\blacktriangle_Y} & \bbQ_\eta \ar[d]^{\bbf \times 1} \\
\Sp\calR_{K, \eta}^\inte \ar[r]^-{\blacktriangle_X} & \bbP_\eta
}
\end{equation}
\end{construction}

\begin{lemma}\label{L:p1-differential}
The morphism $p_1: \bbP_\eta \rar \Sp (\calR^\inte_{K, \eta})$ is given by the continuous homomorphism $\psi:\calR^\inte_{K, \eta} \rar \calR^\inte_{K, \eta} \llbracket \delta_0 / T, \delta_1, \dots, \delta_m \rrbracket$ such that $\psi(T) = T+ \delta_0$, $\psi(B_j) = B_j + \delta_j$ for $j \in J$.  More precisely, for $x \in \calR^\inte_{K, \eta}$, we have
\[
\psi(x) = \sum_{e_{J^+} = 0}^{+\infty} \frac{\partial^{e_{J^+}}_{J^+}(x)}{(e_{J^+})!}
\delta_{J^+}^{e_{J^+}}.
\]
\end{lemma}
\begin{proof}
The first statement follows from the description of $\bbX^\wedge$ above and the second statement follows by the uniqueness of such a homomorphism.
\end{proof}

\begin{construction}
Let $k'/k$ be a finite separable extension of na\"ive ramification degree $e$. Since $\calR^\inte_K$ is Henselian, there exists $\calR^\inte_{K'}$ corresponding to the extension $k'/k$, where $K'$ is the fraction field of a Cohen ring of $\kappa_{k'}$.  For $\eta$ sufficiently close to $1^-$, the extension $\calR^\inte_{K'}$ of $ \calR^\inte_{K}$ descends to a finite \'etale algebra $\calR^\inte_{K', \eta^{1/e}}$ over $\calR^\inte_{K, \eta}$ for some $\eta$ sufficiently close to $1$.  Fix such an $\eta$.  Let $T'$ denote the coordinate of $\calR^\inte_{K', \eta^{1/e}}$.

Let $r \in \NN$ (be a proxy of $eb$ or $eb-1$).
Let $\bbP^{(r)}_{K', \eta} = \Sp \big( \calR_{K', \eta^{1/e}}^\inte \langle T'^{-r-e} \delta_0, T'^{-r}\delta_J \rangle\big)$ be the geometric incarnation of a closed polydisc over $\Sp\calR_{K', \eta^{1/e}}^\inte$; it may be viewed as a subspace of $\bbP_\eta$ (in the sense of geometric incarnation).  Let $\bbQ^{(r)}_{K', \eta}$ be the  preimage (in the sense of geometric incarnation) of $\bbP^{(r)}_{K', \eta}$ under the morphism $\bbQ_\eta \to \bbP_\eta$.
\end{construction}

\begin{proposition}\label{P:calF=p1*calE}
Let $\rho$ be a $p$-adic representation of $G_{l/k}$.  Let $\calF_\rho = \big((\bbf \times 1)_* \calO_{\bbQ_\eta} \otimes V_\rho\big)^{G_{l/k}}$ be the differential module over $\bbP_\eta$ and for $r \in \NN$, let $\calF_{\rho, K'}^{(r)} = \big((\bbf \times 1)_* \calO_{\bbQ_{K', \eta}^{(r)}} \otimes V_\rho\big)^{G_{l/k}}$ be the corresponding differential module over $\bbP_{K', \eta}^{(r)}$.  Then $\calF_\rho$ and $\calF_{\rho, K'}^{(r)}$ are the pullbacks of $\calE_\rho$ along $p_1: \bbP_\eta \rar \Sp\calR^\inte_{K, \eta}$ and $p_1: \bbP_{K', \eta}^{(r)} \rar \Sp\calR^\inte_{K, \eta}$, respectively.
\end{proposition}
\begin{proof}
This follows from the following $G_{l/k}$-equivariant Cartesian diagram of geometric incarnated morphisms.
\[
\xymatrix{
\bbQ_{K', \eta}^{(r)} \ar[r] \ar[d]^{\bbf \times 1} &
\bbQ_\eta \ar[r]^-{p_1}  \ar[d]^{\bbf \times 1} &
\Sp \calR^\inte_{L, \eta^{1/e_{l/k}}} \ar[d]^\bbf \\
\bbP_{K', \eta}^{(r)} \ar[r]  &
\bbP_\eta \ar[r]^-{p_1}   &
\Sp \calR^\inte_{K, \eta} 
}
\]
\end{proof}

\begin{corollary}\label{C:refined-calF-vs-calE}
For $a\in \QQ_{<b}$ and $\eta \in (0,1) \cap p^\QQ$, let $F_{\eta, a}$ denote the completion of $K(T, \delta_{J^+})$ with respect to the $(\eta, \eta^{a+1}, \eta^a, \dots, \eta^a)$-Gauss norm and let $F'_{\eta, a} = F_{\eta, a} \otimes_{\calR^\inte_{K, \eta}} \calR^\inte_{K', \eta^{1/e}}$.
Assume that $\rho$ has pure log-break $b$ and pure refined Swan conductor $\vartheta = \pi_k^{-b} \big( \bar\alpha_0 \frac{d\pi_k}{\pi_k} + \bar\alpha_1 \frac{d\bar b_1}{\bar b_1} + \cdots + \bar\alpha_m \frac{d\bar b_m}{\bar b_m} \big)$, where $\bar\alpha_{J^+} \in \bar \kappa$.  If $r<ea<eb$ and $\eta$ is sufficiently close to $1^-$, then $\calF_\rho \otimes F'_{\eta, a} = \calF_{\rho, K'}^{(r)} \otimes F'_{\eta, a}$ as a $\partial / \partial\delta_{J^+}$-differential module has pure intrinsic radii $\eta^b$ and pure refined intrinsic radii
\[
T^{-b} (\bar\alpha_0 \frac{d\delta_0}T + \bar\alpha_1 \frac{d\delta_1}{B_1} + \cdots + \bar\alpha_m \frac{d\delta_m}{B_m} ).
\]
\end{corollary}
\begin{proof}
By 
Lemma~\ref{L:p1-differential} and Proposition~\ref{P:calF=p1*calE}, $\calF^{(r)}_{\rho, K'}$ is the pullback of $\calE_\rho$ along the multidimensional analog of the generic point homomorphism as in Corollary~\ref{C:refined-radii-for-generic-point}.
However, the calculation of the refined $\partial_j$-radii can be computed independently for each of the $\partial_j$.  Hence the statement follows from Corollary~\ref{C:refined-radii-for-generic-point}.
\end{proof}

Before proceeding, we briefly recall the lifting construction in \cite[Section~1]{xiao1}, which lifts a rigid analytic space over $\kappa_{k'}$ to a rigid analytic space over $A_{K'}^1[\eta^{1/e}, 1)$ for 
$\eta \in p^\QQ \cap (0,1)$ sufficiently close to $1^-$.

\begin{construction}
Let $Z$ be a rigid analytic space over $k'$ with ring of analytic functions $A_{k'} = k' \langle u_1, \dots, u_s \rangle  / I_{k'}$.  Let $I_{K'} \subset \calO_{K'}\langle u_1, \dots, u_s \rangle((T'))$ be an ideal such that $\calO_{K'}\langle u_1, \dots, u_s \rangle((T')) / I_{K'}$ is \emph{flat} over $\calO_{K'}$ and $I_{K'} \otimes_{\calO_{K'}} k' = I_{k'}$.  We call $\bbX_\eta = \Spf \big( \calR_{K', \eta}^\inte \langle u_1, \dots, u_s\rangle / I_{K'} \big)$ a \emph{lifting space} of $X$.
\end{construction}

\begin{proposition}\label{P:lifting-PQR}
Fix $r \in \NN$.

\begin{itemize}
\item [\emph{(a)}] The space $\bbQ^{(r)}_{K', \eta}$ is a lifting space of $\widehat{Q^{(r)}_{k'}}$.

\item [\emph{(b)}] Suppose
that $Q^{(r)}_{\calO_{k'}}$ is a stable model and $r = eb$ or $eb-1$.  Then for $\eta$ sufficiently close to $1^-$, $\bbQ^{(r)}_{K', \eta}$ has $[l:k]/p$ connected components, each of which is isomorphic to a formal scheme $\bbR^{(r)}_{K', \eta}$ finite and \'etale over $\bbP^{(r)}_{K', \eta}$ of degree $p$.

\item [\emph{(c)}] Fix a Dwork pi $\boldsymbol \pi = (-p)^{1/(p-1)}$ and fix $\boldsymbol \alpha_{J^+} \subset \calR_{K'(\bbpi), \eta^{1/e}}^\inte$ lifts of $\alpha_{J^+}$.  By making $\eta$ closer to $1^-$ if needed, we may choose a lift $\bbz$ of $z$ to $\bbR^{(r)}_{K'(\bbpi), \eta}$ whose minimal polynomial over $\bbP_{K', \eta}^{(r)}$ is of the form
\begin{equation}\label{E:z-Q-RK}
\frac 1{p \boldsymbol \pi} \Big((1+ \boldsymbol \pi \bbz)^p - 1 - p\boldsymbol \pi (\boldsymbol \alpha_0 T'^{-eb-e} \delta_0 + \boldsymbol \alpha_1 T'^{-eb} + \cdots + \boldsymbol \alpha_m T'^{-eb}) \Big) = 0.
\end{equation}
\end{itemize}
\end{proposition}
\begin{proof}
The first statement follows from the construction.  The second statement follows from \cite[Proposition~1.2.11]{xiao1}; the fact that all the connected components are isomorphic to the same $\bbR^{(r)}_{K'(\bbpi), \eta}$ is a corollary of (c), proved below.

For (c), pick a lift $\bbz_1$ of $z$ to $\bbR^{(r)}_{K'(\bbpi), \eta}$ whose minimal polynomial reduces to \eqref{E:z-Q-Ok} modulo $\bbpi$. (Note that $K$ is absolutely unramified.)  We define the following substitution process.  Assume that we have defined $\bbz_i$.  We set
\[
\boldsymbol \lambda_i = \frac 1{p \boldsymbol \pi} \Big((1+ \boldsymbol \pi \bbz_i)^p - 1 - p\boldsymbol \pi (\boldsymbol \alpha_0 T'^{-eb-e} \delta_0 + \boldsymbol \alpha_1 T'^{-eb} + \cdots + \boldsymbol \alpha_m T'^{-eb}) \Big).
\]
and set $\bbz_{i+1} = \bbz_i - \boldsymbol \lambda_i$.  Hence we have 
\begin{align*}
\boldsymbol \lambda_{i+1} &= \frac 1{p\bbpi}\Big( (1+\bbpi \bbz_i- \bbpi\boldsymbol \lambda_i)^p - (1+\bbpi \bbz_i)^p + p \bbpi \boldsymbol \lambda_i\Big) \\
&= \big(1-(1+\bbpi \bbz_i)^{p-1})\boldsymbol \lambda_i + \sum_{n=2}^{p-1} \frac 1{p\bbpi} \binom pn (1+\bbpi \bbz_i)^{p-n}(-\bbpi\boldsymbol \lambda_i)^n +(-1)^{p-1} \boldsymbol \lambda_i^p.
\end{align*}
Since $|\boldsymbol \lambda_1|_1 \leq p^{-1/(p-1)}$, by continuity, $|\boldsymbol \lambda_1|_\eta<1$ for $\eta \in [\eta_0, 1]$ for some $\eta_0$ sufficiently close to $1^-$.  Thus,
\[
|\boldsymbol \lambda_{i+1}|_\eta \leq \max \big\{p^{-1/(p-1)}|\boldsymbol \lambda_i|_\eta, 
|\boldsymbol \lambda_i|_\eta^p\big\} \quad \textrm{ for } \eta \in [\eta_0, 1).
\]
As a consequence, this substitution process converges with respect to all $\eta$-Gauss norms for $\eta \in [\eta_0, 1]$.  The limit $\bbz = \lim_{i \rar +\infty} \bbz_i$ satisfies \eqref{E:z-Q-RK}.  By the same argument as in Proposition~\ref{P:form-z}, the limit $\bbz$ generates $\bbR^{(r)}_{K'(\bbpi), \eta}$ over  $\bbP^{(r)}_{K'(\bbpi), \eta}$ when $\eta$ is sufficiently close to $1^-$.
\end{proof}

\subsection{Dwork isocrystals}
\label{S:Dwork-isoc}

In this subsection, we single out a calculation of refined radii for the differential modules coming from a higher dimensional Artin-Scheier cover.  This is the heart of the comparison Theorem~\ref{T:comparison}.  We will state it in a slightly general form because it has its own interest in the study of differential modules.

\begin{hypothesis}
In this subsection, let $K$ be a complete discrete valuation field of characteristic zero, containing $\bbpi$.  Let $\kappa$ denote its residue field, which has characteristic $p>0$.
\end{hypothesis}

\begin{situation}
Let $\bbP$ denote the formal scheme $ \Spf \calR^\inte_{K, \eta} \langle \delta_0, \dots, \delta_m \rangle$, and let $T$ be the coordinate of $R^\inte_{K, \eta}$.  Let $\bbR$ be a finite extension of $\bbP$ generated by $\bbz$ satisfying the relation
\[
(1+ \boldsymbol \pi \bbz)^p =1+ p\boldsymbol \pi  T^{-r}( \boldsymbol \alpha_0 \delta_0 + \cdots + \boldsymbol \alpha_m  \delta_m),
\]
where $r \in \NN$ and $\boldsymbol \alpha_j \in \calR^\inte_{K, \eta}$ for $j = 1, \dots, m$.  Let $\alpha_j \in \kappa$ be the reduction of $\boldsymbol \alpha_j$ for any $j$. We assume that not all $\alpha_j$ are zero.  Let $f: \bbR \rar \bbP$ be the natural morphism, which is finite and \'etale.
\end{situation}
 
\begin{construction}
We reproduce a multi-dimensional version of  the construction in \cite[Lemma~5.4.7]{kedlaya-overview}.
The pushforward $f_*\calO_\bbQ$ decomposes as the direct sum of $p$ differential modules of rank 1, with respect to $\partial_j = \partial / \partial \delta_j$ for $j = 0, \dots, m$.

Let $\calE_i$ be the differential module given by $(1+\bbpi \bbz)^i$ for $i = 1, \dots, p-1$.  (The trivial submodule of $f_* \calO_\bbQ$ is not of interest to us.)
\end{construction}
 
\begin{notation}
For $\eta \in (0, 1)$, let $\bbF_\eta$ be the completion of $K(T, \delta_0, \dots, \delta_m)$ with respect to the $(\eta, 1, \dots, 1)$-Gauss norm.
\end{notation}

\begin{proposition} \label{P:Dwork-isocrystal}
For $\eta$ sufficiently close to $1^-$,
the intrinsic radius $IR(\calE_i \otimes \bbF_\eta)$ is equal to $ \eta^r$ and the refined intrinsic radius of $\calE_i$ for $i = 1, \dots, p-1$ is given by
\[
\calI\Theta(\calE_i \otimes \bbF_\eta) = \big\{i\bbpi T^{-r} (\alpha_0 d \delta_0 + \cdots +\alpha_m d\delta_m)
\big\}.
\] 
\end{proposition}
\begin{proof}
Since
\[
p \frac{d(1+\bbpi\bbz)^i}{(1+ \bbpi \bbz)^i} = i \frac {d \big(1+ p \bbpi T^{-r} (\boldsymbol \alpha_0 \delta_0 + \cdots + \boldsymbol \alpha_m \delta_m) \big)}{ 1+ p \bbpi T^{-r}(
\boldsymbol \alpha_0 \delta_0 + \cdots + \boldsymbol \alpha_m  \delta_m)},
\]
$\calE_i$ is isomorphic to a differential module given by
\[
\nabla \bbv = i \bbpi T^{-r} \big(1+ p \bbpi T^{-r}(
\boldsymbol \alpha_0 \delta_0 + \cdots + \boldsymbol \alpha_m \delta_m) \big)^{-1} \bbv \otimes \big( \boldsymbol \alpha_0  d \delta_0 + \cdots + \boldsymbol \alpha_m d \delta_m \big).
\]

Fix $j =0, \dots, m$.  Using the proof of \cite[Lemma~5.4.7]{kedlaya-overview}, when $\eta$ is sufficiently close to $1^-$ (e.g., $\eta > p^{-1/r}$), viewed as a $\partial_j$-differential module, this is the same as
\[
\partial_j \bbw_j = i \bbpi \boldsymbol \alpha_j T^{-r} \bbw_j,
\]
where $\bbw_j$ is a section of $\calE_i$, dependent on $j$.
Hence we have $\partial_j^n (\bbw_j) = \big(i \bbpi \boldsymbol \alpha_j T^{-r}\big)^n \bbw_j$, and the proposition follows immediately.
\end{proof}

\subsection{Comparison}
\label{S:comparison}

In this subsection, we assemble the results from previous subsections to prove the following comparison theorem.

\begin{theorem}\label{T:comparison}
Assume Hypothesis \emph{(Geom)}.
Then for $b \in \QQ_{>0}$, the homomorphism $\rsw: \Hom( \Fil^b_\log G_k / \Fil^{b+}_\log G_k, \Fp) \rar \Omega^1_k(\log) \otimes \pi_k^{-b} \bar \kappa$ in Theorem~\ref{T:refined-homomorphism} is the same as the homomorphism $\rsw'$ in Proposition~\ref{P:rsw'}.
\end{theorem}
\begin{proof}
Let $\tilde k$ be as in Proposition~\ref{P:saito-base-change}. By \cite[Lemma~1.22]{saito-wild-ram}, $\rsw'$ for $k$ factors as
\[
\Hom( \Fil^b_\log G_k / \Fil^{b+}_\log G_k, \Fp) \rar \Hom( \Fil^{e_{\tilde k/ k}b}_\log G_{\tilde k} / \Fil^{(e_{\tilde k / k}b)+}_\log G_{\tilde k}, \Fp)
\stackrel {\rsw'_{\tilde k}}\lrar \Omega^1_{\calO_{\tilde k}}(\log) \otimes_{\calO_{\tilde k}} \pi_{\tilde k}^{-e_{\tilde k / k}b} \kappa_{\tilde k^\alg}.
\]
The same factorization is also valid for $\rsw$ as in \eqref{E:saito-base-change}. Hence we may choose $e_{\tilde k / k}$ divisible by the denominator of $b$ and reduce to the case when $b$ is an integer.  We also remark that, for the same reason, we may feel free to replace $k$ by a finite tamely ramified extension.

Fix $\zeta_p$ a $p$-th root of unity.  Let $\chi:  \Fil^b_\log G_k / \Fil^{b+}_\log G_k\rar \Fp$ be a nontrivial character and put $\rsw'(\chi) = \pi_k^{-b}(\bar \alpha_0 \frac{d\pi_k}{\pi_k} + \bar \alpha_1 d \bar b_1 + \cdots +\bar \alpha_m d\bar b_m)$, where $\bar \alpha_0, \dots, \bar \alpha_m \in \overline \kappa$.  By identifying $1 \in \Fp$ with $\zeta_p \in \Qp(\zeta_p)$, we get a homomorphism $\Fil^b_\log G_k / \Fil^{b+}_\log G_k \stackrel \chi \rar \Fp \rar \Qp(\zeta_p)^\times$; we still use $\chi$ to denote the composition.  By the argument and the result of Theorem~\ref{T:refined-homomorphism} and by possibly replacing $k$ by a finite tamely ramified extension, we can find a $p$-adic representation $\rho$ of $G_k$ with finite image and pure log-break $b$ such that $\rho|_{\Fil^b_\log G_k}$ is a direct sum of copies of $\chi$.  Moreover, we may assume that $\rho$ is irreducible when restricted to any finite tamely ramified extension of $k'$ of $k$.  The representation $\rho$ factors exact through $l/k$ a finite Galois extension. It must be true that $\Fil^b_\log G_k / G_l \cap \Fil^b_\log G_k \simeq \Fp$.
By possibly making another tamely ramified extension of $k$, we may assume that the second highest log-break of $l/k$ is strictly less than $b-1$; thus $\Fil^{b-1}_\log G_k / G_l \cap \Fil^{b-1}_\log G_k \simeq \Fp$.

We shall now use the results and notation from previous subsections.  By Proposition~\ref{P:lifting-PQR}, $\bbQ_{K', \eta}^{(eb-1)}$ is a disjoint union of $[l:k]/p$ copies of $\bbR_{K', \eta}^{(eb-1)}$, which is finite and \'etale over $\bbP_{K', \eta}^{(eb-1)}$, generated by $\bbz$ with minimal polynomial \eqref{E:z-Q-RK}.  (Here, we made a choice of $z$ and $\bbz$ in accordance with the algebraic group structure on $Q_{\bar \kappa}^b$; see the remarks after \eqref{E:z-Q-kappa}.)
By Proposition~\ref{P:Dwork-isocrystal}, this implies that $\calF_{\rho, K'}^{eb-1} \otimes F'_{\eta, b- 1/2e}$ as $\eta \rar 1^-$ has pure refined intrinsic radii $\bbpi T^{-b}(\bar \alpha_0 \frac{d\delta_0}T + \bar \alpha_1 d\delta_1 + \cdots + \bar \alpha_m d\delta_m)$.  (Here we made a choice of Dwork pi $\bbpi$ so that $\bbpi \equiv \zeta_p-1 \mod (\zeta_p-1)^2$ as in Remark~\ref{R:Dwork-pi-vs-pth-root}.)  By Corollary~\ref{C:refined-calF-vs-calE}, the refined Swan conductor of $\calE_\rho$ has to be $\pi_k^{-b}(\bar \alpha_0 \frac{d\pi_k}{\pi_k} + \bar \alpha_1 d \bar b_1 + \cdots +\bar \alpha_m d\bar b_m)$, the same as $\rsw'$.
\end{proof}

\begin{remark}
By \cite[Theorem~9.1.1]{abbes-saito-micro-local}, the two definitions of  refined Swan conductors above are the same as Kato's definition in \cite{kato}, when the representation is one-dimensional.  So all three definitions agree.  This result is also implicitly contained in \cite{ChPu-cond}.
\end{remark}

\section{Refined Swan conductors and variation of intrinsic radii on polyannuli}
\setcounter{equation}{0}

When we have a differential module over a polyannulus or a polydisc, similar to the one-dimensional situation, we may study how the multiset of intrinsic radii of the differential module change as we complete the module with respect to different Gauss norm.  Kedlaya and the author \cite{kedlaya-xiao} proved that the partial sums of the log of intrinsic radii form continuous convex piecewise affine functions.  The purpose of this section is to prove that the slopes at some point of such affine functions are related to the refined intrinsic radii of the differential module, when completely with respect to the corresponding Gauss norm.  Again, the proof proceeds in two steps, first over an annulus and a disc (Subsection~\ref{S:boundary}) and then over a polyannulus and a polydisc (Subsection~\ref{S:polyannuli-variation}).  The first subsection focuses on some technical results which will be used in the following two subsections.

\begin{hypothesis}
We assume Hypothesis~\ref{H:K-multi} and keep the notation of Section~\ref{Section 1}.  We also assume that $K$ is discretely valued throughout this section.  We do not insist $p>0$ in this section unless otherwise specified.  
\end{hypothesis}

\subsection{Partial decomposition for differential modules}
\label{S:partial-decomposition}

In Subsection~\ref{S:one-dim}, we deliberately restricted ourself to the situation over open annuli.  In many applications, it is also important to understand the theory of differential modules over a bounded analytic ring, e.g. $K \{\{ \alpha / t, t  \rrbracket_0$.  This subsection is devoted to developing a parallel theory in this case, which is not addressed in \cite{kedlaya-xiao}.

We fix some $\alpha \in (0,1)$ for this subsection.

\begin{notation}\label{N:mathsfv_s}
We define $E$ to be the completion of $\Frac \big(K\{\{ \alpha / t, t \rrbracket_0\big)$ with respect to the 1-Gauss norm; it is isomorphic to $\calO_K\llbracket t\rrbracket[\frac 1p]$, and it contains $F_1$ as a subfield.

If $s \in -\log |K^\times|$, we can find an element $x \in K^\times$ with $|x| = \ee^{-s}$.  This $x$ defines an isomorphism $\kappa_E^{(s)} \stackrel {\cdot x^{-1}}{\lrar} \kappa_E \cong \kappa_K((t))$.  Hence we have a canonical valuation $\mathsf v_s(\cdot)$ on $\kappa_E^{(s)}$ given by the $t$-valuation; this does not depend on the choice of $x \in K^\times$.  This valuation extends naturally to $\kappa_{E^\alg}^{(s)}$ for $s \in \QQ \cdot \log |K^\times|$.
\end{notation}

\begin{notation}
Let $j \in J^+$.  For $M$ a $\partial_j$-differential module over $K\{\{ \alpha / t, t \rrbracket_0$ of rank $d$ and $i \in \{1, \dots, d\}$, we put 
$$
f_i^{(j)}(M, 0) = -\log R_{\partial_j}(M \otimes E; i), \quad\textrm{and }\ F_i^{(j)}(M,0) = f_1^{(j)}(M,0) + \cdots + f_i^{(j)}(M,0).
$$
We similarly define $f_i(M,0)$ and $F_i(M,0)$ if $M$ is a $\partial_{J^+}$-differential module over $K\{\{ \alpha / t, t\rrbracket_0$.
\end{notation}

\begin{proposition}\label{P:decomposition-partial}
Fix $j \in J^+$.  Let $M$ be a $\partial_j$- (resp. $\partial_{J^+}$-) differential module of rank $d$ over $K\{\{ \alpha / t, t \rrbracket_0$.  Then we have the following.
\begin{itemize}
\item[\emph{(a)}] The functions $f_i^{(j)}(M, r)$ and $F_i^{(j)}(M, r)$ are continuous, and are affine if $f_i^{(j)}(M, 0)>-\log|u_j|$;  the functions $f_i(M,r)$ and $F_i(M,r)$ are affine.

\item[\emph{(b)}]
Suppose for some $i \in \{\serie{}d-1\}$, the function 
$F_i^{(j)}(M, r)$ (resp. $F_i(M,r)$) is affine and $f_i^{(j)}(M,r) > f_{i+1}^{(j)}(M,r)$ (resp. $f_i(M,r) > f_{i+1}(M,r)$) for $r \in [0, -\log \alpha)$.
Then $M$ admits a unique direct sum 
decomposition $M_0 \oplus M_1$ over $K\{\{\alpha / t, t\rrbracket_0$ such that
\begin{itemize}
\item[(i)]for any $\eta \in (0, -\log \alpha)$,  the multiset of subsidiary $\partial_j$-radii (resp. intrinsic radii) of  $M_0 \otimes F_\eta$ exactly consists of the $i$ smallest elements of the multiset of subsidiary $\partial_j$-radii (resp. intrinsic radii) of $M \otimes F_\eta$, and 
\item[(ii)] the multiset of subsidiary $\partial_j$-radii (resp. intrinsic radii) of $M_0 \otimes E$ exactly consists of the $i$ smallest elements of the multiset of subsidiary $\partial_j$-radii (resp. intrinsic radii) of $M \otimes E$.
\end{itemize}
\end{itemize}
\end{proposition}
\begin{proof}
The statement (a) for $\partial_j$-radii follows from the exact same argument as \cite[Theorem~2.2.6(a)]{kedlaya-xiao}, which follows immediately from the corresponding properties of the associated twisted polynomial.  We now explain how we deduce (a) for intrinsic radii.  Firstly, by Theorem~\ref{T:variation}(a)(b)(d), $d! \cdot F_i(M,r)$ is convex and piecewise affine of integer slopes for $r\in (0, -\log \alpha)$.  We need only to check continuity at $r=0$, which follows from exactly the same argument as in Step 1 of the proof of \cite[Theorem~2.3.9]{kedlaya-xiao}.

The statement (b) is proved in \cite[Theorems~2.3.9, 2.5.5 and Remarks 2.3.11, 2.5.7]{kedlaya-xiao}.
\end{proof}

Note that the statement (b) of the above proposition \emph{excludes} the case when $f_i^{(j)}(M,r) > f_{i+1}^{(j)}(M,r)$ for $r \in (0, -\log \alpha)$ and $f_i^{(j)}(M,0) = f_{i+1}^{(j)}(M,0)$, and the similar case with the superscript $(j)$ removed.  The rest of this subsection is devoted to extending the conclusion of (b) to this case.

\begin{notation}
Set $\calR = \cap_{\alpha \in (0, 1)} K \{\{ \alpha / t, t \}\}$ and $\calR^\bd = \cap_{\alpha \in (0, 1)} K \{\{ \alpha / t, t \rrbracket_0$, where the latter can be identified with the subring of the former consisting of elements with finite 1-Gauss norm.
\end{notation}

\begin{hypothesis}
\label{H:u_j=1}
We assume that $|u_j| = 1$ for $j \in J$.
\end{hypothesis} 

This hypothesis is just to make our presentation simpler.  We can always reduce to this case by replacing $K$ by the completion of $K(x_1, \dots, x_m)$ with respect to the $(|u_1|, \dots, |u_m|)$-Gauss norm and by replacing $u_j$ by $u_j / x_j$, where $\partial_j(x_{j'}) = 0$ for $j, j' \in J$.  Note that $K$ is still discretely valued.

\begin{lemma}
\label{L:Rbdd}
The ring $\calR^\bd$ is a field.  A sequence $(f_n)_{n \in \NN} \subset K \{\{ \alpha /t, t \rrbracket_0$ is convergent if it is convergent for the $r$-Gauss norm for all $r \in (\alpha, 1)$ and is bounded for the $1$-Gauss norm.
\end{lemma}
\begin{proof}
The first statement is well-known; see \cite[Lemma~3.5.2]{kedlaya-overview}. We remark that this would be false if $K$ were not discretely valued.  To see the second statement, we observe that $(f_n)_{n \in \NN}$ converges in $K\{\{\alpha / t, t\}\}$.  The limit has bounded coefficients and hence lies in $K\{\{\alpha / t, t\rrbracket_0$.
\end{proof}

\begin{lemma}
\label{L:twisted-poly-decomposition} 
Fix $j \in J^+$.
Let $\calR^\bd\{T\}$ be the ring of twisted polynomials as in Definition~\ref{D:differential-module}, where $T$ stands for $\partial_j$ if $j \in J$ and for $\frac d{dt}$ if $j=0$.  Let $P = T^d + a_iT^{d-1} + \cdots + a_d \in \calR^\bd\{T\}$ be a monic twisted polynomial whose Newton polygon has \emph{pure} slope $s < 1$.  Let $\{b_1, \dots, b_r\}$ be the set of \emph{$\mathsf v_s$-valuations} of the reduced roots of $P$ (not counting multiplicity, with either increasing or decreasing order), when we view $P$ as a twisted polynomial in $E\{T\}$.
Then $P$ admits a unique factorization $P = Q_1 \cdots Q_r$ as products of monic twisted polynomials such that all the reduced roots of $Q_i$, when viewed as twisted polynomials in $E\{T\}$, have $\mathsf v_s$-valuations $b_i$.
\end{lemma}
\begin{proof}
We assume that $b_1, \dots, b_r$ are in decreasing order.
It then suffices to show that we can write $P = QR$ as a product of two monic polynomials such that the reduced roots of $Q$ (resp. $R$), when viewed as twisted polynomials in $E\{T\}$, have pure $\mathsf v_s$-valuations $b_1$ (resp. strictly less than $b_1$).  We can also write it as $P = RQ$ satisfying the same condition, but with different $Q$ and $R$.
By Lemma~\ref{L:Rbdd}, the claim follows from \cite[Proposition~3.2.2]{kedlaya-semistable3} because the sequences $\{P_l\}$ and $\{Q_l\}$ there is bounded under the 1-Gauss norm.
\end{proof}

\begin{lemma}
\label{L:NP-vs-IR}
Fix $j \in J$.  Let $M$ be a $\partial_j$-differential module of rank $d$ over $K\{\{ \alpha / t, t\rrbracket_0$ such that $M \otimes E$ has pure intrinsic $\partial_j$-radii $IR_{\partial_j}(M \otimes E) <\omega$.  By choosing a cyclic vector of $M \otimes \calR^\bd$, we may identify $M \otimes \calR^\bd$ with $\calR^\bd\{T\} / \calR^\bd\{T\}P$, where $P$ is a twisted polynomial in $\calR^\bd\{T\}$.  Then for $\eta$ sufficiently close to $1^-$, the slopes of Newton polygon of $P$ (for the $\eta$-Gauss norm) are the log of the subsidiary $\partial_j$-radii of $M \otimes F_\eta$ minus $\log \omega$.
\end{lemma}
\begin{proof}
The identification $M \otimes \calR^\bd \simeq \calR^\bd\{T\} / \calR^\bd\{T\}P$ descends to
\[
M \otimes K\{\{\beta / t, t\rrbracket_0 \simeq K\{\{\beta / t, t\rrbracket_0\{T\} / K\{\{\beta / t, t\rrbracket_0\{T\}P
\]
for $\beta$ sufficiently close to $1^-$.
Note that for $\eta$ sufficiently close to $1^-$, all $\partial_j$-radii of $M \otimes F_\eta$ are visible.  The lemma follows from Proposition~\ref{P:spec-norm-from-NP}.
\end{proof}

The following theorem also holds without assume Hypothesis~\ref{H:u_j=1}.

\begin{theorem}
\label{T:touch-decomposition}
Fix $j \in J^+$.  Let $M$ be a $\partial_j$- (resp. $\partial_{J^+}$-) differential module of rank $d$ over $K \{\{ \alpha / t, t\rrbracket_0$ such that $M \otimes E$ has pure intrinsic $\partial_j$-radii $IR_{\partial_j}(M \otimes E) < 1$ (resp. intrinsic radii $IR(M \otimes E)<1$).  Suppose that  for some $i \in \{\serie{}d-1\}$, the function 
$F_i^{(j)}(M, r)$ (resp. $F_i(M, r)$) is affine and $f_i^{(j)}(M,r) > f_{i+1}^{(j)}(M,r)$ (resp. $f_i(M,r) > f_{i+1}(M,r)$) for any $r \in (0, -\log \alpha)$.
Then $M$ admits a unique direct sum 
decomposition $M_0 \oplus M_1$ of $\partial_j$- (resp. $\partial_{J^+}$-) differential module over $K\{\{\alpha / t, t\rrbracket_0$ such that,  for any $\eta \in (0, -\log \alpha)$, the multiset of $\partial_j$-radii (resp. intrinsic radii) of $M_0 \otimes F_\eta$ exactly consists of the smallest $i$ elements of the multiset of $\partial_j$-radii (resp. intrinsic radii) of $M \otimes F_\eta$.
\end{theorem}
\begin{proof}
We first deduce the $\partial_j$-differential module case.  By Theorem~\ref{T:variation-j}(e), it suffices to obtain the decomposition over $K\{\{\beta / t, t \rrbracket_0$ for $\beta \in (\alpha , 1)$ sufficiently close to $1$ and then we may apply Lemma~\ref{L:proj-intersect} and Remark~\ref{R:proj-intersect} to glue this decomposition with the decomposition given by Theorem~\ref{T:variation-j}(e).

To start, we assume that $IR_{\partial_j}(M \otimes E) <\omega $.  By making $\beta$ closer to $1$, we may assume that $IR_{\partial_j}(M \otimes F_\eta)  <\omega$ for all $\eta \in (\beta, 1)$ too.  It is also very easy to reduce to the case when Hypothesis~\ref{H:u_j=1} holds.  Since $\calR^\bd$ is a field, we can find a cyclic vector to identify $M \otimes \calR^\bd$ with $\calR^\bd\{T\} / \calR^\bd\{T\}P$ for a monic twisted polynomial $P$ as in Lemma~\ref{L:twisted-poly-decomposition}.  Applying Lemma~\ref{L:twisted-poly-decomposition} to $M\otimes \calR^\bd$ with the $b$'s in  decreasing order, we  can find a submodule $M_0$ of $M$ such that the multiset of $\partial_j$-radii of $M_0 \otimes F_\eta$ exactly consists of the smallest $i$ elements in the multiset of $\partial_j$-radii of $M \otimes F_\eta$ when  $\eta$ sufficiently close to $1^-$.  Applying Lemma~\ref{L:twisted-poly-decomposition} again with the $b$'s increasing, we can find a quotient $M'_0$ of $M$ satisfying exactly the same condition on $M_0$ as above.  Then the kernel of $M \to M'_0$ together with $M_0$ gives the direct sum decomposition required in the theorem.

We next assume that $p>0$ and $IR_{\partial_j}(M \otimes E) = p^{-1/(p-1)}$.    
If $j \in J$, the $\partial_j$-Frobenius $\varphi^{(\partial_j)}: K^{(\partial_j)}\rar K$ naturally extends to $\varphi^{(\partial_j)}: K^{(\partial_j)} \{\{\alpha / t, t \rrbracket_0 \rar K \{\{\alpha / t, t \rrbracket_0$; 
if $j = 0$, we have $\varphi^{(\partial_0)}: K \{\{\alpha^p / t^p, t^p \rrbracket_0 \rar K \{\{ \alpha / t, t \rrbracket_0$.  Then the desired decomposition follows from the decomposition of $\varphi^{(\partial_j)}_*M$.  Note that $\varphi^{(\partial_j)*} \varphi^{(\partial_j)}_*M \cong M^{\oplus p}$.

If $p>0$ and $IR_{\partial_j}(M \otimes E) >p^{-1/(p-1)}$, we may assume that $IR_{\partial_j}(M \otimes F_\eta)> p^{-1/(p-1)}$ for all $\eta \in (\beta, 1)$, and the decomposition follows from that of the $\partial_j$-Frobenius antecedent of $M$.

Finally, we show that the $\partial_{J^+}$-differential module case follows from the $\partial_j$-differential module case.  By Theorem~\ref{T:variation}(e), it suffices to find the decomposition over $K\{\{\beta / t, t \rrbracket_0$ for $\beta \in (\alpha , 1)$ sufficiently close to $1$ and then we may apply Lemma~\ref{L:proj-intersect} and Remark~\ref{R:proj-intersect} to glue the decompositions.  By Proposition~\ref{P:decomposition-partial}(a) and Theorem~\ref{T:variation-j}(a), there exists $\beta \in (\alpha, 1)$ such that,  if $IR_{\partial_j}(M \otimes E; i) <1$ for some $j$, then the function $f^{(j)}_i(M, r)$ for this $j$ is affine over $[0, -\log \beta)$. By the decompositions given by Proposition~\ref{P:decomposition-partial}(b) and this theorem for $\partial_j$, the restriction of $M$ to $K \{\{\beta/t t\rrbracket_0$ is the direct sum of $\partial_{J^+}$-differential modules $M_l$ such that, for any $j \in J^+$ with $IR_{\partial_j}(M_l \otimes E) <1$, the $\partial_j$-differential module $M_l\otimes F_\eta$ has pure $\partial_j$-radii for any $\eta \in (\beta, 1)$.    Since we already know that $M \otimes E$ has pure intrinsic radii $< 1$, we may take $\beta$ sufficiently close to $1$ such that each direct summand above has pure intrinsic radii equal to the $\partial_j$-radii for some $j$, when tensored with $F_\eta$ for any $\eta \in (\beta, 1)$.  Hence regrouping the direct summand gives the direct sum decomposition we are looking for.
\end{proof}

\begin{remark}
The condition $IR_{\partial_j}(M \otimes E) < 1$ is crucial.  As pointed out in \cite[Remark~12.5.4]{kedlaya-course}, one may give counterexamples in the case $IR_{\partial_j}(M \otimes E) = 1$ using the theory of crystals.  However, in the presence of a Frobenius, one may still get the decomposition.  We plan to come back to this point in a future work.
\end{remark}

\begin{proposition}
\label{P:touch-nonlog}
Let $M$ be a $\partial_{J^+}$-differential module over $K \{\{\alpha/t, t \rrbracket_0$ (resp. $K \llbracket t\rrbracket_0$) or rank $d$.  We put $\hat f_i(M, 0) = -\log ER(M \otimes E; i)$ and $\hat F_i(M, 0) = \hat f_1(M, 0) + \cdots + \hat f_i(M,0)$ for $i =1, \dots, d$.  Then we have the following.
\begin{itemize}
\item[\emph{(a)}] The functions  $\hat f_i(M, r)$ and $\hat F_i(M, r)$ are affine at $r = 0$.

\item[\emph{(b)}]
Suppose for some $i \in \{\serie{}d-1\}$, the function 
 $\hat F_i(M,r)$ is affine and  $\hat f_i(M,r) > \hat f_{i+1}(M,r)$ for $r \in (0, -\log \alpha)$ (resp. whenever $\hat f_i(M, r)>r$), and suppose that $\hat f_i(M, 0)>0$.
Then $M$ admits a unique direct sum 
decomposition $M_0 \oplus M_1$ over $K\{\{\alpha / t, t\rrbracket_0$ (resp. $K \llbracket t \rrbracket_0$) such that the multiset of extrinsic  radii of $M \otimes F_\eta$ for any $\eta \in (0, -\log \alpha)$ (resp. for any $\eta>0$ such that $\hat f_i(M, r)>r$) consists of the smallest $i$ elements of the multiset of extrinsic radii of $M \otimes F_\eta$.
\end{itemize}
\end{proposition}
\begin{proof}
(a) follows from exactly the same argument as in Proposition~\ref{P:decomposition-partial}.  We now prove (b).
By the extrinsic version of Theorem~\ref{T:variation}(e), it suffices to find the decomposition over $K\{\{\beta / t, t \rrbracket_0$ for $\beta \in (\alpha , 1)$ sufficiently close to $1$ and then we may apply Lemma~\ref{L:proj-intersect} and Remark~\ref{R:proj-intersect} to glue the decompositions.  By Proposition~\ref{P:decomposition-partial}(b) and Theorem~\ref{T:touch-decomposition} for $\partial_j$-differential modules, there exists  $\beta \in (\alpha, 1)$ such that when we tensor $M$ with $K\{\{\beta / t, t \rrbracket_0$, it is a direct sum of differential modules $M_l$ such that either for any $j \in J^+$ with $R_{\partial_j}(M_l \otimes E) <1$, $M_l\otimes F_\eta$ has pure $\partial_j$-radii for all $\eta \in (\beta, 1)$, or
we have $ER(M_l \otimes E)=1$.
The proposition then follows from regrouping these direct summands.
\end{proof}

\subsection{Refined radii and the log-slopes of the radii}
\label{S:boundary}

For a differential module over an annulus or a disc, the slopes of the functions coming from the radii of convergence are determined by the multiset of refined radii for the differential module completed for the corresponding Gauss norm.  We also give a refined radii decomposition result for differential modules over bounded analytic rings.

\begin{theorem} \label{T:refined-radii-vs-radii}
Fix $j \in J^+$ and let $M$ be a $\partial_j$-differential module over $K \{\{\alpha / t, t  \rrbracket_0$ of rank $d$.  Assume that  $f_i^{(j)}(M, r)$ for all $i$ are the same and are affine of slope $b$ in $r \in [0, -\log \alpha)$. Moreover, we assume that $R_{\partial_j}(M \otimes E) = \omega \ee^s$ is strictly less than $|u_j|^{-1}$ if $j \in J$ and is strictly less than $1$ if $j =0$.  Then the $\mathsf v_s$-valuation of any element in the multiset of refined $\partial_j$-radii of $M \otimes E$ is $-b$.
\end{theorem}
\begin{proof}
We may assume that $|u_j|=1$.
We first consider the case when $M \otimes E$ has pure visible intrinsic $\partial_j$-radii $IR_{\partial_j}(M \otimes E) < \omega$.  By making $\alpha$ closer to $1^-$, we may assume that the function $f_i^{(j)}(M, r) > -\log \omega$ for each $i$ is affine over $[0, -\log \alpha)$.

As in Theorem~\ref{T:touch-decomposition}, we may  identify $M \otimes \calR_K^\bd$ with $\calR^\bd\{T\} / \calR^\bd\{T\}P$ for some twisted polynomial $P = T^d + a_1 T^{d-1} + \cdots + a_d \in \calR^\bd\{T\}$.  Since $M \otimes E$ has pure $\partial_j$-radii $\omega \ee^s$, the Newton polygon of $P$ with respect to the 1-Gauss norm has pure slope $s$ and the multiset $\Theta_{\partial_j}(M \otimes E)$ is just the multiset of reduced roots of this twisted polynomial.  We put  $\overline P = T^d + \bar a_1^{(s)} T^{d-1} + \cdots + \bar a_d^{(ds)}$, where $\bar a_i^{(is)} \in \kappa_{K}^{(is)}((t))$.

When $\eta$ is sufficiently close to $1^-$, the Newton polygon of $P$ with respect to the $\eta$-Gauss norm is determined by the Newton polygon of $\overline P$ in the following sense:  it is the lower convex hull of the set $\{(-i, -\log |a_i|_1 - \mathsf v(\bar a_i^{(is)})\log \eta)\}$.  By Lemma~\ref{L:NP-vs-IR}, this implies that the collection of all slopes of functions $f_i^{(j)}(M, r)$ for all $i$ at $r = 0$ is exactly the collection of the $\mathsf v_s$-valuations  of the roots of  $\overline P$, which in turn equals the collection of  the $\mathsf v_s$-valuations of the elements of the multiset of refined $\partial_j$-radii of $M \otimes E$.

Now, it suffices to reduce to the case above using $\partial_j$-Frobenius.  Assume $p>0$ from now on.  It is easier to work with intrinsic radii and refined intrinsic radii.  So we put $g_i(M, r) = f_i^{(j)}(M, r) + \log|u_j|$ if $j \in J$ and $g_i(M, r) = f_i^{(j)}(M, r) - r$ if $j = 0$.  Moreover, we set $s' = -\log (\omega IR_{\partial_j}(V)^{-1})$.

If $IR_{\partial_j}(M \otimes E) = \omega = p^{-1/(p-1)}$, we set $M_1 = \varphi^{(\partial_j)}_*M$.  Then  Lemma~\ref{L:frob-properties}(d) implies that
\begin{eqnarray*}
&\big\{ g'_i(M_1, 0) \big\} = \left\{
\begin{array}{ll}
\{pg'_i(M, 0)\ (d \textrm{ times}), 0 \ ((p-1)d \textrm{ times})\} & g'_i(M, 0)<0\\
\{ g'_i(M, 0) \ (pd \textrm{ times})\} & g'_i(M, 0)\geq 0
\end{array} \right., \textrm{ if } j \in J;  \\
&\big\{g'_i(M_1, 0) \big\} = \left\{
\begin{array}{ll}
\{g'_i(M, 0), 0 \ (p-1 \textrm{ times})\} & g'_i(M, 0)<0\\
\{ \frac 1p g'_i(M, 0) \ (p \textrm{ times})\} & g'_i(M, 0)\geq0
\end{array}\right., \textrm{ if } j =0.
\end{eqnarray*}
By Proposition~\ref{P:refined-frob}, the elements in the multiset $\calI\Theta_{\partial'_j}(M_1 \otimes E^{(\partial_j)})$ can be grouped into  $p$-tuples $(\frac \theta p, \frac{\theta+1}p, \dots, \frac{\theta+p-1}p)$, and the multiset $\calI\Theta_\partial(M \otimes E)$ is composed of $(\theta^p-\theta)^{1/p}$ for each $p$-tuple above with the same multiplicity, where $\theta \in \kappa_{E^\alg}$.  
Elementary calculation shows the following relation between the $\mathsf v_0$-valuations of $(\theta^p-\theta)^{1/p}$ and the $\mathsf v_{-\log p}$-valuation of $\theta$:
\begin{itemize}
\item when $\mathsf v_0(\theta)< 0$, we have $\mathsf v_{-\log p}(\frac {\theta+ l}p) = \mathsf v_0(\theta)$ for $l = 0, \dots, p-1$, and $\mathsf v_0((\theta^p-\theta)^{1/p}) = \mathsf v_0(\theta)$; 
\item when $\mathsf v_0(\theta)\geq 0$, we have $\mathsf v_{-\log p}(\frac {\theta+ l}p) = 0$ for $l = 1, \dots, p-1$, and $\mathsf v_0((\theta^p-\theta)^{1/p}) = \frac 1p\mathsf v_0(\theta)$.
\end{itemize}
Hence the statement for $M_1$ with $\mathsf v_{-\log p}$ implies that for $M$ with $\mathsf v_0$.

If $IR_{\partial_j}(M \otimes E) > \omega$, by Lemma~\ref{L:frob-properties}(d) and Remark~\ref{R:frob-over-annulus}, $M$ has a $\partial_j$-Frobenius antecedent $M_0$ if $\alpha$ is sufficiently close to $1^-$.  By Lemma~\ref{L:frob-properties}(d) and Proposition~\ref{P:refined-frob}, we have
\begin{align*}
g_i(M_0, r) = pg_i(M, r) &\textrm{ for any } i, \textrm{ and } \calI\Theta_{\partial'_j}(M_0 \otimes E^{(\partial_j)}) = \big\{(-\theta)^p / p\big| \theta \in \calI \Theta_{\partial_j}(M \otimes E)\big\}, \textrm{ if } j \in J; \\
g_i(M_0, pr) = pg_i(M, r) &\textrm{ for any } i, \textrm{ and } \calI\Theta_{\partial'_j}(M_0 \otimes E^{(\partial_j)}) = \big\{(-\theta)^p / p\big| \theta \in \calI \Theta_{\partial_j}(M \otimes E)\big\}, \textrm{ if } j =0.
\end{align*}
Since $\mathsf v_{(ps' -\log p)}((-\theta)^p/p) = p\mathsf v_{s'}(\theta)$, the statement for $M$ with $\mathsf v_{s'}(-\log p)$ follows from the statement for $M_0$ with $\mathsf v_{ps'-\log p}$ if $j \in J$ and with $\frac 1p \mathsf v_{ps'-\log p}$ if $j=0$ (note that $t^p$ is the coordinate in the latter case).
\end{proof}

\begin{corollary} 
\label{C:refined-radii-vs-radii-j}
Fix $j \in J^+$ and let $M$ be a $\partial_j$-differential module over $K \{\{\alpha / t, t  \rrbracket_0$.  Assume that $M \otimes E$ has pure $\partial_j$-radii $R_{\partial_j}(M \otimes E) = \omega \ee^s$, which is strictly less than $|u_j|^{-1}$ if $j \in J$ and is strictly less than $1$ if $j =0$.  Then the following two multisets are the same:
\begin{itemize}
\item[\emph{(i)}] the multiset composed of the $\mathsf v_s$-valuations of the elements in the multiset of refined $\partial_j$-radii of $M \otimes E$, i.e., $\big\{ \mathsf v_s(\theta) \big| \theta  \in\Theta_{\partial_j}(M \otimes E) \big\}$, and
\item[\emph{(ii)}] the multiset composed of the negatives of the slopes of $f_i^{(j)}(M, r)$ at $r=0$, for $i = 1, \dots, d$.
\end{itemize}
\end{corollary}
\begin{proof}
This follows from combining Theorems~\ref{T:touch-decomposition} and \ref{T:refined-radii-vs-radii}.
\end{proof}

\begin{notation}\label{N:mathsf-v} For any $\goths \in \RR$, the valuation $\mathsf v_\goths$ on $\kappa_E^{(\goths)}$ induces a valuation on $\kappa_E^{(\goths)} \frac{dt}t\oplus\bigoplus_{j \in J}\kappa_E^{(\goths)} \frac{du_j}{u_j}$, still denoted by $\mathsf v_\goths$, by setting
\[
\mathsf v_\goths\big(\theta_0 \frac{dt}t + \theta_1 \frac{du_1}{u_1}+ \cdots+ \theta_m \frac{du_m}{u_m} \big) = \min_{j \in J^+} \big\{\mathsf v_\goths(\theta_j) \big\}, \textrm{ for }\theta_0, \dots, \theta_m \in \kappa_E^{(\goths)}.
\]
\end{notation}

\begin{corollary} 
\label{C:refined-radii-vs-radii}
Let $M$ be a $\partial_{J^+}$-differential module over $K \{\{\alpha / t, t  \rrbracket_0$.  Assume that $M \otimes E$ has pure intrinsic radii $IR(M \otimes E) = \omega \ee^\goths <1$.  Then the following two multisets are the same:
\begin{itemize}
\item[\emph{(i)}] the valuations of the refined intrinsic radii of $M \otimes E$, i.e., $\big\{ \mathsf v_\goths(\theta) \big| \theta  \in\calI\Theta(M \otimes E) \big\}$, and
\item[\emph{(ii)}] the  negatives of the slopes of $f_i(M, r)$ at $r=0$, for $i = 1, \dots, d$.
\end{itemize}

\end{corollary}
\begin{proof}
This follows from combining Theorems~\ref{T:touch-decomposition} and \ref{T:refined-radii-vs-radii}.
\end{proof}

Similar to Theorem~\ref{T:refined-decomposition}, we have the following decomposition by refined radii.

\begin{theorem}\label{T:refined-boudary-j}
Fix $j \in J^+$ and let $M$ be a $\partial_j$-differential module of rank $d$ over $K \{\{\alpha / t, t  \rrbracket_0$.  Assume that $M \otimes F_\eta$, for $\eta \in (\alpha, 1)$, and $M \otimes E$ all have pure $\partial_j$-radii, and assume that the function $f_1^{(j)}(M, r)$ is affine with slope $b$ for $r \in [0, -\log \alpha)$.  Let $e$ be the prime-to-$p$ part of the denominator of $b$.  Moreover,  assume that $R_{\partial_j}(M \otimes E) = \omega \ee^s$ is strictly less than $|u_j|^{-1}$ if $j \in J$ and is strictly less than $1$ if $j =0$.  Then there exists a finite tamely ramified extension $K'$ of $K$ and a unique direct sum  decomposition
\[
M \otimes K'\{\{\alpha^{1/e} / t^{1/e} , t^{1/e} \rrbracket_0 = \bigoplus_{\theta \in \kappa_{K^\alg}^{(s)}} M_\theta
\]
of $\partial_j$-differential modules such that
\begin{itemize}
\item[\emph{(i)}] $M_\theta \otimes F_\eta$ has pure refined $\partial_j$-radii $\theta t^{-b}$ for all $\eta \in(\alpha, 1)$, and
\item[\emph{(ii)}] every element in the multiset of refined $\partial_j$-radii of $M_\theta \otimes E$ is congruent to $\theta t^{-b}$ modulo elements in $\kappa_{K^\alg}^{(s)}$ with $\mathsf v_s$-valuation strictly bigger than $\mathsf v_s(\theta t^{-b}) = -b$.
\end{itemize}
Moreover, this decomposition descents to a unique decomposition of $M$ itself by Galois descent, satisfying analogous properties, but in the fashion stated in terms of $\mu_e \rtimes \Gal(K^\alg/K)$-orbits.
\end{theorem}
\begin{proof}
The proof is identical to that of Theorem~\ref{T:var-refined}, except that we use decomposition Theorem~\ref{T:touch-decomposition} in place of Theorem~\ref{T:variation-j}.
\end{proof}

\begin{theorem}\label{T:refined-boudary}
Let $M$ be a $\partial_{J^+}$-differential module of rank $d$ over $K \{\{\alpha / t, t  \rrbracket_0$.  Assume that $M \otimes F_\eta$, for $\eta \in (\alpha, 1)$, and $M \otimes E$ all have pure intrinsic radii, and assume that the function $f_1(M, r)$ is affine with slope $b$ for $r \in [0, -\log \alpha)$.  Let $e$ be the prime-to-$p$ part of the denominator of $b$.  Moreover, assume that $IR(M \otimes E) = \omega \mathrm e^\goths<1$.  Then there exists a finite tamely ramified extension $K'$ of $K$ and a unique direct sum decomposition
\[
M \otimes K'\{\{\alpha^{1/e} / t^{1/e} , t^{1/e} \rrbracket_0 = \bigoplus_{\vartheta \in \oplus_{j \in J} \kappa_{K^\alg}^{(\goths)}\frac {du_j}{u_j} \oplus \kappa_{K^\alg}^{(\goths)}\frac {dt}{t}} M_\vartheta
\]
of $\partial_{J^+}$-differential modules such that
\begin{itemize}
\item[\emph{(i)}] $M_\vartheta \otimes F_\eta$ has pure refined intrinsic radii $\vartheta t^{-b}$ for all $\eta \in (\alpha, 1)$, and
\item[\emph{(ii)}] every element in the multiset of refined intrinsic radii of $M_\vartheta \otimes E$ is congruent to $\vartheta t^{-b}$ modulo those elements in $\oplus_{j \in J} \kappa_{K^\alg}^{(\goths)}\frac {du_j}{u_j} \oplus \kappa_{K^\alg}^{(\goths)}\frac {dt}{t}$ with $\mathsf v_\goths$-valuation strictly bigger than $\mathsf v_\goths(\vartheta t^{-b}) = -b$.
\end{itemize}
Moreover, this decomposition descents to a unique decomposition of $M$ itself by Galois descent, satisfying analogous properties, but in the fashion stated in terms of $\mu_e \rtimes \Gal(K^\alg/K)$-orbits.
\end{theorem}
\begin{proof}
The proof is identical to that of Theorem~\ref{T:var-refined-multi}, except that we use invoke Theorem~\ref{T:refined-boudary-j} in place of Theorem~\ref{T:variation-j}.
\end{proof}

\begin{corollary}
\label{C:refined-boundary}
Let $M$ be a $\partial_{J^+}$-differential module of rank $d$ over $K \{\{\alpha / t, t  \rrbracket_0$.  Assume that $M \otimes E$ has pure intrinsic radii $IR(M \otimes E) = \omega \ee^\goths <1$ and that the function $f_i(M, r)$ for each $i=1, \dots, d$ is affine over $[0, -\log \alpha)$.  Let $M = \oplus_{b \in \QQ} M_b$ be the unique direct sum decomposition of $M$ over $A^1_K(\alpha, 1)$ such that $f_1(M_b, r) = \cdots = f_{\dim M_b}(M_b, r)$ has slope $b$. Then the following two multisets are the same
\begin{itemize}
\item[\emph{(i)}] The multiset composed of all elements in $\calI\Theta(M_b \otimes F_\eta) \subset  \oplus_{j \in J^+}t^{-b} \kappa_{K^\alg}^{(\goths)} \frac{du_j}{u_j} \oplus t^{-b} \kappa_{K^\alg}^{(\goths)} \frac{dt}t$ for all $b$ and for some fixed $\eta \in (\alpha, 0)$ (this is independent of the choice of $\eta$);
\item[\emph{(ii)}] The multiset composed of $\bar \vartheta$ for all $\vartheta \in \Theta_{\partial_j(V)}$, where $\bar \vartheta$ is the reduction of $\vartheta \in \oplus_{j \in J^+}t^{-b} \kappa_{K^\alg}^{(\goths)} \frac{du_j}{u_j} \oplus t^{-b} \kappa_{K^\alg}^{(\goths)} \frac{dt}t$ modulo those elements with $\mathsf v_s$-valuation strictly bigger than $\mathsf v_s(\vartheta)$.
\end{itemize}
\end{corollary}
\begin{proof}
It follows from the decomposition Theorems~\ref{T:touch-decomposition} and \ref{T:refined-boudary}.
\end{proof}

We have similar results for  extrinsic radii.

\begin{theorem}
\label{T:touch-refined-nonlog}
Assume that $|u_j|=1$ for all $j \in J$.
For $s \in \RR$, let $\hat {\mathsf v}_s$ be the valuation on $\kappa_E^{(s)} dt \oplus \bigoplus_{j \in J} \kappa_E^{(s)} du_j$ given by
\[
\hat {\mathsf v}_s\big( \theta_0 dt+ \theta_1 du_1 + \cdots + \theta_m du_m \big)= \min_{j \in J^+}\big\{ \mathsf v_s(\theta_j)\big\}.
\]
Let $M$ be a $\partial_{J^+}$-differential module of rank $d$ over $K \{\{\alpha / t, t  \rrbracket_0$.  Assume that $M \otimes F_\eta$, for $\eta \in (\alpha, 1)$, and $M \otimes E$ all have pure extrinsic radii, and assume that the function $\hat f_1(M, r)$ is affine with slope $b$ for $r \in [0, -\log \alpha)$.  Let $e$ be the prime-to-$p$ part of the denominator of $b$.  Moreover,  assume that $ER(M \otimes E) = \omega \mathrm e^s<1$.  Then there exists a unique direct sum decomposition
$
M = \bigoplus_{\{\mu_e\hat \vartheta\}} M_{\{\mu_e\hat \vartheta\}}
$
of $\partial_{J^+}$-differential modules over $K \{\{\alpha / t, t  \rrbracket_0$, where the direct sum runs through all $\mu_e \rtimes \Gal(K^\sep/K)$-orbits $\{\mu_e \hat \vartheta\}$ in $ \bigoplus_{j \in J} \kappa_{K^\alg}^{(s)} du_j \oplus \kappa_{K^\alg}^{(s)} dt$ such that
\begin{itemize}
\item[\emph{(i)}] for all $\eta \in (\alpha, 1)$, the multiset of refined extrinsic radii of  $M_{\{\mu_e\vartheta\}} \otimes F_\eta$ is composed of the $\mu_e \rtimes\Gal(K^\alg/K)$-orbit $\{\mu_e \hat \vartheta t^{-b}\}$ with the appropriate multiplicity, and
\item[\emph{(ii)}] the multiset consisting of the reductions of elements in the multiset of refined extrinsic radii of $M_{\{\mu_e\hat \vartheta\}} \otimes E$ modulo those elements with $\hat{\mathsf v}_s$-valuation is strictly  bigger than $-b$, is composed of the the $\mu_e \rtimes\Gal(K^\alg/K)$-orbit $\{\mu_e \hat \vartheta t^{-b}\}$ with the appropriate multiplicity.
\end{itemize}
\end{theorem}
\begin{proof}
The proof is identical to that of Theorem~\ref{T:nonlog-refined-decomposition}, except that we use invoke Theorem~\ref{T:refined-boudary-j} in place of Theorem~\ref{T:variation-j}.
\end{proof}

\begin{corollary}
Assume that $|u_j| =1 $ for all $j \in J$.  Let $M$ be a $\partial_{J^+}$-differential module of rank $d$ over $K \llbracket t\rrbracket_0$.  Assume that $ER(M \otimes E) = \omega \ee^s<1$.  Let $M_e$ denote the unique $\partial_{J^+}$-differential submodule of $M \otimes E$ that has pure extrinsic radii $ER(M \otimes E)$; put $l=\dim M_e$.  Then
\begin{itemize}
\item[\emph{(a)}] The $\hat {\mathsf v}_s$-valuations of elements in $\Theta(M_e \otimes E)$ are all nonnegative.
\item[\emph{(b)}] There exists a unique direct sum decomposition $M = \bigoplus_{\{\hat\vartheta\}} M_{\{\hat\vartheta\}} \oplus M_0$ of $\partial_{J^+}$-differential modules over $K\llbracket t\rrbracket_0$, where the first direct sum is taken over all $ \Gal(\kappa_K^\sep / \kappa_K)$-orbits $\{\hat\vartheta\}  \subset \bigoplus_{j \in J} \kappa_{K^\alg}^{(s)} du_j \oplus \kappa_{K^\alg}^{(s)} dt$ such that
\begin{itemize}
\item[\emph{(i)}] for all $\eta<1$, $M_{\{\hat \vartheta\}} \otimes F_\eta$ has pure extrinsic radii $\min\{\omega \ee^s, \eta\}$ and, when $\eta \in (\omega \ee^s, 1)$, the multiset $\Theta(M_{\{\vartheta\}} \otimes F_\eta)$ is composed of  $\{\hat \vartheta\}$ with multiplicity,
\item[\emph{(ii)}] the multiset consisting of reductions of elements in the multiset of refined extrinsic radii of $M_{\{\vartheta\}} \otimes E$ modulo those elements with positive $\hat{\mathsf v}_s$-valuation, is composed of  $\{ \hat \vartheta\}$ with appropriate multiplicity, and
\item[\emph{(iii)}] For any $r>0$ satisfying $\hat f_1(M_0, r)<r$, we have $\hat f_1(M_0, r) < \omega \ee^s$.
\end{itemize}
\end{itemize}
\end{corollary}
\begin{proof}
(a) By Proposition~\ref{P:touch-nonlog}(a) together with  Theorem~\ref{T:variation}(c'), we know that the functions $f_1(M, r), \dots, f_l(M,r)$ are linear in a neighborhood of $r$ with nonpositive slopes.  Then applying the decomposition in Proposition~\ref{P:touch-nonlog}(b) and Theorem~\ref{T:touch-refined-nonlog} together with  description (ii) in Theorem~\ref{T:touch-refined-nonlog}, we conclude that the $\hat {\mathsf v}_s$-valuations of elements in $\Theta(M_e)$ are all nonnegative.

(b)  
Let $l'$ denote the number of elements in $\Theta(M_e)$ whose $\hat {\mathsf v}_s$-valuation is zero.
By the proof of (a), we see that the derivatives $\hat f'_1(M, 0) = \dots= \hat f'_{l'}(M,0)$ are equal to $0$, and that $f'_{l'+1}(M,0)>0$ or $\hat f_{l'+1}(M,0)> \hat f_l(M,0)$ in case $l=l'$.  By Theorem~\ref{T:variation}(c')(d), we know that 
\[
\hat f_1(M,0)=\hat f_1(M,r) = \dots= \hat f_{l'}(M,r) > \hat f_{l'+1}(M,r)\]
 for any $r< \hat f_1(M,0)$.  We may then apply Proposition~\ref{P:touch-nonlog} to split off the desired $M_0$.  Now, we may apply the standard technique (Lemma~\ref{L:proj-intersect} and Remark~\ref{R:proj-intersect}) to glue the decomposition given by Theorem~\ref{T:touch-refined-nonlog} and Proposition~\ref{P:nonlog-refined-disc}; this gives the further decompositions by $M_{\{\theta\}}$.
 \end{proof}

\subsection{Variation over polyannuli}
\label{S:polyannuli-variation}

In this subsection, we study differential modules over a polyannulus or a polydisc.  In particular, we are interested in the study the functions coming from the radii of convergence when we complete the differential module with respect to various Gauss norms.  We relate the slopes of such functions with the valuations of the refined intrinsic radii.

In this subsection, we assume Hypothesis~\ref{H:K-multi} and we assume that $K$ is discretely valued.

\begin{definition}
A subset $C \subseteq \RR^n$ is called \emph{nondegenerate} if it contains an open subset of $\RR^n$.  Its interior is denoted by $C^\inte$.

An \emph{integral affine functional} on $\RR^n$ is a map $\lambda: \RR^n \to \RR$ of the
form $\lambda(x_1,\dots,x_n) = a_1 x_1 + \cdots + a_n x_n + b$ for some
$a_1,\dots, a_n \in \ZZ$ and $b \in -\log |K^\times |^\QQ$.

A subset $C \subseteq \RR^n$ is \emph{rational polyhedral} (or \emph{RP} for short) if it is \emph{bounded} and there exist 
integral affine functionals $\lambda_1, \dots, \lambda_r$ such that
$C = \{x \in \RR^n| \lambda_i(x) \geq 0 \textrm{ for }i=1,\dots,r\}$.

For $C \subseteq \RR^n$ a RP subset of $\RR^n$,
a function $f: C \to \RR^n$ is \emph{integral polyhedral} 
if there exist finitely many integral
affine functionals 
$\lambda'_1, \dots, \lambda'_d$ such that
$f(x) = \max\{\lambda'_1(x), \dots, \lambda'_d(x)\}$ for any $x \in C.$
\end{definition}

\begin{remark} 
Our convention slightly differs from \cite{kedlaya-xiao}, where RP subsets are not assumed to be bounded.  However, some of the statements below still hold for unbounded RP, and they are often simple corollaries of the statements in the bounded case.  We leave this as an exercise for the reader.
\end{remark}

\begin{notation}
We put $I = \{\serie{}n\}$.  We use $\underline a$ to denote the $n$-tuple $(a, \dots, a)$.
\end{notation}

\begin{definition}
For a subset $C \subseteq \RR^n$, let $\ee^{-C}$ denote the \emph{closure} of the subset $\{\ee^{-r_I}: r_I \in C\} \subseteq (0, +\infty)^n$.
A subset $S$ of $[0, +\infty)^n$ is called \emph{log-RP} if $S = \ee^{-C}$ for some RP subset $C$ of $\RR^n$; it is called \emph{nondegenerate} if $C$ is so.

For $S$ a log-RP subset of $[0, +\infty)^n$, define $A_K(S^\inte)$ to be the 
subspace of the (Berkovich) analytic $n$-space with coordinates $t_1, 
\dots, t_n$ satisfying the condition $(|t_1|, \dots, |t_n|) \in S^\inte$.  We use $K\{\{ S\}\}$ to denote  its ring of functions, and use $K\llbracket S \rrbracket_0$ to denote the subring of $K \{\{ S\}\}$ consisting of functions that are bounded on $|t_I| \in S^\inte$.
\end{definition}

\begin{notation}
Let $S$ be a nondegenerate log-RP subset of $[0, +\infty)^n$ and let $R$ denote either $K\{\{ S \}\}$ or $K \llbracket S \rrbracket_0$.  Let $M$ be a $\partial_{I\cup J}$-differential module over $R$ of rank $d$, with respect to the derivations $\serie{\partial_}m$ and
$\partial_{m+1} = \partial / \partial t_1, \dots, \partial_{m+n} = \partial / \partial t_n$.
For an element $\eta_I$ in $ (\serie{\eta_}n) \in S$ ($S^\inte$ if $R = K\{\{ S\}\}$), let $F_{\eta_I}$ be the completion of $\Frac(R)$ with respect to the $\eta_I$-Gauss norm.  We remark that for $\eta_I$ on the boundary of $S$, $F_{\eta_I}$ ``looks different" (more like $E$ than $F_\eta$ in the 1-dimensional case).

For an element $r_I$ in $ -\log(S)$ ($-\log (S^\inte)$ if $R = K\{\{ S\}\}$), put $f_l(M, r_I) = 
-\log IR(M \otimes F_{\ee^{-r_I}}; l)$ and 
$F_l(M, r_I) = f_1(M, r_I) + \cdots + 
f_l(M, r_I)$ for $l = 1, \dots, d$.
\end{notation}

\begin{theorem}\label{T:polyannuli-variation}
Keep the notation as above. We have the following.
\begin{enumerate}
\item[\emph{(a)}]
(Polyhedrality)
The functions
$d! F_l(M, r_I)$, for $l = 1, \dots, d-1$, and $F_d(M, r_I)$ are integral polyhedral functions.
\item[\emph{(b)}] (Decomposition)  Suppose that for some $l \in \{1, \dots, d\}$, the function $F_l(M, r_I)$ is affine, and suppose that $f_l(M, r_I) > f_{l+1}(M, r_I)$ for any $r_I \in -\log (S)$.  Then $M$ admits a unique direct sum decomposition $M \cong M_0 \oplus M_1$ of differential modules such that for any $\eta_I \in -\log(S^\inte)$, the multiset of intrinsic radii of $M_0$ exactly consists of the smallest $l$ elements in the multiset of  intrinsic radii of $M \otimes F_{\eta_I}$.
\item[\emph{(c)}] (Refined radii) Assume that $R = K\{\{S\}\}$ and that $f_1(M, r_I) = \cdots = f_d(M, r_I) = -\log \omega - \goths+b_1r_1 +\cdots +b_nr_n$ are affine functions on $-\log (S^\inte)$. Let $e_i$ denote the prime-to-$p$ part of the denominator of $b_i$ for all $i \in I$.   Then there exists a finite tamely ramified extension $K'$ of $K$ and a multiset $\calI\Theta(M) \subset \oplus_{i \in I} \kappa_{K'}^{(\goths)} \frac {dt_i}{t_i} \oplus \oplus_{j \in J}\kappa_{K'}^{(\goths)} \frac {du_j}{u_j}$ such that we have a unique direct sum  decomposition of differential modules
\[
M \otimes_R R[t_1^{1/e_1}, \dots, t_n^{1/e_n}] = \bigoplus_{\vartheta \in \calI\Theta(M)} M_\vartheta,
\]
such that each $M_\vartheta \otimes F_{\eta_I}[t_1^{1/e_1}, \dots, t_n^{1/e_n}]$ has pure refined intrinsic radii $t_I^{-b_I}\vartheta$.
\end{enumerate}
\end{theorem}
\begin{proof}
For (a) and (b), see \cite[Theorems~3.3.9 and 3.4.4, and Remark 3.4.7]{kedlaya-xiao}.  (c) follows from the same argument but using Theorem~\ref{T:var-refined-multi} as the decomposition tool.
\end{proof}

To extend (c) of the theorem above to the boundary is a little tricky.  We will prove it in a special case and leave the general case as an exercise for the reader.

\begin{situation}
\label{Sit:polyannuli-boundary}
Consider the subset $C =  \big\{(x_I) \subset \RR^n\big|x_I \geq 0, x_1+ \cdots + x_n \leq 1\big\}$.  Put $S = \ee^{-C}$, and $R = K \llbracket S \rrbracket_0$.  Let $M$ be a differential module over $K\llbracket S\rrbracket_0$.  Assume moreover that $f_1(M, \underline 0) = \cdots f_d(M, \underline 0) = -\log \omega - \goths$ with $\goths<0$.  We define the following two multisets.

\begin{itemize}
\item [(1)] 
Choose $x \in \gothm_K^{(\goths)} \bs \gothm_K^{(\goths)+}$ to identify $\kappa_{F_{\underline 1}}^{(\goths)} \stackrel {\cdot x^{-1}} \lrar \kappa_{F_{\underline 1}}$ and embed the latter into the higher local field $\kappa_K((t_1))\cdots ((t_n))$, which is equipped with a multi-indexed valuation with respect to the parameters $(t_n, \dots, t_1)$.  This gives rise to a valuation $\bbv_\goths: \kappa_{F_{\underline 1}}^{(\goths)} \rar \ZZ^n \subset \QQ^n$, where the latter is equipped with the lexicographical order; this does not depend on the choice of $x$.  Define the following valuation on $\bigoplus_{i \in I}\kappa_{F_{\underline 1}^\alg}^{(\goths)} \frac{dt_i}{t_i} \oplus\bigoplus_{j \in J}\kappa_{F_{\underline 1}^\alg}^{(\goths)} \frac{du_j}{u_j}$, still denoted by $\bbv_\goths$, by taking the minimum of $\bbv_\goths$ over the coefficients. 
We consider the multiset $A = \big\{(\bbv(\vartheta), \bar \vartheta) \big| \vartheta \in \calI\Theta(M \otimes F_{\underline 1})\big\}$, where $\bar \vartheta$ is the reduction of $t_I^{-\bbv_\goths(\vartheta)}\vartheta$ to $\bigoplus_{i \in I}\kappa_{K^\alg}^{(\goths)} \frac{dt_i}{t_i} \oplus\bigoplus_{j \in J}\kappa_{K^\alg}^{(\goths)} \frac{du_j}{u_j}$.

\item [(2)] By Theorem~\ref{T:polyannuli-variation}(a), there exists a RP subset $C'$ of $C$ which is adjacent to the cells  $t_1 = \cdots = t_i = 0$ for $i = 1, \dots, n-1$, such that the function $f_l(M, r_I)$ for each $l$ is affine in $r_I$ over $C'$.   Then, over $\ee^{-C'^\inte}$, we have a unique direct sum decomposition of differential modules $M = \bigoplus_{b_I \in \QQ^n} M_{b_I}$ such that
 \[
f_1(M_{b_I}, r_I) = \cdots = f_{\dim M_{b_I}}(M_{b_I}, r_I) = -\log \omega -\goths +b_1r_1 + \cdots +b_nr_n.
\]
We put
\[
B = \big\{ (-b_1, \dots, -b_n, \vartheta) \big| b_I \in \QQ^n, t_1^{-b_1} \cdots t_n^{-b_n}\vartheta \in \calI\Theta(M \otimes F_{\eta_I})\big\},
\]
for some $\eta_I \in C'^\inte$ and this set does not depend on the choice of $\eta_I$ by Theorem~\ref{T:polyannuli-variation}(c).
\end{itemize}
Choose integers $e_1, \dots, e_n \in \NN$ coprime to $p$ such that $e_ib_i \in \ZZ$ for any $i$ and for any $(-b_1, \dots, -b_n, \vartheta) \in B$. Put $R' = K \llbracket C' \rrbracket_0[t_1^{1/e_1}, \dots, t_n^{1/e_n}]$.
\end{situation}

\begin{theorem}\label{T:polyannuli-boundary}
The two multisets $A$ and $B$ are the same (for any $C'$ that satisfies the condition in (2)).  Moreover, there exists a finite tamely ramified extension $K'/K$ and a unique direct sum  decomposition $M \otimes R' \otimes K' = \bigoplus_{(b_I, \vartheta) \in B} M_{(b_I, \vartheta)}$ such that, if we put $F'_{\ee^{-r_I}} = F_{\ee^{-r_I}} [t_1^{1/e_1}, \dots, t_n^{1/e_n}] \otimes K'$,
\begin{itemize}
\item [\emph{(i)}] for all $r_I \in C'^\inte$, $M_{(b_I, \vartheta)} \otimes F'_{\ee^{-r_I}}$ has pure intrinsic radii $\omega \ee^{-b_1r_1- \cdots -b_nr_n + \goths}$ and pure refined intrinsic radii $t_I^{-b_I}\vartheta$, and
\item [\emph{(ii)}] any element in $\calI\Theta(M \otimes F'_{\underline 1})$ is congruent to $t_I^{-b_I}\vartheta$ modulo elements with $\bbv_s$-valuation strictly bigger than $(-b_1, \dots, -b_n)$.
\end{itemize}
\end{theorem}
\begin{proof}
We first construct the decomposition that satisfies condition (i).
For this, we may replace $K$ by a finite tamely ramified extension such that all $\vartheta$ appearing in $B$ lie in $\oplus_{i \in I} \kappa_K^{(s)} \frac{dt_i}{t_i} \oplus \oplus_{j\in J} \kappa_K^{(s)} \frac{du_j}{u_j}$ for an appropriate $s$.  In this case, we construct the decomposition of $M \otimes R'$ using the same argument as in \cite[Theorem~3.4.4]{kedlaya-xiao} by invoking Theorems~\ref{T:touch-decomposition} and \ref{T:refined-boudary} at appropriate places.

Now we check condition (ii) for this direct sum decomposition; this is equivalent to identifying the multisets A with B for each $M_{b_I, \vartheta}$.  Note that we already know that $M_{b_I, \vartheta} \otimes F_{e^{-r_I}}$ has pure intrinsic radii $\omega e^{-b_1r_1 -\cdots -b_nr_n + s}$.  For simplicity, we put $M=M_{b_I, \vartheta}$.
We do induction on the dimension $n$.  When $n = 0$ there is nothing to prove.  We assume that the theorem is proved for $n-1$.
Let $D$ denote the face $t_1 = 0$ of $C$.
Put $\widetilde C = C \cap D$, $\widetilde C'= C' \cap D$, $\widetilde S = \ee^{-\widetilde C}$, and $\widetilde R = \widetilde K \llbracket \widetilde S \rrbracket_0$ with coordinates $t_2, \dots, t_n$, where $\widetilde K$ is the completion of $\Frac (K\llbracket t_1\rrbracket_0)$ with respect to the 1-Gauss norm.

By applying the induction hypothesis to $\widetilde M = M \otimes_R \widetilde R$, the multiset $A$ is equal to
\[
A' = \big\{(\mathsf v_{s}(\vartheta'), -b_2, \dots, -b_n, \overline{t_1^{-\mathsf v_s(\vartheta')} \vartheta'} ) \big| (-b_2, \dots, -b_n) \in \QQ^{n-1}, t_2^{-b_2} \cdots t_n^{-b_n} \vartheta' \in \calI\Theta(M \otimes F_{\eta_I})\big\},
\]
for any $(r_2, \dots, r_n) \in \widetilde C'$, where $\mathsf v_s$ is the valuation on $\bigoplus_{i \in I}\kappa_{\widetilde K^\alg}^{(s)} \frac{dt_i}{t_i} \oplus\bigoplus_{j \in J}\kappa_{\widetilde K^\alg}^{(s)} \frac{du_j}{u_j}$ as in Notation~\ref{N:mathsf-v}, and $\overline{t_1^{-\mathsf v_s(\vartheta')} \vartheta'}$ is the reduction of $t_1^{-\mathsf v_s(\vartheta')} \vartheta'$ in 
$\bigoplus_{i \in I}\kappa_{K^\alg}^{(s)} \frac{dt_i}{t_i} \oplus\bigoplus_{j \in J}\kappa_{K^\alg}^{(s)} \frac{du_j}{u_j}$.

It suffices to identify the multiset $A'$ with $B$.  When $r_I \in \QQ^n \cap C'$, this follows from applying Corollary~\ref{C:refined-boundary} to the line which passes through the point $r_I$ and is parallel to the $t_1$-axis.  In particular, this says that for any $\vartheta'$ above, $\overline{t_1^{-\mathsf v_s(\vartheta')} \vartheta'}$ is the same as $\vartheta$.  When $r_I$ is not rational, the same statement follows from the ``continuity" result in Theorem~\ref{T:polyannuli-variation}(c).
\end{proof}

\begin{remark}
One can also describe the intrinsic radii of $M_{b_I, \vartheta}$ at the point $(r_I) \in C'$ with $r_1 = \cdots = r_l = 0$ for some $l \in \{1, \dots, d-1\}$.   We leave this as an exercise for interested readers.
\end{remark}

Next we consider the situation for solvable differential modules.

\begin{definition}
Let $C = \big\{(x_I) \subset \RR^n\big|x_I \geq 0, x_1+ \cdots + x_n = 1\big\}$.  For $[\alpha, \beta] \in (0,1)$, we put $S_{[\alpha, \beta]}= \{\rho^C| \rho \in [\alpha, \beta]\}$ and $R_{[\alpha, \beta]} = K \llbracket S_{[\alpha, \beta]} \rrbracket_0$.  For $\alpha \in (0, 1)$, we put $R_\alpha = \cap_{\beta \in (\alpha, 1)} R_{[\alpha, \beta]}$.

Fix $\alpha \in (0, 1)$.  Let $M$ be a differential module over $R_\alpha$.  Assume that $M$ is \emph{solvable}, that is, for each $x_I \in C$, we have $f_1(M, \rho^{x_I}) \rar 0$ as $\rho \rar 1^-$. 

By Theorem~\ref{T:swan1}, for $x_I \in C$, there exists $b_1(M, x_I), \dots, b_d(M, x_I)$ such that $f_l(M, -x_I\log \rho) = \rho^{b_l(M, x_I)}$ when $\rho \rar 1^-$, for $l = 1, \dots, d$.  Put $B_l(M, x_I) = b_1(M, x_I) +\cdots + b_l(M, x_I)$ for $l = 1, \dots, d$.
\end{definition}

\begin{proposition}
Keep the notation as above.
Then the functions $d! B_l(M, x_I)$ and $B_d(M, x_I)$ are integral polyhedral functions.
\end{proposition}
\begin{proof}
See \cite[Theorem~3.3.3]{kedlaya-swan2}.  It also follows from Theorem~\ref{T:polyannuli-variation}(a).
\end{proof}

\begin{construction}
Keep the notation as above.

Let $\underline x = (0, \dots, 1) \in C$ be the point.  Let $\gothF$ be the completion of the fraction field of $\calO_K(( t_1 )) \cdots (( t_{n-1} ))$; it is a higher dimensional local field.  We have a natural embedding $R_\alpha \inj \gothF\{\{ \eta / t_n, t_n \}\} = \widetilde \gothF_\eta$, if $\eta \in (\alpha, 1)$.  This means to restrict the picture to the line $(0, \dots, 0, \rho)$ for $\rho \in (\eta, 1)$.  We assume that $M \otimes \widetilde \gothF_\eta$ has pure-log break $b$.

Recall that, as in Situation~\ref{Sit:polyannuli-boundary}, we have a valuation $\bbv: \oplus_{i\in I} \kappa_{\gothF^\alg}\frac{dt_i}{t_i} \oplus \oplus_{ j \in J} \kappa_{\gothF^\alg} \frac{du_j}{u_j} \rar \QQ^n$.
\end{construction}

\begin{proposition}\label{P:polyannuli-Swan-conductor}
Keep the notation as above.
The following two multisets of $(n-1)$-tuples are the same.
\begin{itemize}
\item[\emph{(i)}] The multiset composed of valuations $\bbv$ of the elements of $\frac 1\bbpi\calI\Theta(M \otimes \widetilde \gothF_\eta)$, where $\bbpi$ is a Dwork pi.
\item[\emph{(ii)}] The multiset of slopes of $b_l(M, x_I)$, for $l = 1, \dots, d$, on a RP subset of $C$ which is adjacent to the cells $\big\{ t_1 = \cdots = t_i = 0, t_{i+1} + \cdots +t_n = 1\big\}$ for all $i = 1, \dots, n$.
\end{itemize}
\end{proposition}
\begin{proof}
It follows from Theorem~\ref{T:polyannuli-boundary}.
\end{proof}

\begin{remark}
One may interpret the above proposition geometrically, as in \cite{kedlaya-swan2}.  We will come back to this discussion in a future work.
\end{remark}

\end{document}